\documentclass[10pt]{siamltex}
\usepackage{amsmath,amsfonts,amscd,amssymb,bm,cite}
\usepackage{mathrsfs,listings,url}
\usepackage{graphics,color,ulem}
\usepackage{float}
\usepackage[titletoc,page]{appendix}
\usepackage{subfigure}
\usepackage[colorlinks=true, bookmarksopen,
pdfauthor={CST},
pdfcreator={pdftex},
pdfsubject={algorithms},
linkcolor={blue},
anchorcolor={black},
citecolor={firebrick},
filecolor={magenta},
menucolor={black},
pagecolor={red},
plainpages=false,pdfpagelabels,
urlcolor={db} ]{hyperref}
\usepackage[capitalise]{cleveref}
\usepackage{comment}
\usepackage{verbatim}
\usepackage{marginnote}
\usepackage{float}
\usepackage[makeroom]{cancel}
\usepackage[pdftex]{graphicx}
\usepackage[section]{placeins}
\usepackage{mathptmx}
\usepackage{cases}
\usepackage{subeqnarray}
\usepackage{hhline}
\usepackage[linesnumbered,ruled,vlined]{algorithm2e}
\usepackage{xcolor}
\usepackage{tikz}
\usepackage{caption} 
\usepackage{lmodern}

\usepackage{comment}
\excludecomment{confidential}

\usetikzlibrary{arrows}
\usetikzlibrary{positioning,shapes,snakes,calc,decorations,decorations.markings}

\newtheorem{remark}{Remark}
\restylefloat{table}
\allowdisplaybreaks
\definecolor{db}{rgb}{0.0470,0,0.5294}
\definecolor{dg}{rgb}{0.0,0.392,0.0}
\definecolor{firebrick}{rgb}{0.698,0.133,0.133}
\definecolor{bl}{rgb}{0.0,0.0,0.0}
\definecolor{linen}{rgb}{0.980,0.941,0.902}
\definecolor{ivory}{rgb}{1.0,1.0,0.941}
\definecolor{aliceblue}{rgb}{0.941,0.973,1.0}
\definecolor{beige}{rgb}{0.961,0.961,0.863}
\definecolor{tan}{rgb}{0.824,0.706,0.549}
\definecolor{lightsteelblue}{rgb}{0.690,0.769,0.871}
\definecolor{paleturquoise}{rgb}{0.686,0.933,0.933}
\definecolor{lightblue}{rgb}{0.678,0.847,0.902}
\definecolor{skyblue}{rgb}{0.529,0.808,0.922}
\definecolor{palegoldenrod}{rgb}{0.933,0.910,0.667}
\definecolor{lightgoldenrod}{rgb}{0.933,0.867,0.510}
\definecolor{lightyellow}{rgb}{1.0,1.0,0.878}
\definecolor{yellow}{rgb}{1.0,1.0,0.0}
\definecolor{lightyellow1}{rgb}{1.0,1.0,0.878}
\definecolor{lemonchiffon}{rgb}{1.0,0.980,0.804}
\definecolor{myyellow}{rgb}{1,1,.9}
\definecolor{darkgreen}{rgb}{0.0,0.392,0.0}
\definecolor{darkviolet}{rgb}{0.580,0.0,0.827}
\definecolor{lightsalmon}{rgb}{1.0,0.627,0.478}
\definecolor{orange}{rgb}{1.0,0.647,0.0}
\definecolor{darkblue}{rgb}{0.00,0.00,0.55}
\numberwithin{equation}{section}

\begin{document}
	
	\title{The Semi-implicit DLN Algorithm for the Navier Stokes Equations} 
	\author{
			Wenlong Pei\thanks{
			Department of Mathematics, The Ohio State University, Columbus, OH 43210,
			USA. Email: \href{mailto:pei.176@osu.edu}{pei.176@osu.edu}. } 
		    }
	\date{\emty}
	\maketitle
	
	\begin{abstract}
		Dahlquist, Liniger, and Nevanlinna design a family of one-leg, two-step methods (the DLN method)
		that is second order, $A-$ and $G-$ stable for arbitrary, non-uniform time steps. 
		Recently, the implementation of the DLN method can be simplified by the refactorization process (adding time filters on backward Euler scheme).
		Due to these fine properties, the DLN method has strong potential for the numerical simulation of time-dependent fluid models. 
        In the report, we propose a semi-implicit DLN algorithm for the Navier Stokes equations (avoiding non-linear solver at each time step) and prove the unconditional, long-term stability and second-order convergence with the moderate time step restriction. 
		Moreover, the adaptive DLN algorithms by the required error or numerical dissipation criterion are presented to balance the accuracy and computational cost. 
		Numerical tests will be given to support the main conclusions.

	\end{abstract}
	
	
	
	
	
	
	\begin{keywords}
		Refactorization, $G$-stability, second-order, time adaptivity, semi-implicit
	\end{keywords}
	
	\begin{AMS}
		65M12, 35Q30, 76D05
	\end{AMS}
	
	\section{Introduction}
	In the simulation of time-dependent fluid models, various time-stepping schemes have been constructed based on stability and consistency. 
	The backward Euler method, unconditionally stable and easily implemented, can only have first-order accuracy \cite{FM17_arXiv,JL14_IJUQ,RLZ18_JSC,SX23-NMPDE}. 
	The trapezoidal rule or two-step backward difference method (BDF2) are both second-order accurate and widely used in computational fluid dynamics
	\cite{GGS08_SIAM_JSC,GLCS80_Springer_Berlin,GLSS78_Tech,GS98_Vol2_Wiley,Lay08_SIAM,BQV16_IJNMF,CWW19_CCP,Emm04_M2AN_MMNA,Emm04_IJNSNS,LRT12_JMFM,TTW21_NMPDE}.
	However, the trapezoid rule with some unfavorable combinations of time steps leads to instability of the numerical solutions \cite{DLN83_SIAM_JNA,Ste73_Springer_Verlag}. 
	The variable-step BDF2 method only has conditional stability if the time step 
	ratio is small enough \cite{CGG90_NM,CMR93_JCAM,Emm05_JAMC,Emm09_BIT,Gri91_TS,GP84_BIT}.
 
	Dahlquist, Liniger, and Nevanlinna propose a one-parameter family of one-leg, two-step method (thus the DLN method herein) which is $G$-stable (non-linear stable) \cite{Dah76_Tech_RIT,Dah78_BIT,Dah78_AP_NYL,DLN83_SIAM_JNA} and second-order accurate under arbitrary time grids. 
	To our knowledge, the DLN method is the \textit{only} time-stepping algorithm possessing these two properties under \textit{arbitrary time step sequence}.
	Hence its essential properties of stability and consistency have been carefully studied and explored 
	in \cite{LPT21_AML,LPT22_Tech}.
	Recently the variable step DLN method has been applied to the unsteady Stoke/Darcy model and Navier Stokes equations (NSE) and performs well in specific test 
	problems \cite{LPQT21_NMPDE,QHPL21_JCAM,QCWLL23_ANM}.

	Given the initial value problem: 
	\begin{gather}
		\label{eq:IVP}
		y'(t) = f(t,y(t)), \ \ \ 0 \leq t \leq T, \ \ \ y(0) = y_{0},
	\end{gather}
	for $y: [0,T] \rightarrow \mathbb{R}^{d}$, 
	$f: [0,T] \times \mathbb{R}^{d} \rightarrow \mathbb{R}^{d}$ and $y_{0} \in \mathbb{R}^{d}$.
	The family of one-leg, two-step DLN method (with parameter $\theta \in [0,1]$), applying to \eqref{eq:IVP} is written 
	\begin{gather}
		\label{eq:1legDLN}
		\tag{DLN}
		\sum_{\ell =0}^{2}{\alpha _{\ell }}y_{n-1+\ell }
		= \widehat{k}_{n} f \Big( \sum_{\ell =0}^{2}{\beta _{\ell }^{(n)}}t_{n-1+\ell } ,
		\sum_{\ell =0}^{2}{\beta _{\ell }^{(n)}}y_{n-1+\ell} \Big)
		,
		\qquad
		n=1,\ldots,N-1.
	\end{gather}
	Here $\{ 0 = t_{0} < t_{1} < \cdots < t_{N-1} <t_{N}=T \}_{n=0}^{N}$ are the time grids on interval $[0,T]$ and $y_{n}$ is the DLN solution to $y(t_{n})$.
	The coefficients in \eqref{eq:1legDLN} are
	\begin{gather*}
		\begin{bmatrix}
			\alpha _{2} \vspace{0.2cm} \\
			\alpha _{1} \vspace{0.2cm} \\
			\alpha _{0} 
		\end{bmatrix}
		= 
		\begin{bmatrix}
			\frac{1}{2}(\theta +1) \vspace{0.2cm} \\
			-\theta \vspace{0.2cm} \\
			\frac{1}{2}(\theta -1)
		\end{bmatrix}, \ \ \ 
		\begin{bmatrix}
			\beta _{2}^{(n)}  \vspace{0.2cm} \\
			\beta _{1}^{(n)}  \vspace{0.2cm} \\
			\beta _{0}^{(n)}
		\end{bmatrix}
		= 
		\begin{bmatrix}
			\frac{1}{4}\Big(1+\frac{1-{\theta }^{2}}{(1+{%
					\varepsilon _{n}}{\theta })^{2}}+{\varepsilon _{n}}^{2}\frac{\theta (1-{%
					\theta }^{2})}{(1+{\varepsilon _{n}}{\theta })^{2}}+\theta \Big)\vspace{0.2cm%
			} \\
			\frac{1}{2}\Big(1-\frac{1-{\theta }^{2}}{(1+{\varepsilon _{n}}{%
					\theta })^{2}}\Big)\vspace{0.2cm} \\
			\frac{1}{4}\Big(1+\frac{1-{\theta }^{2}}{(1+{%
					\varepsilon _{n}}{\theta })^{2}}-{\varepsilon _{n}}^{2}\frac{\theta (1-{%
					\theta }^{2})}{(1+{\varepsilon _{n}}{\theta })^{2}}-\theta \Big)%
		\end{bmatrix}.
	\end{gather*}
	The step variability $\varepsilon _{n} = (k_n - k_{n-1})/(k_n + k_{n-1})$ is the function of two step size. $\widehat{k}_{n} = {\alpha _{2}}k_{n}-{\alpha _{0}}k_{n-1}$ is the average time step. Given 
	sequence $\{ z_{n} \}_{n=0}^{\infty}$, we denote 
	\begin{gather*}
		z_{n,\beta} := \beta _{2}^{(n)} z_{n+1} + \beta _{1}^{(n)} z_{n} + \beta _{0}^{(n)} z_{n-1}
		= \sum_{\ell =0}^{2}{\beta _{\ell }^{(n)}} z_{n-1+\ell },
	\end{gather*}
	for convenience in the remaining paragraphs. Then the DLN method in \eqref{eq:1legDLN} can be shorten 
	\begin{gather*}
		\sum_{\ell =0}^{2}{\alpha _{\ell }}y_{n-1+\ell }
		= \widehat{k}_{n} f ( t_{n,\beta} , y_{n,\beta} ).
	\end{gather*}

	Herein we propose the variable-step, semi-implicit DLN algorithm for NSE and present a detailed numerical analysis of stability and convergence. 
	Let the open, connected and bounded set $\Omega \subset \mathbb{R}^{d}$($d=2 \text{ or } 3$) be the domain, the fluid velocity $u(x,t)$, pressure $p(x,t)$ and the source $f(x,t)$ in the NSE are governed by the following system
	\begin{gather}
		u_{t} + u \cdot \nabla u - \nu \Delta u + \nabla p = f, \ \ x \in \Omega, \ \ 0 < t \leq T, \notag \\
		\nabla \cdot u = 0, \ \ x \in \Omega, \ \ 0 < t \leq T, \ \ \ u(x,0) = u_{0}(x), \ \ 0 < t \leq T, 
		\label{eq:NSE} \\
		u = 0 \ \ \text{on } \  \partial{\Omega}, \ \ \ \int_{\Omega} p dx = 0, \ \ 0 < t \leq T. \notag
	\end{gather}
	Let $u_{n}^{h}$ and $p_{n}^{h}$ are the numerical solutions to velocity $u(x,t_{n})$ 
	and pressure $p(x,t_{n})$ respectively on certain finite element space with diameter $h$, the fully-implicit DLN algorithm for NSE approximate the non-linear term $u \cdot \nabla u $ at $t_{n+1}$ by 
	$u_{n,\beta}^{h} \cdot \nabla u_{n,\beta}^{h}$ \cite{LPQT21_NMPDE}, which results in two main disadvantages of the algorithm:
	\begin{itemize}
		\item[i.] the rigorous time step restriction like $\Delta t \leq \mathcal{O}(\nu^{3})$ for convergence,
		\item[ii.] the non-linear solver in each time step computation.
	\end{itemize}
	The above time step restriction in error analysis arises from the use of the discrete Gr$\ddot{\rm{o}}$nwall inequality \cite[p.369]{HR90_SIAM_NA} and would be very strict even under moderate viscosity 
	value (like $\nu = 1.\rm{e}-2$)\footnote[1]{To our knowledge, the restriction can not be avoided as long as the fully-implicit, time-stepping methods are applied.}. 
	Fixed point iteration and Newton's iteration are two common choices for non-linear solvers. 
	Fixed point iteration is easily implemented while Newton's iteration possesses fast convergence. 
	However, they usually cost more than solving a linear system and have the risk of divergence if the initial value for the iteration is poorly guessed.
	
	To address the two issues, we extend Baker's idea \cite{Bak76_Tech} and propose the semi-implicit DLN scheme for NSE. The essence of the idea is to extrapolate the first $u_{n,\beta}^{h}$ in the non-linear 
	term $u_{n,\beta}^{h} \cdot \nabla u_{n,\beta}^{h}$ by its second-order 
	extrapolation in time (the linear combination of $u_{n}^{h}$ and $u_{n-1}^{h}$).
	In return, the non-linear solver at each time step is replaced by a linear system. 
	Meanwhile, the strict time step restriction for convergence can be released. 
	In addition, we simplify the DLN implementation by the refactorization process (pre- and post-process on backward Euler scheme) and improve the computational efficiency by the corresponding time adaptivity algorithms.

	The paper is organized as follows. 
	Necessary notations and preliminaries are given in Section \ref{sec:Prelim}.
	In section \ref{sec:DLN-Alg}, we propose the semi-implicit DLN algorithm for NSE and its equivalent implementation by the refactorizaion process. 
	In section \ref{sec:Sta-Ana}, we will show that the numerical solution is long-term, unconditionally stable. The variable step error analysis with the loose time step restriction is given in Section \ref{sec:Err-Ana}. 
	The time adaptive algorithms (using error or numerical dissipation criterion)
	in Section \ref{sec:Implement-DLN} are provided to reduce the computational cost.
	In Section \ref{sec:Num-Test}, the Taylor-Green benchmark problem \cite{TG1937_RSL} is to confirm 
	the second-order convergence.
	The unconditional stability of the variable step DLN and the advantage of time adaptivity are verified in 
	the revised Taylor-Green problem and the 2D offset problem \cite{JL14_IJUQ}.

	\subsection{Related Work}
	Semi-implicit schemes are effective ways to simulate non-linear, time-dependent fluid models. Baker studied the semi-implicit Crank-Nicolson method and applied it to NSE \cite{Bak76_Tech} early. 
	Baker, Dougalis, and Karakashian increase the accuracy of numerical solutions by use of the three-step backward difference method coupled with extrapolation for non-linear terms \cite{BDK82_Math_Comp}. 
	Girault and Raviart prove the convergence of general linear two-step semi-implicit algorithms for NSE \cite{GR79_Springer}. 
	They solve the two-dimensional NSE by a fully discrete two-level finite element method: Crank-Nicolson extrapolation scheme on spatial-time coarse grids and backward Euler scheme on fine grids \cite{He03_SIAM_NA}. 
	Labovsky, Layton, Manica, Neda, and Rebholz add artificial viscosity stabilization to extrapolated trapezoidal finite-element method for NSE, leading to a more easily solvable linear system at each time step \cite{LLMNR09_CMAME}. 
	Ingram offers a detailed numerical analysis of the semi-implicit Crank-Nicolson scheme for NSE and proves the convergence of both velocity and pressure \cite{Ing13_IJNAM}.

	Time adaptivity based on certain criteria (required local truncation error, minimum numerical dissipation, etc.) is an optimal choice to treat the conflicts between time accuracy and efficiency. 
	Inspired by the pioneering work of Gear \cite{Gea71_PrenticeHall}, Kay, Gresho, Griffiths, and Silvester 
	implement the trapezoidal scheme for NSE in a memory efficient way and estimate the error in time by 
	the explicit two-step Adams Bashforth method (AB2 method) in the time adaptivity 
	\cite{GLCS80_Springer_Berlin,GS98_Vol2_Wiley,GGS08_SIAM_JSC,KGGS10_SIAM_JSC}.
	John and Rang propose diagonal-implicit RK methods (DIRK methods) with an embedding lower-order scheme for adaptivity and implement this algorithm for 2D laminar flow around a cylinder \cite{JR10_CMAME}. 
	Guzel and Layton add the time filter to the backward Euler method to increase the accuracy of numerical 
	solutions \cite{GL18_BIT}.
	The corresponding BE-filter adaptivity is applied to various fluid models \cite{CEK20_JSC,DLZ19_arXiv,LHLZ20_CMAME}. 
	Recently Buka\v{c}, Burkardt, Seboldt, and Trenchea refactorize the midpoint rule and adjust time steps by the revised AB2 method in fluid-structure interaction 
	problems \cite{BT20_AML,BST21_JMFM,BT21_CMAME,BPT22_IJNAM}. 
	The numerical dissipation criterion for adaptivity, proposed by Capuano, Sanderse, De Angelis, and Coppola in \cite{CSDC17}, has been tested in the DLN simulations of the NSE and the coupled Stokes-Darcy model \cite{LPQT21_NMPDE,QCWLL23_ANM}.

	\section{Preliminaries}
	\label{sec:Prelim}
	Let $\Omega \subset \mathbb{R}^{d}$ ($d = 2 \text{ or } 3$) be the domain. For $1 \leq p < \infty$,
	$L^{p}(\Omega)$ is the normed linear space containing Lebesgue measurable function $f$ such that 
	$|f|^{p}$ is integragle. For $r \in \mathbb{N}$, the Sobolev space $W^{r,p}(\Omega)$ with usual norm $\| \cdot \|_{W^{r,p}}$ contains all Lebesgue measurable functions whose weak derivatives up to $r$-th belong to $L^{p}(\Omega)$ and $W^{0,p}(\Omega)$ is the usually $L^{p}(\Omega)$ space. 
	$\| \cdot \|$ and $(\cdot, \cdot)$ denote the $L^{2}$-norm and $L^{2}$-inner product. 
	The space $H^{r}$ with norm $\| \cdot \|_{r}$ and semi-norm $| \cdot |_{r}$ denotes the Sobolev space with $p=2$. 
	The velocity space $X$ and the pressue space $Q$ for the NSE in \eqref{eq:NSE} are
	\begin{gather*}
		X := \big(H_{0}^{1}(\Omega) \big)^{d} = \big\{ v \in \big(H^{1}(\Omega) \big)^{d}: v = 0 \text{ on } \partial \Omega \big\},
		\ \ \ 
		Q := \{ q \in L^{2}(\Omega): (q,1) = 0  \}.
	\end{gather*}
	The divergence-free space for the velocity is 
	\begin{gather*}
		V := \big\{ v \in X: \nabla \cdot v = 0 \big\}.
	\end{gather*}
	For any function $v \in X \cap \big( H^{r} \big)^{d}$, the norm $\| \cdot \|_{r}$ 
	and semi-norm $| \cdot |_{r}$ are define
	\begin{gather*}
		\| v \|_{r} := \Big( \sum_{\ell=1}^{d} \| v_{i} \|_{r}^{2} \Big)^{1/2}, \ \ \ 
		| v |_{r} := \Big( \sum_{\ell=1}^{d} | v_{i} |_{r}^{2} \Big)^{1/2}.
	\end{gather*}
	$X'$ is the dual space of $X$ with the dual norm
	\begin{gather}
		\label{eq:dual-norm}
		\| f \|_{-1} := \sup_{v \in X, v \neq 0} \frac{(f,v)}{\| \nabla v \|}, \ \ \ \forall f \in X'.
	\end{gather}
	We need the Bochner space on the time interval $[0,T]$
	\begin{align*}
		L^{p} \!\big(0,T;\big(H^r \big)^{d} \big)
		&= \Big\{ f(\cdot,t) \in \big(H^r \big)^{d}:  
		\| f \|_{p,r} = \Big( \int_{0}^{T} \| f(\cdot,t)  \|_{r}^{p} dt \Big)^{1/p} < \infty \Big\}, \\
		L^{\infty} \big( 0,T;\big(H^r \big)^{d} \big) 
		&= \Big\{ f(\cdot,t) \in \big(H^r \big)^{d}: 
		\|f\|_{\infty,r} = \sup_{0 < t <T} \|f(\cdot, t) \|_{r} < \infty \Big\}, \\
		L^{p}\! \big(0,T;X'\big) &= \Big\{ f(\cdot,t) \in X' :  
		\| f \|_{p,-1} = \Big( \int_{0}^{T} \| f(\cdot,t)  \|_{-1}^{p} dt \Big)^{1/p} < \infty \Big\}, 
	\end{align*}
	and the discrete Bochner space with the time grids $\{t_{n} \}_{n=0}^{N}$ on the 
	time interval $[0,T]$
	\begin{align*}
		\ell^{\infty} \big( 0,N;\big(H^r \big)^{d}  \big)
		&= \big\{ f(\cdot,t) \in \big(H^r \big)^{d}: 
		\| |f| \|_{\infty,r} < \infty \big\},
		\\
		\ell^{\infty,\beta} \big( 0,N;\big(H^r \big)^{d} \big)
		&= \big\{ f(\cdot,t) \in \big(H^r \big)^{d}: 
		\| |f| \|_{\infty,r,\beta} < \infty \big\}, \\
		\ell^{p,\beta} \big( 0,N;\big(H^r \big)^{d} \big) 
		&= \big\{ f(\cdot,t) \in \big(H^r \big)^{d}: \| |f| \|_{p,r,\beta} < \infty \big\}, \\
		\ell^{p,\beta} \big( 0,N; X' \big) 
		&= \big\{ f(\cdot,t) \in X': \| |f| \|_{p,-1,\beta} < \infty \big\},
	\end{align*}
	where the corresponding discrete norms are 
	\begin{gather}
		\| |f| \|_{\infty,r} := \max_{0 \leq n \leq N} \| f(\cdot, t_{n}) \|_{r}, \ \ \ 
		\| |f| \|_{\infty,r,\beta} := \max_{1 \leq n \leq N-1} \| f(\cdot, t_{n,\beta}) \|_{r} \notag \\
		\| |f| \|_{p,r,\beta} 
		:= \Big( \sum_{n=1}^{N-1} (k_{n} + k_{n-1})\| f(\cdot,t_{n,\beta}) \|_{r}^{p} \Big)^{1/p}. \notag \\
		\| |f| \|_{p,-1,\beta} 
		:= \Big( \sum_{n=1}^{N-1} (k_{n} + k_{n-1})\| f(\cdot,t_{n,\beta}) \|_{-1}^{p} \Big)^{1/p}
		\label{eq:def-norm-dis}
	\end{gather}
	The discrete norm $\| | \cdot | \|_{p,r,\beta}$ in \eqref{eq:def-norm-dis} is the form of Riemann sum in which the function $f$ is evaluated at $t_{n,\beta} \in [t_{n-1},t_{n+1}]$.
	For any $u,v,w \in \big(H^{1}(\Omega) \big)^{d}$, we define the skew-symmetric, non-linear operator 
	\begin{gather}
		\label{eq:non-linear-b}
		b(u,v,w) := \frac{1}{2} (u \cdot \nabla v, w) - \frac{1}{2} (u \cdot \nabla w, v).
	\end{gather}
	\begin{confidential}
		\color{darkblue}
		\begin{gather*}
			b(u,v,w) = (u \cdot \nabla v, w) + \frac{1}{2} \big( (\nabla \cdot u) v, w \big)
			- \frac{1}{2} \int_{\partial{\Omega}} (u \cdot \overrightarrow{n} ) v \cdot w dS,
		\end{gather*}
		where $\overrightarrow{n}$ is the outward unit normal vector of $\partial{\Omega}$.
		\normalcolor
	\end{confidential}
	We apply divergence theorem and integration to \eqref{eq:non-linear-b}
	\begin{gather*}
		b(u,v,w) = (u \cdot \nabla v, w) + \frac{1}{2} \big( (\nabla \cdot u) v, w \big),
		\ \ \ \text{for any } u,v,w \in X.
	\end{gather*}
	Thus if $u \in V$, $b(u,v,w) = (u \cdot \nabla v, w)$.
	We have the following lemma about the bounds of the operator $b$.
	\begin{lemma}
		For any $u,v,w \in H^{1}$,
		\begin{gather}
			|(u\cdot \nabla v, w)| \leq C
			\left\{
			\begin{aligned}
				& \| u \|_{1} |v|_{1} \| w \|_{1}   \\
				& \| u \| \| v \|_{2} \| w \|_{1} & & \forall v \in H^{2} \\
				& \| u \|_{2} |v|_{1} \| w \| & & \forall u \in H^{2} 
			\end{aligned}
			\right. , \label{eq:b-bound1}     \\
			|b(u,v,w)| \leq C \big( \| u \| \| u \|_{1} \big)^{1/2} \| v \|_{1} \| w \|_{1},  
			\label{eq:b-bound2}	
		\end{gather}
		Moreover, if $u,v,w \in X$,
		\begin{gather}
			\label{eq:b-bound3}
        	|b(u,v,w)| \leq C
        	\left\{
        	\begin{aligned}
        		&  \| u \|_{1} \| v \|_{1} \big( \| w \| \| w \|_{1} \big)^{1/2}   \\
        		& \| u \|_{1} \| v \|_{2} \| w \| & & \forall v \in H^{2} 
        	\end{aligned}
			\right. .
        \end{gather}
	\end{lemma}
	\begin{proof}
		See \cite[p.273-275]{Ing13_IJNAM}.
	\end{proof}

	For spatial discretization, $X^{h} \subset X$ and $Q^{h} \subset Q$ are certain finite element spaces for velocity and pressure respectively based on the edge-to-edge triangulation of domain $\Omega$ (with the maximum diameter of the triangles $h>0$).
	$X_{h}$ and $Q_{h}$ satisfy the discrete inf-sup condition, i.e.
	\begin{gather}
		\label{eq:inf-sup}
		\inf_{q^{h} \in Q^{h}} \sup_{v^{h} \in X^{h}} \frac{(\nabla \cdot v^{h}, q^{h})}{\| \nabla v^{h}\| \| q^{h} \|} \geq C, 
	\end{gather}
	for some positive constant $C$. Typical examples of such finite element spaces are Taylor-Hood element spaces and Mini element spaces.
	The approximation theorem for $X^{h}$ and $Q^{h}$ is 
	\begin{align}
		\label{eq:approx-thm}
		& \inf_{v^{h} \in X^{h}} \| u - v^{h} \|_{1} \leq Ch^{r} \| u \|_{r+1}, \ \ \ 
		r \geq 0, \ u \in (H^{r+1})^{d} \cap X,   \notag \\
		& \inf_{q^{h} \in Q^{h}} \| p - q^{h} \| \leq Ch^{s+1} \| p \|_{s+1}, \ \ \ 
		s \geq 0, \ p \in H^{s+1} \cap Q,   
	\end{align}
	where $r$ and $s$ are the polynomial degrees of $X^{h}$ and $Q^{h}$ respectively. 
	The inverse inequality for $X^{h}$ is
	\begin{gather}
		\label{eq:inverse-estimator}
		|v^{h}|_{1} \leq Ch^{-1} \|v^{h} \|, \ \ \ \forall v^{h} \in X^{h}.
	\end{gather}
	The discrete divergence-free space is 
	\begin{gather*}
		V^{h} := \{ v^{h} \in X^{h}: (q^{h}, \nabla \cdot v^{h}) = 0, \ \ \ \forall q^{h} \in Q^{h} \}.
	\end{gather*}
	For any pair $(u,p) \in V \times Q$, the Stokes projection 
	$(P_{S}^{(u)} u, P_{S}^{(p)}p) \in V^{h} \times Q^{h}$ is defined as the solution to the problem 
	\begin{align*}
		\nu( \nabla u, \nabla v^{h}) - (p, \nabla \cdot v^{h}) 
		&= \nu( \nabla P_{S}^{(u)}u , \nabla v^{h}) - (P_{S}^{(p)}p, \nabla \cdot v^{h}), \\
		(q^{h}, \nabla \cdot P_{S}^{(u)}u) &= 0, \qquad \qquad \qquad \qquad \qquad \qquad \qquad
		\forall \ (v^{h},q^{h}) \in X^{h} \times Q^{h}. 
	\end{align*}
	The Stokes projection has the following approximations (see \cite{GR86_Springer,Joh16_Springer}
	for proof)
	\begin{gather}
		\| u -  P_{S}^{(u)}u \| \leq C h \big( \nu^{-1} \inf_{q^{h} \in Q^{h}} \|p - q^{h} \| 
		+ \inf_{v^{h} \in X^{h}} | u - v^{h} |_{1}\big), \notag \\
		\| u -  P_{S}^{(u)}u \|_{1} \leq C \big( \nu^{-1} \inf_{q^{h} \in Q^{h}} \|p - q^{h} \| 
		+ \inf_{v^{h} \in X^{h}} | u - v^{h} |_{1} \big).  \label{eq:Stoke-Approx} 
	\end{gather}

	\section{The algorithm}
	\label{sec:DLN-Alg}
	Let $u_{n}^{h} \in X^{h}$ and $p_{n}^{h} \in Q^{h}$ be the numerical solutions of velocity $u(x,t_{n})$ 
	and pressure $p(x,t_{n})$ respectively. 
	Then the semi-implicit DLN algorithm for the NSE in \eqref{eq:NSE} is: given two previous solutions $u_{n}^{h}, u_{n-1}^{h} \in X^{h}$, $p_{n}^{h}, p_{n-1}^{h} \in Q^{h}$, find $u_{n+1}^{h} \in X^{h}$ and $p_{n+1}^{h} \in Q^{h}$ such that for all $(v^{h},q^{h}) \in X^{h} \times Q^{h}$
	\begin{gather}
		\Big(\! \frac{\alpha_{2} u_{n\!+\!1}^{h} \!+\! \alpha_{1} u_{n}^{h} \!+\! \alpha_{0} u_{n\!-\!1}^{h}}{\widehat{k}_{n}}\!,\!v^{h} \!\Big)
		\!+\! b \big( \widetilde{u}_{n}^{h}, u_{n\!,\!\beta}^{h}\!, v^{h} \big)  
		\!+\! \nu \big( \nabla u_{n\!,\!\beta}^{h}\!,\! \nabla v^{h} \big)
		\!-\! \big( p_{n\!,\!\beta}^{h}\!,\! \nabla \!\cdot \!v^{h} \big) 
		\!=\! \big( f_{n\!,\!\beta}\!, v^{h} \big) ,   \notag \\
		\big( \nabla \cdot u_{n,\beta}^{h}, q^{h} \big) = 0,  \label{eq:DLN-Semi-Alg}
	\end{gather} 
	where the second-order, linear extrapolation $\widetilde{u}_{n}^{h}$ for $u_{n,\beta}^{h}$ is
	\begin{gather*}
		\widetilde{u}_{n}^{h} 
		= \beta _{2}^{(n)} \Big[(1 + \frac{k_{n}}{k_{n-1}}) u_{n}^{h} - \frac{k_{n}}{k_{n-1}} u_{n-1}^{h} \Big]
		+ \beta _{1}^{(n)} u_{n}^{h} + \beta _{0}^{(n)} u_{n-1}^{h} 
		\approx u_{n,\beta}^{h}.
	\end{gather*}
	The above semi-implicit DLN algorithm in \eqref{eq:DLN-Semi-Alg} can be simplified by the following refactorizaion process (See \cite{LPT21_AML} for the proof of equivalence)

	Step 1. Pre-possess: 
	\begin{gather*}
			u_{n}^{h\!,\!\rm{old}} = a_{1}^{(n)} u_{n}^{h} + a_{0}^{(n)} u_{n-1}^{h}, \ \ \ 
			\widehat{k}_{n}^{\rm{BE}} = b^{(n)} \widehat{k}_{n}, \\
			\widetilde{u}_{n}^{h} 
		= \beta _{2}^{(n)} \Big[(1 + \frac{k_{n}}{k_{n-1}}) u_{n}^{h} - \frac{k_{n}}{k_{n-1}} u_{n-1}^{h} \Big]
		+ \beta _{1}^{(n)} u_{n}^{h} + \beta _{0}^{(n)} u_{n-1}^{h},
		\end{gather*}

	Step 2. Semi-implicit backward Euler solver: solve for $u_{n+1}^{h,\rm{temp}}$ and $p_{n+1}^{h,\rm{temp}}$
	\begin{gather*}
			\Big(\! \frac{u_{n\!+\!1}^{h\!,\!\rm{temp}}\!-\! u_{n}^{h\!,\!\rm{old}}}{\widehat{k}_{n}^{\rm{BE}}}, \!v^{h} \!\Big)
			\!+\! b \big(\! \widetilde{u}_{n}^{h},\! u_{n\!+\!1}^{h\!,\!\rm{temp}}, \!v^{h} \!\big)
			\!+\! \nu \!\big(\! \nabla u_{n\!+\!1}^{h\!,\!\rm{temp}}, \!\nabla v^{h} \!\big) 
			\!-\! \big(\! p_{n\!+\!1}^{h\!,\!\rm{temp}}, \!\nabla \!\cdot \! v^{h} \!\big)
			\!=\! (\! f_{n,\!\beta},\! v^{h} \!),  \\
			( \nabla \cdot u_{n+1}^{h,\rm{temp}}, q^{h} ) = 0,
		\end{gather*}

	Step 3. Post-possess:
	\begin{gather*}
			u_{n+1}^{h} = c_{2}^{(n)} u_{n+1}^{h\!,\!\rm{temp}} + c_{1}^{(n)} u_{n}^{h} 
			+ c_{0}^{(n)} u_{n-1}^{h}, \ 
			p_{n+1}^{h} = c_{2}^{(n)} p_{n+1}^{h\!,\!\rm{temp}} + c_{1}^{(n)} p_{n}^{h} 
			+ c_{0}^{(n)} p_{n-1}^{h},
	\end{gather*}
	where the coefficents in the refactorizaion process are 
	\begin{align*}
		&		a_{1}^{(n)} = \beta_{1}^{(n)} - 
		\alpha_{1} \beta_{2}^{(n)} / \alpha_{2}, 
				\quad
				a_{0}^{(n)} = \beta_{0}^{(n)} - 
		\alpha_{0} \beta_{2}^{(n)} / \alpha_{2} , 
				\quad
				b^{(n)} = 
		\beta_{2}^{(n)} / \alpha_{2} , 
		\\	
		&		
		c_{2}^{(n)} = 
		1 / \beta_{2}^{(n)}, 
				\quad
				c_{1}^{(n)} = -
		\beta_{1}^{(n)} / \beta_{2}^{(n)} ,
				\quad
				c_{0}^{(n)} = 
		- \beta_{0}^{(n)} / \beta_{2}^{(n)} .
	\end{align*}

	\section{Numerical Analysis}
	For numerical analysis, we need the following two lemma about the stability and consistency of the DLN method.
	\begin{lemma}
		Let $Y$ be the inner product space over $\mathbb{R}$ with the inner product $(\cdot, \cdot)_{Y}$ and the induced norm $\| \cdot \|_{Y}$. For any sequence $\{ y_{n} \}_{n=0}^{N}$ in $Y$, $\theta \in [0,1]$ and $n \in \{ 1,2, \cdots, N-1 \}$
		\begin{align}
			\label{eq:G-stab}
			\Big(\sum_{\ell\!=\!0}^{2}{\alpha_{\ell }}y_{n\!-\!1\!+\!\ell },\sum_{\ell \!=\!0}^{2}{%
				\beta_{\ell }^{(n)}}y_{n\!-\!1\!+\!\ell }\Big)_{\!Y}
			\!=\!
			\begin{Vmatrix}
				{y_{n\!+\!1}} \\
				{y_{n}}
			\end{Vmatrix}
			_{G(\!\theta \!)}^{2} \!-\!
			\begin{Vmatrix}
				{y_{n}} \\
				{y_{n\!-\!1}}
			\end{Vmatrix}
			_{G(\!\theta \!)}^{2} 
			\!+\! \Big\|\!\sum_{\ell \!=\!0}^{2}{\gamma_{\ell }^{(n)}}y_{n\!-\!1\!+\!\ell} \!\Big\|_{\!Y}^{2},
		\end{align}
		where the $\| \cdot \|_{G(\theta)}$-norm is
		\begin{align}  \label{eq:G-norm}
			\begin{Vmatrix}
				u \\
				v\end{Vmatrix}_{G(\theta)}^{2}
			=& \frac{1}{4} (1+{\theta})\| u \|_{Y}^{2}
			+ \frac{1}{4} (1 - \theta ) \| v \|_{Y}^{2}
			\qquad
			\forall u,v\in Y.
		\end{align}
		and the coefficents $\gamma_{\ell }^{(n)}\ (\ell =0,1,2)$ are
		\begin{align}
			\label{eq:G-coeff}
			\gamma_{1}^{(n)}=-\frac{\sqrt{\theta \left( 1-{\theta }^{2}\right) }}{\sqrt{2}%
				(1+\varepsilon _{n}\theta )},\quad \gamma_{2}^{(n)}=-\frac{1-\varepsilon _{n}}{2}%
			\gamma_{1}^{(n)},\quad \gamma_{0}^{(n)}=-\frac{1+\varepsilon _{n}}{2}\gamma_{1}^{(n)},
		\end{align}
		By the above identity in \eqref{eq:G-stab}, the whole family of variable step,
		one-leg \eqref{eq:1legDLN} methods are G-stable (see \cite[p.2]{Dah76_Tech_RIT} for the definition).
	\end{lemma}
	\begin{proof}
		The proof of identity in \eqref{eq:G-stab} (implicit in \cite{DLN83_SIAM_JNA}) is an algebraic calculation. 
	\end{proof}
	\begin{lemma}
		\label{lemma:DLN-consistency}
		Given Banach space $Y$ with the norm $\| \cdot \|_{Y}$, time grids $\{ t_{n} \}_{n=0}^{N}$ on time interval $[0,T]$ and the mapping $u: [0,T] \rightarrow Y$, $u_{n}$ denotes $u(t_{n})$ and 
		$\widetilde{u}_{n}$ represents second-order, linear extrapolation of $u_{n,\beta}$ in time, i.e.
		\begin{gather}
			\label{eq:2nd-approx}
			\widetilde{u}_{n} 
			= \beta _{2}^{(n)} \Big[(1 + \frac{k_{n}}{k_{n-1}}) u_{n} - \frac{k_{n}}{k_{n-1}} u_{n-1} \Big] + \beta _{1}^{(n)} u_{n} + \beta _{0}^{(n)} u_{n-1}.
		\end{gather}
		If the mapping $u(t)$ is smooth enough about variable $t$, then for any $\theta \in [0,1]$
		\begin{gather}
			\big\| u_{n,\beta} - u ( t_{n,\beta} ) \big\|_{Y}^{2}
			\leq
			C(\theta) ( k_{n} + k_{n-1} )^{3}  \int_{t_{n-1}}^{t_{n+1}} \| u_{tt} \|_{Y}^{2} dt, \notag \\
			\big\| \widetilde{u}_{n} - u ( t_{n,\beta} ) \big\|_{Y}^{2}
			\leq
			C(\theta) ( k_{n} + k_{n-1} )^{3}  \int_{t_{n-1}}^{t_{n+1}} \| u_{tt} \|_{Y}^{2} dt,
			\label{eq:DLN-Consistency1}
		\end{gather}
		Moreover, if there exists constants $C_{L},C_{U}>0$ such that the ratio of time steps satisfies
		\begin{gather*}
			0 < C_{L} \leq \frac{k_{n}}{k_{n-1}} \leq C_{U},		
		\end{gather*}
		then 
		\begin{gather}
			\label{eq:DLN-Consistency2}
			\Big\| \frac{1}{\widehat{k}_{n}} \sum_{\ell=0}^{2} \alpha_{\ell} u_{n-1+\ell} 
			- u_{t} ( t_{n,\beta} ) \Big\|_{Y}^{2}
			\leq
			C ( \theta ) ( k_{n} + k_{n-1} )^{3} \int_{t_{n-1}}^{t_{n+1}} \| u_{ttt} \|_{Y}^{2} dt.
		\end{gather}
	\end{lemma}
	\begin{proof}
		Using Taylor theorem and expanding $u_{n+1}$, $u_{n}$ and $u_{n-1}$ at time $t_{n,\beta}$.
	\end{proof}

	\subsection{Stability Analysis}
	\label{sec:Sta-Ana} \ 
	\begin{theorem}
		\label{thm:Semi-DLN-Stab}
		If the body force $f$ in NSE satisfies $f \in L^{2}(0,T;X') \cap \ell^{2,\beta}(0,N;X')$,
		the semi-implicit DLN algorithm for NSE in \eqref{eq:DLN-Semi-Alg} satisfies the following unconditional, long-time energy bounds: for any integer $N>1$,
		\begin{gather}
			\frac{1}{4} (1+\theta) \| u_{N}^{h} \|^{2} + \frac{1}{4}(1-\theta) \| u_{N-1}^{h} \|^{2}
			+\sum_{n=1}^{N-1} \Big\|\sum_{\ell =0}^{2} \gamma_{\ell}^{(n)} u_{n-1+\ell }^{h} \Big\|^{2}
			+\frac{\nu}{2}\sum_{n=1}^{N-1} \widehat{k}_{n} \| \nabla {u_{n,\beta}^{h}} \|^{2}
			\notag
			\\
			\leq \frac{C(\theta)}{\nu} \big( \| |f| \|_{2,-1,\beta}^{2} + \| f \|_{2,-1}^{2} \big)
			+ \frac{1}{4}(1+\theta) \| u_{1}^{h} \|^{2}
			+ \frac{1}{4}(1-\theta) \| u_{0}^{h} \|^{2},
			\label{eq:Semi-DLN-Stab}
		\end{gather}
		where $\{\gamma_{i}^{(n)} \}_{i=0,1,2}$ are defined in \eqref{eq:G-coeff}.
	\end{theorem}
	\begin{proof}
		We set $v^{h}=u_{n,\beta}^{h}$ in \eqref{eq:DLN-Semi-Alg} and use the skew-symmetry property of the operator $b$
		\begin{align*}
			\frac{1}{\widehat{k}_{n}}
			\Big( \sum_{\ell=0}^{2} \alpha_{\ell} u_{n-1+\ell}^{h} , u_{n,\beta}^{h} \Big)
			+ \nu \| \nabla u_{n,\beta}^{h} \|^{2}
			= ( f_{n,\beta}  , u_{n,\beta}^{h} ) .
		\end{align*}
		We apply \eqref{eq:dual-norm} and Young's inequality to $( f_{n,\beta}  , u_{n,\beta}^{h} )$
		\begin{confidential}
			\color{darkblue}
			\begin{align*}
				\big(  f_{n,\beta} ,\nabla u_{n,\beta}^{h} \big)
				\leq
				\| f_{n,\beta} \|_{-1} \| \nabla u_{n,\beta}^{h} \|
				& \leq
				\frac{\nu}{2} \| \nabla u_{n,\beta}^{h} \|^{2} + \frac{1}{2 \nu } 
				\| f_{n,\beta} \|_{-1}^{2}  \\
				& \leq 
				\frac{\nu}{2} \| \nabla u_{n,\beta}^{h} \|^{2} + \frac{1}{\nu } 
				\| f(t_{n,\beta}) \|_{-1}^{2}  
				+ \frac{1}{\nu } \| f_{n,\beta} - f(t_{n,\beta}) \|_{-1}^{2}.
			\end{align*}
			\normalcolor
		\end{confidential}
		\begin{align*}
			\Big(\! \sum_{\ell\!=\!0}^{2}{\alpha_{\ell}}{u_{n\!-\!1\!+\!\ell}^{h}}\ ,u_{n,\beta}^{h} \!\Big)
			\!+\! \frac{\nu}{2} \widehat{k}_n \| \nabla u_{n,\beta}^{h} \|^{2}
			\!\leq\!
			\frac{\widehat{k}_n}{\nu} \| f(t_{n,\beta}) \|_{-1}^{2}
			\!+\! \frac{\widehat{k}_n}{\nu} \| f_{n,\beta} \!-\! f(t_{n,\beta}) \|_{-1}^{2} .
		\end{align*}
		\begin{confidential}
			\color{darkblue}
			\begin{gather*}
				\| f_{n,\beta} - f(t_{n,\beta}) \|_{-1}^{2}
				\leq C(\theta) ( k_{n} + k_{n-1} )^{3} \int_{t_{n-1}}^{t_{n+1}} \| f_{tt} \|_{-1}^{2} dt
			\end{gather*}
			\normalcolor
		\end{confidential}
		Then we use the $G$-stability identity in \eqref{eq:G-stab} and Lemma \ref{lemma:DLN-consistency}
		\begin{gather}
			\label{eq:Stab-eq1}
			\begin{Vmatrix}
				{u_{n+1}^{h}} \\
				{u_{n}^{h}}
			\end{Vmatrix}
			_{G(\theta)}^{2} -
			\begin{Vmatrix}
				{u_{n}^{h}} \\
				{u_{n-1}^{h}}
			\end{Vmatrix}
			_{G(\theta)}^{2} + \Big\| \sum_{\ell=0}^{2} \gamma_{\ell}^{(n)} u_{n-1+\ell}^{h} \Big\|^{2}
			+ \frac{\nu}{2} \widehat{k}_{n} \| \nabla u_{n,\beta}^{h} \|^2 \\
			\leq
			\frac{C(\theta)}{\nu} ( k_{n} + k_{n-1} )
			\| f ( t_{n,\beta}) \|_{-1}^{2} 
			+ \frac{C(\theta)}{\nu} ( k_{n} + k_{n-1} )^{4} \int_{t_{n-1}}^{t_{n+1}} \| f_{tt} \|_{-1}^{2} dt.
			\notag 
		\end{gather}
		By the definition of $G(\theta)$-norm in \eqref{eq:G-norm} and the notations 
		in \eqref{eq:def-norm-dis},
		we sum \eqref{eq:Stab-eq1} over $n$ from $1$ to $N-1$ to yield \eqref{eq:Semi-DLN-Stab}.
	\end{proof}
	\begin{remark}
		We define the numerical dissipation and viscosity dissipation of the semi-implicit DLN algorithm in \eqref{eq:DLN-Semi-Alg} at time $t_{n+1}$
		\begin{align*}
			\text{Numerical dissipation } \mathcal{E}_{n+1}^{\tt ND} &:= \frac{1}{\widehat{k}_{n}} 
			\Big\| \sum_{\ell=0}^{2} \gamma_{\ell}^{(n)} u_{n-1+\ell}^{h} \Big\|^{2}, \\
			\text{Viscosity dissipation } \mathcal{E}_{n+1}^{\tt VD} &:= \nu \| \nabla u_{n,\beta}^{h} \|^{2}.
	    \end{align*}
	\end{remark}

	\subsection{Error Analysis}
	\label{sec:Err-Ana}
	Let $r$ and $s$ be the polynomial degree of $X^{h}$ and $Q^{h}$ respectively and 
	\begin{gather*}
		k_{\rm{max}} = \max_{0 \leq n \leq N} k_{n}.
	\end{gather*}
	Let $u_{n}$ and $p_{n}$ be the exact velocity and pressure of the NSE in \eqref{eq:NSE} at time $t_{n}$.
	We need the following upper and lower bound for the ratio of time steps: there are positive constants $C_{L}$ and $C_{U}$ such that
	\begin{gather}
		\label{eq:ratio-limit}
		0 < C_{L} \leq \frac{k_{n}}{k_{n-1}} \leq C_{U}, \ \ \ \forall \ n.		
	\end{gather}
	\begin{theorem}
		\label{thm:error-velocity-L2}
		Suppose the velocity $u \in X$, the pressure $p \in Q$ and the body force $f(x,t)$ of the NSE  
		in \eqref{eq:NSE} satisfy
		\begin{gather*}
			u \in \ell^{\infty}(0,N;H^r) \cap \ell^{\infty}(0,N;H^1) \cap \ell^{\infty,\beta}(0,N;H^{1})
			\cap \ell^{2,\beta}(0,N;H^{r+1} \cap H^{2}), \\
			u_{t} \in L^{2}(0,T;H^{r+1}),  \ \ 
			u_{tt} \in L^{2}(0,T;H^{r+1} \cap H^{1}), \ \ 
			u_{ttt} \in L^{2}(0,T;X'), \\
			p \in \ell^{2,\beta}(0,N;H^{s+1}), \
			f \in L^{2}(0,T;X').
		\end{gather*}
		Under the time step bounds in \eqref{eq:ratio-limit}, the numerical solutions of the semi-implicit DLN scheme in \eqref{eq:DLN-Semi-Alg} satisfy 
		\begin{align}
			\label{eq:error-L2-conclusion}
			&\max_{0 \leq n \leq N} \| u_{n} - u_{n}^{h} \| 
			+ \Big( \nu \sum_{n=1}^{N-1} \widehat{k}_{n} 
			\| \nabla \big( u_{n} - u_{n}^{h} \big) \|^{2} \Big)^{1/2} 
			\leq \mathcal{O} \big( h^{r},h^{s+1}, k_{\rm{max}}^{2} \big).
		\end{align}
	\end{theorem}
	\begin{proof}
	The exact solutions of NSE at time $t_{n,\beta}$ satisfy: 
	\begin{confidential}
		\color{darkblue}
		\begin{gather*}
			\big( u_{t}(t_{n,\beta}), v^{h} \big) + b \big( u(t_{n,\beta}), u(t_{n,\beta}), v^{h} \big)
			+ \nu(\nabla u(t_{n,\beta}),v^{h}) - (p(t_{n,\beta}),\nabla \cdot v^{h}) 
			= (f(t_{n,\beta}),v^{h}).
		\end{gather*}
		\normalcolor
	\end{confidential}
	\begin{gather}
		\label{eq:NSE-exact}
		\frac{1}{\widehat{k}_{n}} \Big( \sum_{\ell=0}^{2}{\alpha_{\ell}} u_{n-1+\ell}, v^{h} \Big)
		+ b \big( \widetilde{u}_{n}, u_{n,\beta}, v^{h} \big) + \nu \big( \nabla u_{n,\beta}, v^{h} \big)
		- \big( p(t_{n,\beta}), \nabla \cdot v^{h} \big)  \\
		= \big( f_{n,\beta}, v^{h} \big) + \tau_{n}(v^{h}),  
		\qquad \qquad \qquad \qquad \forall  v^{h} \in V^{h}   \notag
	\end{gather}
	where $\widetilde{u}_{n}$ is second-order, linear extrapolation of $u_{n,\beta}$ in \eqref{eq:2nd-approx}
	and the truncation error $\tau_{n}$ is
	\begin{align*}
		\tau_{n}(v^{h})
		=& \Big( \frac{1}{\widehat{k}_{n}} \big( \sum_{\ell=0}^{2}{\alpha_{\ell}} u_{n-1+\ell} \big) 
		- u(t_{n,\beta}), v^{h} \Big) + \nu \big( \nabla (u_{n,\beta} - u(t_{n,\beta})), v^{h} \big) \\
		&+ b \big( \widetilde{u}_{n}, u_{n,\beta}, v^{h} \big) 
		 - b ( u(t_{n,\beta}), u(t_{n,\beta}), v^{h})
		 + (f(t_{n,\beta}) - f_{n,\beta},v^{h}).
	\end{align*}
	Let $P_{S}^{(u)}u_{n}$ be velocity component of the Stokes projection of $(u_{n},0)$ onto $V^{h} \times Q^{h}$. We set
	\begin{gather*}
		\phi_{n}^{h} = u_{n}^{h} - P_{S}^{(u)}u_{n}, \ \ \ \eta_{n} = u_{n} - P_{S}^{(u)}u_{n},
	\end{gather*}
	and decompose the error of velocity $e_{n}^{u}$ to be  
	\begin{gather}
    	e_{n}^{u} = u_{n}^{h} - u_{n} = \phi_{n}^{h} - \eta_{n}.
    	\label{eq:velocity-decompose}
    \end{gather} 
	We restrict $v^{h} \in V^{h}$ in \eqref{eq:DLN-Semi-Alg} and subtract \eqref{eq:NSE-exact} 
	from the first equation of \eqref{eq:DLN-Semi-Alg}
	\begin{confidential}
		\color{darkblue}
		\begin{gather*}
			\frac{1}{\widehat{k}_{n}} \Big( \sum_{\ell=0}^{2}{\alpha_{\ell}} u_{n-1+\ell}^{h}, v^{h} \Big)
			- \frac{1}{\widehat{k}_{n}} \Big( \sum_{\ell=0}^{2}{\alpha_{\ell}} u_{n-1+\ell}, v^{h} \Big)
			+ b \big( \widetilde{u}_{n}^{h}, u_{n,\beta}^{h}, v^{h} \big)
			- b \big( \widetilde{u}_{n}, u_{n,\beta}, v^{h} \big) \\
			+ \nu \big( \nabla u_{n,\beta}^{h}, \nabla v^{h} \big) 
			- \nu \big( \nabla u_{n,\beta}, \nabla v^{h} \big) 
			= - (p(t_{n,\beta}) - q^{h}, \nabla \cdot v^{h}) - \tau_{n}(v^{h})
		\end{gather*}
		\normalcolor
	\end{confidential}
	\begin{align}
		\label{eq:Diff-exact-DLN}
		&\frac{1}{\widehat{k}_{n}} \big( \sum_{\ell=0}^{2}{\alpha_{\ell}} \phi_{n-1+\ell}^{h}, v^{h} \big) 
		+ \nu \big( \nabla \phi_{n,\beta}^{h}, \nabla v^{h} \big)
		\\
		=& \frac{1}{\widehat{k}_{n}} \big( \sum_{\ell=0}^{2}{\alpha_{\ell}} \eta_{n-1+\ell}, v^{h} \big)
		+ \nu \big( \nabla \eta_{n,\beta}, \nabla v^{h} \big) 
		+ b \big( \widetilde{u}_{n}, u_{n,\beta}, v^{h} \big) 
		- b \big( \widetilde{u}_{n}^{h}, u_{n,\beta}^{h}, v^{h} \big) \notag \\
		&- (p(t_{n,\beta}) - q^{h}, \nabla \cdot v^{h}) - \tau_{n}(v^{h}),     
		\qquad \qquad \qquad \qquad \qquad \qquad \qquad \forall q^{h} \in Q^{h}
		\notag 
	\end{align}
	We set $v^{h} = \phi_{n,\beta}^{h}$ in \eqref{eq:Diff-exact-DLN} and use the $G$-stability identity in \eqref{eq:G-stab},
	\begin{gather}
		\frac{1}{\widehat{k}_{n}} \Big(\begin{Vmatrix}
			{\phi_{n+1}^{h}} \\
			{\phi_{n}^{h}}
		\end{Vmatrix}
		_{G(\theta)}^{2} -
		\begin{Vmatrix}
			{\phi_{n}^{h}} \\
			{\phi_{n-1}^{h}}
		\end{Vmatrix}%
		_{G(\theta)}^{2} 
		+ \Big\| \sum_{\ell=0}^{2} \gamma_{\ell}^{(n)} \phi_{n-1+\ell}^{h} \Big\| ^{2} \Big)
		+ \nu \| \nabla \phi_{n,\beta}^{h} \|^2  
		\notag  \\
		= 
		\frac{1}{\widehat{k}_{n}} \big( \sum_{\ell=0}^{2}{\alpha_{\ell}} \eta_{n-1+\ell}, 
		\phi_{n,\beta}^{h} \big)
		+ \nu \big( \nabla \eta_{n,\beta}, \nabla \phi_{n,\beta}^{h} \big)
		+ b \big( \widetilde{u}_{n}, u_{n,\beta}, \phi_{n,\beta}^{h} \big) 
		- b \big( \widetilde{u}_{n}^{h}, u_{n,\beta}^{h}, \phi_{n,\beta}^{h} \big)  \notag \\
		- (p(t_{n,\beta}) - q^{h}, \nabla \cdot \phi_{n,\beta}^{h}) - \tau_{n}(\phi_{n,\beta}^{h}).
		\label{eq:error-1-DLN}
	\end{gather}
	By Cauchy-Schwarz inequality, Poincar$\acute{\rm{e}}$ inequality, Young's equality, \eqref{eq:approx-thm} and \eqref{eq:Stoke-Approx} ($\inf_{q^{h}\in Q^{h}} \| p - q^{h} \|$ vanishes since $p = 0$ in the Stokes projection)
 	\begin{confidential}
		\color{darkblue}
		\begin{align*}
			\frac{1}{\widehat{k}_{n}}
			\big( \sum_{\ell=0}^{2}{\alpha_{\ell}} \eta_{n-1+\ell}, \phi_{n,\beta}^{h} \big)
			\leq \frac{1}{\widehat{k}_{n}}
			\big\| \sum_{\ell=0}^{2}{\alpha_{\ell}} \eta_{n-1+\ell} \big\| \| \phi_{n,\beta}^{h} \|
			\leq \frac{C}{\widehat{k}_{n}}
			\big\| \sum_{\ell=0}^{2}{\alpha_{\ell}} \eta_{n-1+\ell} \big\| 
			\| \nabla \phi_{n,\beta}^{h} \|
		\end{align*}
		\normalcolor
	\end{confidential}
	\begin{align}
		\label{eq:error-diff-eta-phi-1}
		\frac{1}{\widehat{k}_{n}}
		\big(\! \sum_{\ell=0}^{2}{\alpha_{\ell}} \eta_{n\!-\!1\!+\!\ell}, \phi_{n,\beta}^{h} \!\big)
		\leq& \frac{C}{\nu \widehat{k}^{2}_{n}} 
		\big\| \sum_{\ell=0}^{2}{\alpha_{\ell}} \eta_{n-1+\ell} \big\|^{2}
		+ \frac{\nu}{16} \| \nabla \phi_{n,\beta}^{h} \|^{2} \\
		\leq& \frac{C h^{2r+2}}{\nu \widehat{k}^{2}_{n}} \big\| \sum_{\ell=0}^{2}{\alpha_{\ell}} u_{n-1+\ell} \big\|_{r+1}^{2} + \frac{\nu}{16} \| \nabla \phi_{n,\beta}^{h} \|^{2} \notag \\
		\leq& \frac{C(\!\theta\!) h^{2r\!+\!2}}{\nu \widehat{k}^{2}_{n}} \big( \|u_{n\!+\!1} \!-\! u_{n}\|_{r\!+\!1}^{2} \!+\! \|u_{n\!+\!1} \!-\! u_{n\!-\!1}\|_{r\!+\!1}^{2} \big)
		\!+\! \frac{\nu}{16} \| \!\nabla \phi_{n,\beta}^{h} \!\|^{2}. \notag 
	\end{align}
	By Holder's inequality
	\begin{confidential}
		\color{darkblue}
		\begin{align*}
			\big\| \sum_{\ell=0}^{2}{\alpha_{\ell}} u_{n-1+\ell} \big\|_{r+1}^{2}
			=& \big\| -\alpha_{1} (u_{n+1} - u_{n}) - \alpha_{0}(u_{n+1} - u_{n-1}) \big\|_{r+1}^{2} \\
			\leq& C(\theta) \big( \|u_{n+1} - u_{n}\|_{r+1}^{2} + \|u_{n+1} - u_{n-1}\|_{r+1}^{2} \big)
		\end{align*}
		\begin{align*}
			\|u_{n+1} - u_{n-1}\|_{k+1}^{2} 
			=& \Big\| \int_{t_{n-1}}^{t+1} \partial_{t} u(\cdot,t) dt \Big\|_{r+1}^{2} 
			= \int_{\Omega} \sum_{\ell \leq r+1} 
			\Big| D^{\ell} \Big( \int_{t_{n-1}}^{t+1} \partial_{t} u(x,t) dt \Big) \Big|^{2} dx  \\
			=& \int_{\Omega} \sum_{\ell \leq r+1}
			\Big| \int_{t_{n-1}}^{t+1} D^{\ell} \partial_{t} u(x,t) dt \Big|^{2} dx
			\leq \int_{\Omega} \sum_{\ell \leq r+1}
			\Big( \int_{t_{n-1}}^{t+1} \big| D^{\ell} \partial_{t} u(x,t) \big| dt \Big)^{2} dx \\
			\leq& \int_{\Omega} \sum_{\ell \leq r+1} 
			\Big[ \big( \int_{t_{n-1}}^{t+1} 1^2 dt \big) 
			\Big( \int_{t_{n-1}}^{t+1} \big| D^{\ell} \partial_{t} u(x,t) \big|^2 dt \Big) \Big] dx \\
			=& (k_{n}+k_{n-1}) \int_{\Omega} \int_{t_{n-1}}^{t+1} 
			\sum_{\ell \leq r+1} \big| D^{\ell} \partial_{t} u(x,t) \big|^2 dt dx \\
			=& (k_{n}+k_{n-1}) \int_{t_{n-1}}^{t+1} \int_{\Omega}
			\sum_{\ell \leq r+1} \big| D^{\ell} \partial_{t} u(x,t) \big|^2 dx dt 
			= (k_{n}+k_{n-1}) \int_{t_{n-1}}^{t_{n+1}} \| u_{t} \|_{r+1}^{2} dt.
		\end{align*}
		Similarly 
		\begin{gather*}
			\|u_{n+1} - u_{n}\|_{r+1}^{2}
			\leq k_{n} \int_{t_{n}}^{t_{n+1}} \| u_{t}\|_{r+1}^{2} dt
		\end{gather*}
		\begin{align*}
			\big\| \sum_{\ell=0}^{2}{\alpha_{\ell}} u_{n-1+\ell} \big\|_{r+1}^{2}
			\leq & C(\theta) \big( \|u_{n+1} - u_{n}\|_{r+1}^{2} + \|u_{n+1} - u_{n-1}\|_{r+1}^{2} \big) \\
			\leq & C(\theta) (k_{n}+k_{n-1}) \int_{t_{n-1}}^{t_{n+1}} \| u_{t} \|_{r+1}^{2} dt.
		\end{align*}
		\normalcolor
	\end{confidential}
	\begin{align}
		\label{eq:error-diff-eta-phi-2}
		\|u_{n+1} - u_{n}\|_{r+1}^{2} 
		&= \Big\| \int_{t_{n}}^{t_{n+1}} u_{t}(\cdot,t) dt \Big\|_{r+1}^{2}
		\leq k_{n} \int_{t_{n}}^{t_{n+1}} \| u_{t} \|_{r+1}^{2} dt, \\
		\|u_{n+1} - u_{n-1}\|_{r+1}^{2}
		&= \Big\| \int_{t_{n-1}}^{t_{n+1}} u_{t}(\cdot,t) dt \Big\|_{r+1}^{2}
		\leq (k_{n}+k_{n-1}) \int_{t_{n-1}}^{t_{n+1}} \| u_{t} \|_{r+1}^{2} dt. \notag 
	\end{align}
	We combine \eqref{eq:error-diff-eta-phi-1} and \eqref{eq:error-diff-eta-phi-2}, 
	\begin{align}
		\label{eq:error-diff-eta-phi}
		\frac{1}{\widehat{k}_{n}}
		\big( \sum_{\ell=0}^{2}{\alpha_{\ell}} \eta_{n-1+\ell}, \phi_{n,\beta}^{h} \big) 
		\leq \frac{C(\theta) h^{2r+2}}{\nu \widehat{k}_{n}} 
		\int_{t_{n-1}}^{t_{n+1}} \| u_{t} \|_{r+1}^{2} dt + \frac{\nu}{16} \| \nabla \phi_{n,\beta}^{h} \|^{2}. 
	\end{align}
	By the definition of the Stokes projection, 
	$\nu \big( \nabla \eta_{n,\beta}, \nabla \phi_{n,\beta}^{h} \big) = 0$.
	We set
	\begin{align*}
		\widetilde{\eta}_{n} 
		&= \beta _{2}^{(n)} \Big[(1 + \frac{k_{n}}{k_{n-1}}) \eta_{n} - \frac{k_{n}}{k_{n-1}} \eta_{n-1} \Big]
		+ \beta _{1}^{(n)} \eta_{n} + \beta _{0}^{(n)} \eta_{n-1}, \\
		\widetilde{\phi}_{n}^{h} 
		&= \beta _{2}^{(n)} \Big[(1 + \frac{k_{n}}{k_{n-1}}) \phi_{n}^{h} - \frac{k_{n}}{k_{n-1}} \phi_{n-1}^{h} \Big]
		+ \beta _{1}^{(n)} \phi_{n}^{h} + \beta _{0}^{(n)} \phi_{n-1}^{h}.
	\end{align*}
	The non-linear terms in \eqref{eq:error-1-DLN} become
	\begin{confidential}
		\color{darkblue}
		\begin{align*}
			&b \big( \widetilde{u}_{n}, u_{n,\beta}, \phi_{n,\beta}^{h} \big) 
		- b \big( \widetilde{u}_{n}^{h}, u_{n,\beta}^{h}, \phi_{n,\beta}^{h} \big) \\
		=& b \big( \widetilde{u}_{n}, u_{n,\beta}, \phi_{n,\beta}^{h} \big) 
		- b \big( \widetilde{u}_{n}^{h}, u_{n,\beta}, \phi_{n,\beta}^{h} \big) 
		+ b \big( \widetilde{u}_{n}^{h}, u_{n,\beta}, \phi_{n,\beta}^{h} \big)
		- b \big( \widetilde{u}_{n}^{h}, u_{n,\beta}^{h}, \phi_{n,\beta}^{h} \big) \\
		=& b \big( \widetilde{u}_{n} - \widetilde{u}_{n}^{h}, u_{n,\beta}, \phi_{n,\beta}^{h} \big)
		+ b \big( \widetilde{u}_{n}^{h}, u_{n,\beta} - u_{n,\beta}^{h}, \phi_{n,\beta}^{h} \big) \\
		=& b \big( \widetilde{\eta}_{n} - \widetilde{\phi}_{n}^{h}, u_{n,\beta}, \phi_{n,\beta}^{h} \big)
		+ b \big( \widetilde{u}_{n}^{h}, \eta_{n,\beta} - \phi_{n,\beta}^{h}, \phi_{n,\beta}^{h} \big) \\
		=& b \big( \widetilde{\eta}_{n}, u_{n,\beta}, \phi_{n,\beta}^{h} \big)
		- b \big( \widetilde{\phi}_{n}^{h}, u_{n,\beta}, \phi_{n,\beta}^{h} \big)
		+ b \big( \widetilde{u}_{n}^{h}, \eta_{n,\beta}, \phi_{n,\beta}^{h} \big)
		- b \big( \widetilde{u}_{n}^{h}, \phi_{n,\beta}^{h}, \phi_{n,\beta}^{h} \big) \\
		=& b \big( \widetilde{\eta}_{n}, u_{n,\beta}, \phi_{n,\beta}^{h} \big)
		- b \big( \widetilde{\phi}_{n}^{h}, u_{n,\beta}, \phi_{n,\beta}^{h} \big)
		+ b \big( \widetilde{u}_{n}^{h} - \widetilde{u}_{n}, \eta_{n,\beta}, \phi_{n,\beta}^{h} \big)
		+ b \big( \widetilde{u}_{n}, \eta_{n,\beta}, \phi_{n,\beta}^{h} \big) \\
		=& b \big( \widetilde{\eta}_{n}, u_{n,\beta}, \phi_{n,\beta}^{h} \big)
		- b \big( \widetilde{\phi}_{n}^{h}, u_{n,\beta}, \phi_{n,\beta}^{h} \big)
		+ b \big( \widetilde{\phi}_{n}^{h} - \widetilde{\eta}_{n}, \eta_{n,\beta}, \phi_{n,\beta}^{h} \big)
		+ b \big( \widetilde{u}_{n}, \eta_{n,\beta}, \phi_{n,\beta}^{h} \big) 
		\end{align*}
		\normalcolor
	\end{confidential}
	\begin{align*}
		&b \big( \widetilde{u}_{n}, u_{n,\beta}, \phi_{n,\beta}^{h} \big) 
		\!-\! b \big( \widetilde{u}_{n}^{h}, u_{n,\beta}^{h}, \phi_{n,\beta}^{h} \big) \\
		=&\! b \big( \widetilde{u}_{n}, u_{n,\beta}, \phi_{n,\beta}^{h} \big) 
		\!-\! b \big( \widetilde{u}_{n}^{h}, u_{n,\beta}, \phi_{n,\beta}^{h} \big) 
		\!+\! b \big( \widetilde{u}_{n}^{h}, u_{n,\beta}, \phi_{n,\beta}^{h} \big)
		\!-\! b \big( \widetilde{u}_{n}^{h}, u_{n,\beta}^{h}, \phi_{n,\beta}^{h} \big)  \\
		=&\! b \big(\! \widetilde{\eta}_{n}, u_{n,\beta}, \phi_{n,\beta}^{h} \!\big)
		\!-\! b \big(\! \widetilde{\phi}_{n}^{h}, u_{n,\beta}, \phi_{n,\beta}^{h} \!\big)
		\!+\! b \big(\! \widetilde{u}_{n}^{h}, \eta_{n,\beta}, \phi_{n,\beta}^{h} \!\big)   \\
		=& \! b \big(\! \widetilde{\eta}_{n}, u_{n\!,\!\beta}, \phi_{n\!,\!\beta}^{h} \!\big)
		\!-\! b \big(\! \widetilde{\phi}_{n}^{h}, u_{n\!,\!\beta}, \phi_{n\!,\!\beta}^{h} \!\big)
		\!+\! b \big(\! \widetilde{\phi}_{n}^{h}, \eta_{n\!,\!\beta}, \phi_{n\!,\!\beta}^{h} \!\big)
		\!-\! b \big(\! \widetilde{\eta}_{n}, \eta_{n\!,\!\beta}, \phi_{n\!,\!\beta}^{h} \!\big)    
		\!+\! b \big(\! \widetilde{u}_{n}, \eta_{n\!,\!\beta}, \phi_{n\!,\!\beta}^{h} \!\big).
	\end{align*}
	By \eqref{eq:b-bound1}, \eqref{eq:b-bound2}, \eqref{eq:inverse-estimator} and \eqref{eq:Stoke-Approx},
	Poincar$\acute{\rm{e}}$ inequality and step requirement in \eqref{eq:ratio-limit}
	\begin{align}
		\label{eq:non-linear-bound1}
		b \big( \widetilde{\eta}_{n}, u_{n,\beta}, \phi_{n,\beta}^{h} \big) 
		\leq& C \| \nabla \widetilde{\eta}_{n} \| \| \nabla u_{n,\beta} \| \| \nabla \phi_{n,\beta}^{h} \|,    \\
		b \big( \widetilde{\phi}_{n}^{h}, u_{n,\beta}, \phi_{n,\beta}^{h} \big) 
		\leq& C(\theta) \| u_{n,\beta} \|_{2} \big( \| \phi_{n}^{h} \| + \| \phi_{n-1}^{h} \| \big)
		\| \nabla \phi_{n,\beta}^{h} \|,  \notag \\
		b \big( \widetilde{u}_{n}, \eta_{n,\beta}, \phi_{n,\beta}^{h} \big) 
		\leq& C(\theta) \big( \| \nabla u_{n} \| + \| \nabla u_{n-1} \| \big)
		\| \nabla \eta_{n,\beta} \| \| \nabla \phi_{n,\beta}^{h} \|,  \notag \\
		b \big( \widetilde{\phi}_{n}^{h}, \eta_{n,\beta}, \phi_{n,\beta}^{h} \big) 
		\leq& C \| \widetilde{\phi}_{n}^{h} \|^{1/2} \| \nabla \widetilde{\phi}_{n}^{h} \|^{1/2} 
		\| \nabla \eta_{n,\beta} \| \| \nabla \phi_{n,\beta}^{h} \|                    \notag \\
		\leq& C h \| \widetilde{\phi}_{n}^{h} \|^{1/2} \| \nabla \widetilde{\phi}_{n}^{h} \|^{1/2} 
		\| u_{n,\beta} \|_{2} \| \nabla \phi_{n,\beta}^{h} \|        \notag \\
		\leq& C(\theta) h^{1/2} \big( \| \phi_{n}^{h} \| + \| \phi_{n-1}^{h} \| \big) \| u_{n,\beta} \|_{2} 
		\| \nabla \phi_{n,\beta}^{h} \|                         \notag \\
		b \big( \widetilde{\eta}_{n}, \eta_{n,\beta}, \phi_{n,\beta}^{h} \big) 
		\leq& C(\theta) \big( \| \nabla \eta_{n} \| + \| \nabla \eta_{n-1} \| \big)
		\| \nabla \eta_{n,\beta} \| \| \nabla \phi_{n,\beta}^{h} \|                   \notag \\ 
		\leq& C(\theta) \big( \| \nabla u_{n} \| + \| \nabla u_{n-1} \| \big) 
		\| \nabla \eta_{n,\beta} \| \| \nabla \phi_{n,\beta}^{h} \|. \notag 
	\end{align}
	We apply Young's inequality to all non-linear terms in \eqref{eq:non-linear-bound1} 
	\begin{confidential}
		\color{darkblue}
		\begin{align*}
			&b \big( \widetilde{u}_{n}, u_{n,\beta}, \phi_{n,\beta}^{h} \big) 
		- b \big( \widetilde{u}_{n}^{h}, u_{n,\beta}^{h}, \phi_{n,\beta}^{h} \big) \\
		\leq& C(\theta) \Big( \| u_{n,\beta} \|_{2} \big( \| \phi_{n}^{h} \| + \| \phi_{n-1}^{h} \| \big) 
		+ \| \nabla \widetilde{\eta}_{n} \| \| \nabla u_{n,\beta} \|
		+ \big( \| \nabla u_{n} \| + \| \nabla u_{n-1} \| \big) \| \nabla \eta_{n,\beta} \| \Big) 
		\| \nabla \phi_{n,\beta}^{h} \| \\
		\leq& \frac{C(\theta)}{\nu} \Big( \| u_{n,\beta} \|_{2}^{2} 
		\big( \| \phi_{n}^{h} \|^{2} + \| \phi_{n-1}^{h} \|^{2} \big) 
		+ \| \nabla \widetilde{\eta}_{n} \|^{2} \| \nabla u_{n,\beta} \|^{2}
		+ \big( \| \nabla u_{n} \|^{2} + \| \nabla u_{n-1} \|^{2} \big) \| \nabla \eta_{n,\beta} \|^{2} \Big)
		+ \frac{\nu}{16} \| \nabla \phi_{n,\beta}^{h} \|^{2} \\
		\leq& \frac{C(\theta)}{\nu} \Big( \| u_{n,\beta} \|_{2}^{2} 
		\big( \| \phi_{n}^{h} \|^{2} + \| \phi_{n-1}^{h} \|^{2} \big) 
		+ \| | \nabla u | \|_{\infty,0}^{2} \| \nabla \widetilde{\eta}_{n} \|^{2} 
		+ \| |\nabla u| \|_{\infty,0}^{2}  \| \nabla \eta_{n,\beta} \|^{2} \Big)
		+ \frac{\nu}{16} \| \nabla \phi_{n,\beta}^{h} \|^{2} \\
		\end{align*}
		\normalcolor
	\end{confidential}
	\begin{align}
		\label{eq:non-linear-bound2}
		&b \big( \widetilde{u}_{n}, u_{n,\beta}, \phi_{n,\beta}^{h} \big) 
		\!-\! b \big( \widetilde{u}_{n}^{h}, u_{n,\beta}^{h}, \phi_{n,\beta}^{h} \big)  \\
		\leq& \!\frac{C(\!\theta\!)}{\nu} \!\Big[\! \| u_{n\!,\!\beta} \|_{2}^{2} \!
		\big(\! \| \phi_{n}^{h} \|^{2} \!+\! \| \phi_{n\!-\!1}^{h} \|^{2} \!\big) 
		\!+\! \| | \!\nabla u | \|_{\infty\!,\!0}^{2} \| \!\nabla \widetilde{\eta}_{n} \|^{2} 
		\!+\! \| | \! \nabla u| \|_{\infty\!,\!0}^{2} \| \!\nabla \eta_{n\!,\!\beta} \|^{2} \!\Big]
		\!+\! \frac{\nu}{16}\! \| \!\nabla \phi_{n\!,\!\beta}^{h} \|^{2}.  \notag 
	\end{align}
	By \eqref{eq:approx-thm}, triangle inequality and \eqref{eq:DLN-Consistency1} 
	in Lemma \ref{lemma:DLN-consistency}
	\begin{confidential}
		\color{darkblue}
		\begin{align*}
			\| \nabla \widetilde{\eta}_{n} \|^{2}
			\leq& C h^{2r} \| \widetilde{u}_{n} \|_{r+1}^{2} 
			\leq C h^{2r} \big( \| \widetilde{u}_{n} - u(t_{n,\beta}) \|_{r+1}^{2} 
			+ \| u(t_{n,\beta}) \|_{r+1}^{2} \big)  \\
			\leq& C(\theta) h^{2r} (k_{n}+k_{n-1})^{3} \int_{t_{n-1}}^{t_{n+1}} \| u_{tt} \|_{r+1}^{2} dt 
			+ C h^{2r} \| u(t_{n,\beta}) \|_{r+1}^{2} 
		\end{align*}
		\begin{align*}
			\| \nabla \eta_{n,\beta} \|^{2}
			\leq& C h^{2r} \| {u}_{n,\beta} \|_{r+1}^{2} 
			\leq C h^{2r} \big( \| {u}_{n,\beta} - u(t_{n,\beta}) \|_{r+1}^{2} 
			+ \| u(t_{n,\beta}) \|_{r+1}^{2} \big)  \\
			\leq& C(\theta) h^{2r} (k_{n}+k_{n-1})^{3} \int_{t_{n-1}}^{t_{n+1}} \| u_{tt} \|_{r+1}^{2} dt 
			+ C h^{2r} \| u(t_{n,\beta}) \|_{r+1}^{2} 
		\end{align*}
		\normalcolor
	\end{confidential}
	\begin{align*}
		\| \!\nabla \widetilde{\eta}_{n}\! \|^{2}
		\leq& C h^{2r} \| \!\widetilde{u}_{n}\! \|_{r\!+\!1}^{2} 
		\leq C(\!\theta\!) h^{2r} (k_{n}\!+\!k_{n-1})^{3} 
		\int_{t_{n\!-\!1}}^{t_{n\!+\!1}} \| \! u_{tt}\! \|_{r\!+\!1}^{2} dt 
		\!+\! C h^{2r} \|\! u(t_{n,\beta})\! \|_{r\!+\!1}^{2}, \\
		\| \!\nabla \eta_{n,\beta} \! \|^{2}
		\leq& C h^{2r} \| \! {u}_{n,\beta} \! \|_{r\!+\!1}^{2} 
		\leq C(\!\theta\!) h^{2r} (k_{n}\!+\!k_{n-1})^{3} 
		\int_{t_{n\!-\!1}}^{t_{n\!+\!1}} \| \!u_{tt} \!\|_{r\!+\!1}^{2} dt 
		\!+\! C h^{2r} \| \!u(t_{n,\beta}) \!\|_{r\!+\!1}^{2}.
	\end{align*}
	\eqref{eq:non-linear-bound2} becomes 
	\begin{align}
		\label{eq:non-linear-bound3}
		b \big(\! \widetilde{u}_{n}, u_{n\!,\!\beta}, \phi_{n\!,\!\beta}^{h} \!\big) 
		\!-\! b \big(\! \widetilde{u}_{n}^{h}, u_{n\!,\!\beta}^{h}, \phi_{n\!,\!\beta}^{h} \!\big)  
		\leq& \!\frac{C(\!\theta\!) \!\| u_{n\!,\!\beta} \|_{2}^{2} }{\nu} \! 
		\big(\! \| \phi_{n}^{h} \|^{2} \!+\! \| \phi_{n\!-\!1}^{h} \|^{2} \!\big) 
		\!+\! \frac{\nu}{16}\! \|\! \nabla \phi_{n\!,\!\beta}^{h} \|^{2}    \\
		\!+\! \frac{C(\!\theta\!) \!h^{2r} }{\nu} \! &\| | \!\nabla u | \|_{\infty\!,\!0}^{2}
		\Big(\! k_{\rm{max}}^{3}\! \int_{t_{n\!-\!1}}^{t_{n\!+\!1}} \| u_{tt} \|_{r\!+\!1}^{2}\! dt 
		\!+\! \| u(t_{n\!,\!\beta}) \|_{r\!+\!1}^{2} \!\Big) \notag 
	\end{align}
	We set $q^{h}$ to be $L^2$ projection of $p(t_{n,\beta})$ onto $Q^{h}$ 
	in \eqref{eq:error-1-DLN} and use \eqref{eq:approx-thm}
	\begin{align}
		\label{eq:error-pressure}
		(p(t_{n,\beta}) \!-\! q^{h}, \nabla \cdot \phi_{n,\beta}^{h})
		\!\leq& \! \sqrt{d} \| p(t_{n,\beta}) \!-\! q^{h} \| \| \nabla \phi_{n,\beta}^{h} \| 
		\!\leq \! \frac{Ch^{2s\!+\!2}}{\nu} \| p(t_{n,\beta}) \|_{s\!+\!1}^{2}
		\!+\! \frac{\nu}{16} \| \nabla \phi_{n,\beta}^{h} \|^{2} 
	\end{align}
	Now we treat $\tau_{n}(\phi_{n,\beta}^{h})$: 
	by \eqref{eq:DLN-Consistency1} and \eqref{eq:DLN-Consistency2} in Lemma \ref{lemma:DLN-consistency}
	\begin{align}
		\label{eq:tau-term1}
		\Big( \frac{\sum_{\ell=0}^{2}{\alpha_{\ell}} u_{n-1+\ell}}{\widehat{k}_{n}}  
		- u(t_{n,\beta}), \phi_{n,\beta}^{h} \Big) 
		\leq& \Big\| \frac{\sum_{\ell=0}^{2}{\alpha_{\ell}} u_{n-1+\ell}}{\widehat{k}_{n}} 
		- u(t_{n,\beta}) \Big\|_{-1} 
		\| \nabla \phi_{n,\beta}^{h} \|   \\
		\leq& \frac{C(\theta)}{\nu} k_{\rm{max}}^{3} \int_{t_{n-1}}^{t_{n+1}} \| u_{ttt} \|_{-1}^{2} dt
		+ \frac{\nu}{16} \| \nabla \phi_{n,\beta}^{h} \|^{2}.  \notag 
	\end{align}
	\begin{align}
		\label{eq:tau-term2}
		\nu \big( \nabla (u_{n,\beta} - u(t_{n,\beta})), \nabla \phi_{n,\beta}^{h} \big) 
		\leq& \nu \big\| \nabla (u_{n,\beta} - u(t_{n,\beta})) \big\| 
		\| \nabla \phi_{n,\beta}^{h} \| \\
		\leq& C(\theta) \nu k_{\rm{max}}^{3} \int_{t_{n-1}}^{t_{n+1}} \| \nabla u_{tt} \|^{2} dt 
		+ \frac{\nu}{16} \| \nabla \phi_{n,\beta}^{h} \|^{2}. \notag 
	\end{align}
	\begin{align}
		\label{eq:tau-term5}
		(f(t_{n,\beta}) - f_{n,\beta},\phi_{n,\beta}^{h}) 
		\leq& \| f(t_{n,\beta}) - f_{n,\beta} \|_{-1} \| \nabla \phi_{n,\beta}^{h} \|   \\
		\leq& \frac{C(\theta)}{\nu} k_{\rm{max}}^{3} \int_{t_{n-1}}^{t_{n+1}} \| f_{tt} \|_{-1}^{2} dt 
		+ \frac{\nu}{16} \| \nabla \phi_{n,\beta}^{h} \|^{2}.   \notag 
	\end{align}
	\begin{confidential}
		\color{darkblue}
		\begin{align*}
			&b \big( \widetilde{u}_{n}, u_{n,\beta}, \phi_{n,\beta}^{h} \big)
			- b ( u(t_{n,\beta}), u(t_{n,\beta}), \phi_{n,\beta}^{h} ) \\
			=& b \big( \widetilde{u}_{n}, u_{n,\beta}, \phi_{n,\beta}^{h} \big)
			- b \big( u(t_{n,\beta}), u_{n,\beta}, \phi_{n,\beta}^{h} \big) 
			+ b \big( u(t_{n,\beta}), u_{n,\beta}, \phi_{n,\beta}^{h} \big)
			- b ( u(t_{n,\beta}), u(t_{n,\beta}), \phi_{n,\beta}^{h} )
		\end{align*}
		\normalcolor
	\end{confidential}
	By \eqref{eq:b-bound1} and \eqref{eq:DLN-Consistency1} in Lemma \ref{lemma:DLN-consistency}
	\begin{align}
		\label{eq:tau-nonlinear}
		&b \big( \widetilde{u}_{n}, u_{n,\beta}, \phi_{n,\beta}^{h} \big)
		- b ( u(t_{n,\beta}), u(t_{n,\beta}), \phi_{n,\beta}^{h} )  \\
		=& b \big( \widetilde{u}_{n} - u(t_{n,\beta}), u_{n,\beta}, \phi_{n,\beta}^{h} \big)
		+ b \big( u(t_{n,\beta}), u_{n,\beta} - u(t_{n,\beta}), \phi_{n,\beta}^{h} \big) \notag \\
		\leq& C \| \nabla \big( \widetilde{u}_{n} - u(t_{n,\beta}) \big) \|
		\| \nabla u_{n,\beta} \| \| \nabla \phi_{n,\beta}^{h} \| 
		+ C \| \nabla u(t_{n,\beta}) \| \| \nabla \big( u_{n,\beta} - u(t_{n,\beta}) \big) \|
		\| \nabla \phi_{n,\beta}^{h} \|    \notag \\
		\leq& \frac{C(\theta)}{\nu} k_{\rm{max}}^{3} \big( \| |u| \|_{\infty,1}^{2} 
		+ \| |u| \|_{\infty,1,\beta}^{2} \big)
		\int_{t_{n-1}}^{t_{n+1}} \| \nabla u_{tt} \|^{2} dt 
		+ \frac{\nu}{16} \| \nabla \phi_{n,\beta}^{h} \|^{2}. \notag 
	\end{align}
	\begin{confidential}
		\color{darkblue}
		\begin{align*}
			C \| \nabla \big( \widetilde{u}_{n} - u(t_{n,\beta}) \big) \|
		\| \nabla u_{n,\beta} \| \| \nabla \phi_{n,\beta}^{h} \|      
		\leq& C \| \nabla \big( \widetilde{u}_{n} - u(t_{n,\beta}) \big) \|
		\| |u| \|_{\infty,1} \| \nabla \phi_{n,\beta}^{h} \| \\
		\leq& \frac{\nu}{32} \| \nabla \phi_{n,\beta}^{h} \|^{2}     
		+ \frac{C}{\nu} \| \nabla \big( \widetilde{u}_{n} - u(t_{n,\beta}) \big) \|^{2}
		\| |u| \|_{\infty,1}^{2} \\
		\leq& \frac{\nu}{32} \| \nabla \phi_{n,\beta}^{h} \|^{2}
		+ \frac{C(\theta)}{\nu} (k_{n} + k_{n-1})^{3} \| |u| \|_{\infty,1}^{2}
		\int_{t_{n-1}}^{t_{n+1}} \| \nabla u_{tt} \|^{2} dt \\
		\leq& \frac{\nu}{32} \| \nabla \phi_{n,\beta}^{h} \|^{2}
		+ \frac{C(\theta)}{\nu} (k_{n} + k_{n-1})^{3} \| |u| \|_{\infty,1}^{2}
		\int_{t_{n-1}}^{t_{n+1}} \| \nabla u_{tt} \|^{2} dt
		\end{align*}
		\begin{align*}
			C \| \nabla u(t_{n,\beta}) \| \| \nabla \big( u_{n,\beta} - u(t_{n,\beta}) \big) \|
		\| \nabla \phi_{n,\beta}^{h} \| 
		\leq& \frac{\nu}{32} \| \nabla \phi_{n,\beta}^{h} \|^{2} 
		+ \frac{C}{\nu} \| \nabla \big( \widetilde{u}_{n} - u(t_{n,\beta}) \big) \|^{2}
		\| \nabla u(t_{n,\beta}) \|^{2}    \\
		\leq& \frac{\nu}{32} \| \nabla \phi_{n,\beta}^{h} \|^{2} 
		+ \frac{C}{\nu} \| |u| \|_{\infty,1,\beta}^{2}
		\| \nabla \big( \widetilde{u}_{n} - u(t_{n,\beta}) \big) \|^{2} \\
		\leq& \frac{\nu}{32} \| \nabla \phi_{n,\beta}^{h} \|^{2} 
		+ \frac{C(\theta)}{\nu} \| |u| \|_{\infty,1,\beta}^{2} (k_{n}+k_{n-1})^{3}
		\int_{t_{n-1}}^{t_{n+1}} \| \nabla u_{tt} \|^{2} dt 
		\end{align*}
		\normalcolor
	\end{confidential}
	We combine \eqref{eq:error-diff-eta-phi}, \eqref{eq:non-linear-bound3}, \eqref{eq:error-pressure}, 
	\eqref{eq:tau-term1}, \eqref{eq:tau-term2}, \eqref{eq:tau-term5}, \eqref{eq:tau-nonlinear} and 
	sum \eqref{eq:error-1-DLN} over $n$ from $1$ to $N-1$ 
	\begin{confidential}
		\color{darkblue}
		\begin{align*}
			&\begin{Vmatrix}
				{\phi_{n+1}^{h}} \\
				{\phi_{n}^{h}}
			\end{Vmatrix}
			_{G(\theta)}^{2} -
			\begin{Vmatrix}
				{\phi_{n}^{h}} \\
				{\phi_{n-1}^{h}}
			\end{Vmatrix}%
			_{G(\theta)}^{2} 
			+ \Big\| \sum_{\ell=0}^{2} \gamma_{\ell}^{(n)} \phi_{n-1+\ell}^{h} \Big\| ^{2} 
			+ \nu \widehat{k}_{n} \| \nabla \phi_{n,\beta}^{h} \|^2    \\
			\leq&  \frac{C(\theta) h^{2r+2}}{\nu} \int_{t_{n-1}}^{t_{n+1}} \| u_{t} \|_{r+1}^{2} dt
			+ \frac{\nu \widehat{k}_{n}}{16} \| \nabla \phi_{n,\beta}^{h} \|^{2}
			+ \frac{C(\theta)}{\nu} \widehat{k}_{n} \| u_{n,\beta} \|_{2}^{2}  
			\big( \| \phi_{n}^{h} \|^{2} + \| \phi_{n-1}^{h} \|^{2} \big) 
			+ \frac{\nu \widehat{k}_{n}}{16} \| \nabla \phi_{n,\beta}^{h} \|^{2}    \\
			+& \frac{C(\theta) h^{2r}}{\nu}  \| | \nabla u | \|_{\infty,0}^{2}
			\Big( k_{\rm{max}}^{4} \int_{t_{n-1}}^{t_{n+1}} \| u_{tt} \|_{r+1}^{2} dt 
			+ \widehat{k}_{n} \| u(t_{n,\beta}) \|_{r+1}^{2} \Big)  \\
			+& \frac{Ch^{2s+2}}{\nu} \widehat{k}_{n} \| p(t_{n,\beta}) \|_{s+1}^{2}
			+ \frac{\nu \widehat{k}_{n}}{16} \| \nabla \phi_{n,\beta}^{h} \|^{2} 
			+ \frac{C(\theta)}{\nu} k_{\rm{max}}^{4} \int_{t_{n-1}}^{t_{n+1}} \| u_{ttt} \|_{-1}^{2} dt
			+ \frac{\nu \widehat{k}_{n}}{16} \| \nabla \phi_{n,\beta}^{h} \|^{2} \\
			+& C(\theta) \nu k_{\rm{max}}^{4} \int_{t_{n-1}}^{t_{n+1}} \| \nabla u_{tt} \|^{2} dt 
			+ \frac{\nu \widehat{k}_{n}}{16} \| \nabla \phi_{n,\beta}^{h} \|^{2} 
			+ \frac{C(\theta)}{\nu} k_{\rm{max}}^{4} \int_{t_{n-1}}^{t_{n+1}} \| f_{tt} \|_{-1}^{2} dt 
			+ \frac{\nu \widehat{k}_{n}}{16} \| \nabla \phi_{n,\beta}^{h} \|^{2} \\
			+& \frac{C(\theta)}{\nu} k_{\rm{max}}^{4} \big( \| |u| \|_{\infty,1}^{2} 
			+ \| |u| \|_{\infty,1,\beta}^{2} \big)
			\int_{t_{n-1}}^{t_{n+1}} \| \nabla u_{tt} \|^{2} dt 
			+ \frac{\nu \widehat{k}_{n}}{16} \| \nabla \phi_{n,\beta}^{h} \|^{2}
		\end{align*}
		\begin{align*}
			&\begin{Vmatrix}
				{\phi_{N}^{h}} \\
				{\phi_{N-1}^{h}}
			\end{Vmatrix}
			_{G(\theta)}^{2} -
			\begin{Vmatrix}
				{\phi_{1}^{h}} \\
				{\phi_{0}^{h}}
			\end{Vmatrix}%
			_{G(\theta)}^{2} 
			+ \sum_{n=1}^{N-1} \Big\| \sum_{\ell=0}^{2} \gamma_{\ell}^{(n)} \phi_{n-1+\ell}^{h} \Big\| ^{2} 
			+ \frac{\nu}{2} \sum_{n=1}^{N-1} \widehat{k}_{n} \| \nabla \phi_{n,\beta}^{h} \|^2    \\
			\leq& \frac{C(\theta) h^{2r+2}}{\nu} \sum_{n=1}^{N-1} 
			\int_{t_{n-1}}^{t_{n+1}} \| u_{t} \|_{r+1}^{2} dt
			+ \frac{C(\theta)}{\nu} \sum_{n=1}^{N-1} \widehat{k}_{n} \| u_{n,\beta} \|_{2}^{2}  
			\big( \| \phi_{n}^{h} \|^{2} + \| \phi_{n-1}^{h} \|^{2} \big) \\
			+& \frac{C(\theta) h^{2r}}{\nu} \| |\nabla u| \|_{\infty,0}^{2} 
			\Big( k_{\rm{max}}^{4} \sum_{n=1}^{N-1} \int_{t_{n-1}}^{t_{n+1}} \| u_{tt} \|_{r+1}^{2} dt 
			+ \sum_{n=1}^{N-1} (k_{n} + k_{n-1}) \| u(t_{n,\beta}) \|_{r+1}^{2} \Big)  \\
			+& \frac{C(\theta) h^{2s+2}}{\nu} \sum_{n=1}^{N-1} (k_{n} + k_{n-1}) \| p(t_{n,\beta}) \|_{s+1}^{2} 
			+ \frac{C(\theta)}{\nu} k_{\rm{max}}^{4} \sum_{n=1}^{N-1}
			\int_{t_{n-1}}^{t_{n+1}} \| u_{ttt} \|_{-1}^{2} dt            \\
			+& C(\theta) \nu k_{\rm{max}}^{4} \sum_{n=1}^{N-1}
			\int_{t_{n-1}}^{t_{n+1}} \| \nabla u_{tt} \|^{2} dt 
			+ \frac{C(\theta)}{\nu} k_{\rm{max}}^{4} \sum_{n=1}^{N-1} 
			\int_{t_{n-1}}^{t_{n+1}} \| f_{tt} \|_{-1}^{2} dt \\
			+& \frac{C(\theta)}{\nu} k_{\rm{max}}^{4} \big( \| |u| \|_{\infty,1}^{2} 
			+ \| |u| \|_{\infty,1,\beta}^{2} \big) \sum_{n=1}^{N-1}
			\int_{t_{n-1}}^{t_{n+1}} \| \nabla u_{tt} \|^{2} dt 
		\end{align*}
		\normalcolor
	\end{confidential}
	\begin{align}
		\label{eq:phi-estimator-L2}
		&\begin{Vmatrix}
			{\phi_{N}^{h}} \\
			{\phi_{N-1}^{h}}
		\end{Vmatrix}
		_{G(\theta)}^{2} -
		\begin{Vmatrix}
			{\phi_{1}^{h}} \\
			{\phi_{0}^{h}}
		\end{Vmatrix}%
		_{G(\theta)}^{2} 
		+ \sum_{n=1}^{N-1} \Big\| \sum_{\ell=0}^{2} \gamma_{\ell}^{(n)} \phi_{n-1+\ell}^{h} \Big\| ^{2} 
		+ \frac{\nu}{2} \sum_{n=1}^{N-1} \widehat{k}_{n} \| \nabla \phi_{n,\beta}^{h} \|^2    \\
		\leq& \frac{C(\theta) h^{2r+2}}{\nu} \sum_{n=1}^{N-1} 
		\int_{t_{n-1}}^{t_{n+1}} \| u_{t} \|_{r+1}^{2} dt
		+ \frac{C(\theta)}{\nu} \sum_{n=1}^{N-1} \widehat{k}_{n} \| u_{n,\beta} \|_{2}^{2}  
		\big( \| \phi_{n}^{h} \|^{2} + \| \phi_{n-1}^{h} \|^{2} \big) \notag \\
		&+ \frac{C(\theta) h^{2r}}{\nu} \| |\nabla u| \|_{\infty,0}^{2} 
		\Big( k_{\rm{max}}^{4} \sum_{n=1}^{N-1} \int_{t_{n-1}}^{t_{n+1}} \| u_{tt} \|_{r+1}^{2} dt 
		+ \sum_{n=1}^{N-1} (k_{n} + k_{n-1}) \| u(t_{n,\beta}) \|_{r+1}^{2} \Big)  \notag \\
		&+ \frac{C(\theta)h^{2s+2}}{\nu} \sum_{n=1}^{N-1} (k_{n} + k_{n-1}) \| p(t_{n,\beta}) \|_{s+1}^{2} 
		+ \frac{C(\theta)k_{\rm{max}}^{4}}{\nu} \sum_{n=1}^{N-1}
		\int_{t_{n-1}}^{t_{n+1}} \| u_{ttt} \|_{-1}^{2} dt  \notag \\
		&+ C(\theta) \nu k_{\rm{max}}^{4} \sum_{n=1}^{N-1}
		\int_{t_{n-1}}^{t_{n+1}} \| \nabla u_{tt} \|^{2} dt 
		+ \frac{C(\theta) k_{\rm{max}}^{4}}{\nu} \sum_{n=1}^{N-1} 
		\int_{t_{n-1}}^{t_{n+1}} \| f_{tt}  \|_{-1}^{2} dt  \notag \\
		&+ \frac{C(\theta) \!k_{\rm{max}}^{4}}{\nu}\! \big( \| |u| \|_{\infty,1}^{2} 
		+ \| |u| \|_{\infty,1,\beta}^{2} \!\big) \!\sum_{n=1}^{N-1}
		\int_{t_{n-1}}^{t_{n+1}} \| \nabla u_{tt} \|^{2} \!dt.  \notag 
	\end{align}
	By the definition of the $\| \cdot \|_{G(\theta)}$-norm in \eqref{eq:G-norm}, 
	\eqref{eq:phi-estimator-L2} becomes 
	\begin{align}
		\label{eq:phi-estimator-L2-1}
		& \| \phi_{N}^{h} \|^2 
		+ C(\theta) \nu \sum_{n=1}^{N-1} \widehat{k}_{n} \| \nabla \phi_{n,\beta}^{h} \|^2    \\
		\leq& \! \frac{C(\!\theta\!)}{\nu}\! \Big[\! \widehat{k}_{N\!-\!1}\! \|\! u_{N\!-\!1\!,\!\beta}\! \|_{2}^{2}
		\| \!\phi_{N\!-\!1}^{h}\! \|^{2}\!\! +\!\! \sum_{n\!=\!1}^{N\!-\!2}
		\big(\! \widehat{k}_{n\!+\!1}\! \|\! u_{n\!+\!1,\beta}\! \|_{2}^{2}\! 
		+\! \widehat{k}_{n}\! \|\! u_{n,\beta} \!\|_{2}^{2}\!  \big)\! \|\! \phi_{n}^{h}\! \|^{2} 
		\!+\! \widehat{k}_{1} \|\! u_{1,\beta}\! \|_{2}^{2} \|\! \phi_{0}^{h}\! \|^{2}\! \Big]  \notag \\
		+& \frac{C(\theta) h^{2r+2}}{\nu} \| u_{t} \|_{2,r+1}^{2}
		+ \frac{C(\theta) h^{2r}}{\nu} 
		\| |\nabla u| \|_{\infty,0}^{2}  \big( k_{\rm{max}}^{4} \| u_{tt} \|_{2,r+1}^{2}
		+ \| |u|\|_{2,r+1,\beta}^{2} \big) \notag \\
		+&\! \frac{C(\theta)h^{2s\!+\!2}}{\nu} \| |p| \|_{2,s\!+\!1,\beta}^{2} 
		\!+\! \frac{C(\!\theta\!)k_{\rm{max}}^{4}}{\nu} \| u_{ttt} \|_{2,-1}^{2} 
		\!+\! C(\!\theta\!) \nu k_{\rm{max}}^{4} \| \nabla u_{tt} \|_{2,0}^{2} 
		\!+\! \frac{C(\theta) \! k_{\rm{max}}^{4}}{\nu} \| f_{tt} \|_{2,\!-1}^{2}   \notag \\
		\!+&\! \frac{C(\theta) k_{\rm{max}}^{4}}{\nu}\! \big( \| |u| \|_{\infty,1}^{2} 
		\!+\! \| |u| \|_{\infty,1,\beta}^{2} \big) \| \nabla u_{tt} \|_{2,0}^{2}         
		+ C(\theta) \big( \| \phi_{1}^{h} \|^2 + \| \phi_{0}^{h} \|^2 \big). \notag 
	\end{align}
	By the discrete Gr$\rm{\ddot{o}}$nwall inequality without restrictions (\cite[p.369]{HR90_SIAM_NA}), 
	\eqref{eq:phi-estimator-L2-1} becomes
	\begin{confidential}
		\color{darkblue}
	
		\noindent The coefficient is 
		\begin{align*}
			\sum_{n=1}^{N-1}  \widehat{k}_{n} \| u_{n,\beta} \|_{2}^{2}
			\leq& \sum_{n=1}^{N-1} (k_{n} + k_{n-1}) \| u_{n,\beta} - u(t_{n,\beta}) \|_{2}^{2}
			+ \sum_{n=1}^{N-1} (k_{n} + k_{n-1}) \| u(t_{n,\beta}) \|_{2}^{2} \\
			\leq& C(\theta) \sum_{n=1}^{M-1} (k_{n} + k_{n-1})^{4} 
			\int_{t_{n-1}}^{t_{n+1}} \| u_{tt} \|_{2}^{2} dt 
			+ \sum_{n=1}^{M-1} (k_{n} + k_{n-1}) \| u(t_{n,\beta}) \|_{2}^{2} \\
			\leq& C(\theta) k_{\rm{max}}^{4} \| u_{tt} \|_{2,2}^{2} 
			+ \| |u| \|_{2,2,\beta}^{2}
		\end{align*}
		\normalcolor
	\end{confidential}
	\begin{align}
		\label{eq:phi-estimator-L2-2}
		\| \phi_{N}^{h} \|^2
		+ C(\theta) \nu \sum_{n=1}^{N-1} \widehat{k}_{n} \| \nabla \phi_{n,\beta}^{h} \|^2   
		\leq \exp \Big( \frac{C(\theta)}{\nu} \sum_{n=1}^{N-1}  \widehat{k}_{n} \| u_{n,\beta} \|_{2}^{2} \Big) 
		F_{1},
	\end{align}
	where
	\begin{align}
		\label{eq:def_F1}
		F_{1} 
		= & \frac{C(\theta) h^{2r+2}}{\nu} \| u_{t} \|_{2,r+1}^{2}
		+ \frac{C(\theta) h^{2r}}{\nu} 
		\| |\nabla u| \|_{\infty,0}^{2}  \big( k_{\rm{max}}^{4} \| u_{tt} \|_{2,r+1}^{2}
		+ \| |u|\|_{2,r+1,\beta}^{2} \big) \notag \\
		&+\! \frac{C(\theta) h^{2s\!+\!2}}{\nu} \| |p| \|_{2,s\!+\!1,\beta}^{2} 
		\!+\! \frac{C(\!\theta\!)k_{\rm{max}}^{4}}{\nu} \| u_{ttt} \|_{2,-1}^{2} 
		\!+\! C(\!\theta\!) \nu k_{\rm{max}}^{4} \| \nabla u_{tt} \|_{2,0}^{2}     \notag \\
		&+\! \frac{C(\theta) k_{\rm{max}}^{4}}{\nu} \| f_{tt} \|_{2,-1}^{2}       
		\!+\! \frac{C(\theta) k_{\rm{max}}^{4}}{\nu}\! \big( \| |u| \|_{\infty,1}^{2} 
		\!+\! \| |u| \|_{\infty,1,\beta}^{2} \big) \| \nabla u_{tt} \|_{2,0}^{2}   \notag \\       
		&+ C(\theta) \big( \| \phi_{1}^{h} \|^2 + \| \phi_{0}^{h} \|^2 \big).
	\end{align}
	By triangle inequality and \eqref{eq:DLN-Consistency1} in Lemma \ref{lemma:DLN-consistency}, 
	\eqref{eq:phi-estimator-L2-2} can be simplifed 
	\begin{align*}
		\| \phi_{N}^{h} \|^2
		+ C(\theta) \nu \sum_{n=1}^{N-1} \widehat{k}_{n} \| \nabla \phi_{n,\beta}^{h} \|^2   
		\leq \exp \Big[ \frac{C(\theta)}{\nu} \big( k_{\rm{max}}^{4} \| u_{tt} \|_{2,2}^{2} 
		+ \| |u| \|_{2,2,\beta}^{2} \big) \Big] F_{1}.
	\end{align*}
	By triangle inequality, \eqref{eq:approx-thm}, \eqref{eq:DLN-Consistency1} 
	in Lemma \ref{lemma:DLN-consistency} and \eqref{eq:phi-estimator-L2-2}
	\begin{confidential}
		\color{darkblue}
		\begin{align*}
			&\max_{0 \leq n \leq N}\| e_{n}^{u} \|  
			\leq \max_{0 \leq n \leq N}\| \phi_{n}^{h} \| + \max_{0 \leq n \leq N}\| \eta_{n} \| \\
			\leq& \exp \Big[ \frac{C(\theta)}{\nu} \big( k_{\rm{max}}^{4} \| u_{tt} \|_{2,2}^{2} 
			+ \| |u| \|_{2,2,\beta}^{2} \big) \Big] \sqrt{F_{1}} 
			+ Ch^{r} \| |u| \|_{\infty,r}.
		\end{align*}
		\begin{align*}
			&\Big( \nu \sum_{n=1}^{N-1} \widehat{k}_{n} \| \nabla e_{n,\beta}^{u} \|^{2} \Big)^{1/2}  \\
			\leq& \Big( 2 \nu \sum_{n=1}^{N-1} \widehat{k}_{n} \| \nabla \phi_{n,\beta}^{h} \|^{2} \Big)^{1/2}
			+ \Big( 2 \nu \sum_{n=1}^{N-1} \widehat{k}_{n} \| \nabla \eta_{n,\beta} \|^{2} \Big)^{1/2} \\
			\leq& \exp \Big[ \frac{C(\theta)}{\nu} \big( k_{\rm{max}}^{4} \| u_{tt} \|_{2,2}^{2} 
			+ \| |u| \|_{2,2,\beta}^{2} \big) \Big] \sqrt{F_{1}} 
			+ Ch^{r} \Big( \nu \sum_{n=1}^{N-1} \widehat{k}_{n} \| u_{n,\beta} \|_{r+1}^2 \Big)^{1/2} \\
			\leq& \exp \Big[ \frac{C(\theta)}{\nu} \big( k_{\rm{max}}^{4} \| u_{tt} \|_{2,2}^{2} 
			+ \| |u| \|_{2,2,\beta}^{2} \big) \Big] \sqrt{F_{1}} \\
			&+ C\sqrt{\nu} h^{r} \Big( \sum_{n=1}^{N-1} \widehat{k}_{n} \| u_{n,\beta} - u(t_{n,\beta}) \|_{r+1}^2 + \sum_{n=1}^{N-1} \widehat{k}_{n} \| u(t_{n,\beta}) \|_{r+1}^2 \Big)^{1/2} \\
			\leq& \exp \Big[ \frac{C(\theta)}{\nu} \big( k_{\rm{max}}^{4} \| u_{tt} \|_{2,2}^{2} 
			+ \| |u| \|_{2,2,\beta}^{2} \big) \Big] \sqrt{F_{1}} \\
			&+ C\sqrt{\nu} h^{r} \Big( C(\theta) \sum_{n=1}^{M-1} (k_{n}+k_{n-1})^{4} 
			\int_{t_{n-1}}^{t_{n+1}}\| u_{tt} \|_{r+1}^2 dt 
			+ \sum_{n=1}^{M-1} (k_{n}+k_{n-1}) \| u(t_{n,\beta}) \|_{r+1}^2 \Big)^{1/2} \\
			\leq& \exp \Big[ \frac{C(\theta)}{\nu} \big( k_{\rm{max}}^{4} \| u_{tt} \|_{2,2}^{2} 
			+ \| |u| \|_{2,2,\beta}^{2} \big) \Big] \sqrt{F_{1}} 
			+ C(\theta) \sqrt{\nu} h^{r} \Big( k_{\rm{max}}^{4} \| u_{tt} \|_{2,r+1}^{2}
			+ \| |u| \|_{2,r+1,\beta}^2 \Big)^{1/2} \\
			\leq& \exp \Big[ \frac{C(\theta)}{\nu} \big( k_{\rm{max}}^{4} \| u_{tt} \|_{2,2}^{2} 
			+ \| |u| \|_{2,2,\beta}^{2} \big) \Big] \sqrt{F_{1}}
			+ C(\theta) \sqrt{\nu} h^{r} \big( k_{\rm{max}}^{2} \| u_{tt} \|_{2,r+1}
			+ \| |u| \|_{2,r+1,\beta} \big),
		\end{align*}

		\normalcolor
	\end{confidential}
	\begin{align}
			\label{eq:error-L2-inf-final-L2}
			\max_{0 \leq n \leq N}\| e_{n}^{u} \|  
			\leq& \max_{0 \leq n \leq N}\| \phi_{n}^{h} \| + \max_{0 \leq n \leq N}\| \eta_{n} \| \\
			\leq& \exp \Big[ \frac{C(\theta)}{\nu} \big( k_{\rm{max}}^{4} \| u_{tt} \|_{2,2}^{2} 
			+ \| |u| \|_{2,2,\beta}^{2} \big) \Big] \sqrt{F_{1}} 
			+ Ch^{r} \| |u| \|_{\infty,r}.  \notag 
	\end{align}
	\begin{align}
		\label{eq:error-L2-inf-final-H1}
		\Big(\! \nu \sum_{n=1}^{N\!-\!1} \widehat{k}_{n} \| \nabla e_{n,\beta}^{u} \|^{2} \!\Big)^{1/2}
		\leq& \Big(\! 2 \nu \sum_{n=1}^{N\!-\!1} \widehat{k}_{n} \| \nabla \phi_{n,\beta}^{h} \|^{2} \!\Big)^{1/2}
		\!+\! \Big(\! 2 \nu \sum_{n=1}^{N\!-\!1} \widehat{k}_{n} \| \nabla \eta_{n,\beta} \|^{2} \!\Big)^{1/2} \\
		\leq& \exp \Big[ \frac{C(\theta)}{\nu} \big( k_{\rm{max}}^{4} \| u_{tt} \|_{2,2}^{2} 
		+ \| |u| \|_{2,2,\beta}^{2} \big) \Big] \sqrt{F_{1}} \notag \\
		& \qquad \qquad \qquad + C(\theta) \sqrt{\nu} h^{r} \big( k_{\rm{max}}^{2} \| u_{tt} \|_{2,r+1}
		+ \| |u| \|_{2,r+1,\beta} \big). \notag 
	\end{align}
	We combine \eqref{eq:error-L2-inf-final-L2} and \eqref{eq:error-L2-inf-final-H1} 
	\begin{align}
		\label{eq:error-L2-inf-final}
		\max_{0 \!\leq \!n \!\leq\! N} \! \| e_{n}^{u} \| 
		\!+\! \Big(\! \nu \!\sum_{n\!=\!1}^{N\!-\!1} \!\widehat{k}_{n} \!\| \!\nabla e_{n,\beta}^{u} \|^{2} \!\Big)^{\frac{1}{2}} 
		&\!\leq \! \exp \!\Big[\! \frac{C(\theta)}{\nu} \big(\! k_{\rm{max}}^{4} \| u_{tt} \|_{2,2}^{2} 
		+ \| |u| \|_{2,2,\beta}^{2} \!\big) \!\Big] \sqrt{F_{1}} \\
		+\!& Ch^{r} \!\| |u| \|_{\infty,r}
		\!+\! C(\!\theta\!) \!\sqrt{\nu} h^{r} \!\big( \!k_{\rm{max}}^{2} \!\| u_{tt} \|_{2,r\!+\!1}
		\!+\! \| |u| \|_{2,r\!+\!1,\beta} \!\big),      \notag 
	\end{align}
	which implies \eqref{eq:error-L2-conclusion}.
	\end{proof}

	\begin{theorem}
		\label{thm:error-velocity-H1}
		Suppose the velocity $u \in X$ and the pressure $p \in Q$ of the NSE 
		in \eqref{eq:NSE} satisfy 
		\begin{gather*}
			u \in \ell^{\infty}(0,N;H^r) \cap \ell^{\infty}(0,N;H^2) \cap \ell^{\infty,\beta}(0,N;H^{2})
			\cap \ell^{2,\beta}(0,N;H^{r+1} \cap H^{2}), \\
			u_{t} \in L^{2}(0,T;H^{r+1}),  \ \ 
			u_{tt} \in L^{2}(0,T;H^{r+1} \cap H^{2}), \ \ 
			u_{ttt} \in L^{2}(0,T;X' \cap L^{2}), \\
			p \!\in\! \ell^{\infty}(0,N;H^{s\!+\!1}) \!\cap\! \ell^{2,\beta}(0,N;H^{\!s+\!1}), \
			p_{t} \!\in\! L^{2}(0,T;H^{s\!+\!1}), \ p_{tt} \!\in\! L^{2}(0,T;H^{\!s+\!1} \!\cap\! H^{1}),
		\end{gather*} and 
		body force $f \in L^{2}(0,T;X'\cap L^{2})$. Under the time step bounds in \eqref{eq:ratio-limit} and the time-diameter condition
		\begin{align}
			\label{eq:time-h-limit}
			k_{\rm{max}} \leq h^{1/4},
		\end{align}
		the numerical solutions by the semi-implicit DLN algorithm in \eqref{eq:DLN-Semi-Alg} satisfy 
		\begin{align}
			&\max_{0 \leq n \leq M}\| u_{n} - u_{n}^{h} \|_{1} 
			\leq \mathcal{O} \big( h^{r},h^{s+1}, k_{\rm{max}}^{2} \big), \label{eq:error-H1-conclusion} \\
			&\sum_{n=1}^{M-1} \frac{\widehat{k}_{n}}{\nu} 
			\Big\| \frac{\sum_{\ell=0}^{2} \alpha_{\ell}(u_{n\!-\!1\!+\!\ell}^{h} 
			\!-\! u_{n\!-\!1\!+\!\ell})}{\widehat{k}_{n}}  \Big\|^{2} 
			\leq \mathcal{O} \big( h^{r},h^{s+1}, k_{\rm{max}}^{2} \big). 
			\label{eq:error-diff-L2-conclusion}
		\end{align}
	\end{theorem}
	\begin{proof}
		Let $(\!P_{S}^{(\!u\!)} u_{n}, P_{S}^{(\!p\!)} p_{n}\!)$ be Stokes projection of $(\!u_n,p_n\!)$ 
		onto $V^{h} \times Q^{h}$. We set
		\begin{gather*}
			\phi_{n}^{h} = u_{n}^{h} - P_{S}^{(u)} u_{n}, \ \ \ \eta_{n} = u_{n} - P_{S}^{(u)} u_{n}, \ \ \ 
			e_{n}^{u} = \phi_{n}^{h} - \eta_{{n}}, \ \ \ 
			\widetilde{e}_{n}^{u} = \widetilde{u}_{n}^{h} - \widetilde{u}_{n}, \\
			\phi_{n,\alpha}^{h} \!=\! \sum_{\ell=0}^{2} \alpha_{\ell} \phi_{n\!-\!1\!+\!\ell}^{h}, \   
			\eta_{n,\alpha} \!=\! \sum_{\ell=0}^{2} \alpha_{\ell} \eta_{n\!-\!1\!+\!\ell}, \ 
			e_{n,\alpha}^{u} \!=\! \sum_{\ell=0}^{2} \alpha_{\ell} \big(u_{n\!-\!1\!+\!\ell}^{h} 
			\!-\! u_{n\!-\!1\!+\!\ell} \big) \!=\! \eta_{n,\alpha} \!-\! \phi_{n,\alpha}^{h}.
		\end{gather*}
		We let $v^{h} = \widehat{k}_{n}^{-1} \phi_{n,\alpha}^{h}$ in \eqref{eq:Diff-exact-DLN}. By the $G$-stability identity in \eqref{eq:G-stab}, \eqref{eq:Diff-exact-DLN} becomes 
		\begin{gather}
			\Big\| \widehat{k}_{n}^{-1} \phi_{n,\alpha}^{h} \Big\|^{2} + 
			\frac{\nu}{\widehat{k}_{n}} \Big(\begin{Vmatrix}
				\nabla {\phi_{n+1}^{h}} \\
				\nabla {\phi_{n}^{h}}
			\end{Vmatrix}
			_{G(\theta)}^{2} -
			\begin{Vmatrix}
				\nabla {\phi_{n}^{h}} \\
				\nabla {\phi_{n-1}^{h}}
			\end{Vmatrix}%
			_{G(\theta)}^{2} 
			+ \Big\| \nabla \big(\sum_{\ell=0}^{2} \gamma_{\ell}^{(n)} \phi_{n\!-\!1\!+\!\ell}^{h} \big) \Big\| ^{2} \Big) 
			\notag  \\
			= \!
			\big(\! \widehat{k}_{n}^{-1} \sum_{\ell=0}^{2}{\alpha_{\ell}} \eta_{n\!-\!1\!+\!\ell}, 
			\widehat{k}_{n}^{-1} \phi_{n,\alpha}^{h} \!\big) \!
			+ \! \nu \big( \nabla \eta_{n,\beta}, \nabla \widehat{k}_{n}^{-1} \phi_{n,\alpha}^{h} \big)
			\!-\! \big(\! p_{n,\beta}, \nabla \cdot \widehat{k}_{n}^{-1} \phi_{n,\alpha}^{h} \!\big) 
			\!-\! \tau_{n} \big(\! \widehat{k}_{n}^{-1} \phi_{n,\alpha}^{h} \! \big)
			\notag \\
			+ b \big( \widetilde{u}_{n}, u_{n,\beta}, \widehat{k}_{n}^{-1} \phi_{n,\alpha}^{h} \big) 
			- b \big( \widetilde{u}_{n}^{h}, u_{n,\beta}^{h}, \widehat{k}_{n}^{-1} \phi_{n,\alpha}^{h}  \big) 
			+ \big(p_{n,\beta} - p(t_{n,\beta}), \nabla \cdot \widehat{k}_{n}^{-1} \phi_{n,\alpha}^{h} \big).
			\label{eq:error-H1-1-DLN}
		\end{gather}
		By Cauchy Schwarz inequality, Young's inequality, \eqref{eq:approx-thm}, \eqref{eq:Stoke-Approx} and H$\ddot{\rm{o}}$lder's inequality
		\begin{confidential}
			\color{darkblue}
			\begin{align*}
				&\big( \widehat{k}_{n}^{-1} \sum_{\ell=0}^{2}{\alpha_{\ell}} \eta_{n-1+\ell}, 
			\widehat{k}_{n}^{-1} \phi_{n,\alpha}^{h} \big)
			\leq \| \widehat{k}_{n}^{-1} \sum_{\ell=0}^{2}{\alpha_{\ell}} \eta_{n-1+\ell} \|
			\| \widehat{k}_{n}^{-1} \phi_{n,\alpha}^{h} \| \\
			\leq& \frac{C}{\widehat{k}_{n}^{2}} \| \sum_{\ell=0}^{2}{\alpha_{\ell}} \eta_{n-1+\ell} \|^{2}
			+ \frac{1}{16}\| \widehat{k}_{n}^{-1} \phi_{n,\alpha}^{h} \|^{2}  \\
			\leq& \frac{C}{\widehat{k}_{n}^{2}} \big( h^{2r+2} 
			\| \sum_{\ell=0}^{2}{\alpha_{\ell}} u_{n-1+\ell} \|_{r+1}^{2} 
			+ \frac{h^{2s+4}}{\nu^{2}} \| \sum_{\ell=0}^{2}{\alpha_{\ell}} p_{n-1+\ell} \|_{s+1}^{2} \big) 
			+ \frac{1}{16}\| \widehat{k}_{n}^{-1} \phi_{n,\alpha}^{h} \|^{2}  \\
			\leq& \frac{C}{\widehat{k}_{n}^{2}} \Big( h^{2r+2} 
			C(\theta) (k_{n}+k_{n-1}) \int_{t_{n-1}}^{t_{n+1}} \| u_{t} \|_{r+1}^{2} dt
			+ \frac{h^{2s+4}}{\nu^{2}} C(\theta) (k_{n}+k_{n-1}) \int_{t_{n-1}}^{t_{n+1}} \| p_{t} \|_{s+1}^{2} dt \Big) 
			+ \frac{1}{16}\| \widehat{k}_{n}^{-1} \phi_{n,\alpha}^{h} \|^{2}
			\end{align*}
			\normalcolor
		\end{confidential}
		\begin{align}
			\label{eq:error-H1-diff-eta-phi}
			&\big( \widehat{k}_{n}^{-1} \sum_{\ell=0}^{2}{\alpha_{\ell}} \eta_{n-1+\ell}, 
			\widehat{k}_{n}^{-1} \phi_{n,\alpha}^{h} \big)  \\
			\leq& \frac{C}{\widehat{k}_{n}^{2}} \Big( h^{2r+2} 
			\Big\| \sum_{\ell=0}^{2}{\alpha_{\ell}} u_{n-1+\ell} \Big\|_{r+1}^{2} 
			+ \frac{h^{2s+4}}{\nu^{2}} \Big\| \sum_{\ell=0}^{2}{\alpha_{\ell}} p_{n-1+\ell} \Big\|_{s+1}^{2} \Big) 
			+ \frac{1}{16}\| \widehat{k}_{n}^{-1} \phi_{n,\alpha}^{h} \|^{2} \notag \\
			\leq& \frac{C(\theta)}{\widehat{k}_{n}} \Big( h^{2r+2} 
			\int_{t_{n-1}}^{t_{n+1}} \| u_{t} \|_{r+1}^{2} dt 
			+ \frac{h^{2s+4}}{\nu^{2}} \int_{t_{n-1}}^{t_{n+1}} \| p_{t} \|_{s+1}^{2} dt \Big) 
			+ \frac{1}{16}\| \widehat{k}_{n}^{-1} \phi_{n,\alpha}^{h} \|^{2}.   \notag 
		\end{align}
		By the definition of Stokes projection and the fact that 
		$\widehat{k}_{n}^{-1} \phi_{n,\alpha}^{h} \in V^{h} \subset X^{h}$ 
		\begin{confidential}
			\color{darkblue}
			\begin{align*}
				&\nu \big( \nabla \eta_{n,\beta}, \nabla \widehat{k}_{n}^{-1} \phi_{n,\alpha}^{h} \big)
			\!-\! \big( p_{n,\beta}, \nabla \cdot \widehat{k}_{n}^{-1} \phi_{n,\alpha}^{h} \big) \\
			= & \nu \big( \nabla \eta_{n,\beta}, \nabla \widehat{k}_{n}^{-1} \phi_{n,\alpha}^{h} \big)
			\!-\! \big( p_{n,\beta} - P_{S}^{(p)}p_{n,\beta} , \nabla \cdot \widehat{k}_{n}^{-1} \phi_{n,\alpha}^{h} \big) = 0.
			\end{align*}
			\normalcolor
		\end{confidential}
		\begin{gather*}
			\nu \big( \nabla \eta_{n,\beta}, \nabla \widehat{k}_{n}^{-1} \phi_{n,\alpha}^{h} \big)
			\!-\! \big( p_{n,\beta}, \nabla \cdot \widehat{k}_{n}^{-1} \phi_{n,\alpha}^{h} \big)
			= 0.
		\end{gather*}
		For  $\tau(\widehat{k}_{n}^{-1} \phi_{n,\alpha}^{h})$: 
		\begin{align}
			\label{eq:tau-H1-term1}
			\Big( \frac{\sum_{\ell=0}^{2}{\alpha_{\ell}} u_{n\!-\!1\!+\!\ell}}{\widehat{k}_{n}}  
			\!-\! u(t_{n,\beta}), \widehat{k}_{n}^{-1} \phi_{n,\alpha}^{h} \Big) 
			\leq& \Big\| \frac{\sum_{\ell=0}^{2}{\alpha_{\ell}} u_{n\!-\!1\!+\!\ell}}{\widehat{k}_{n}} 
			\!-\! u(t_{n,\beta}) \Big\| 
			\| \widehat{k}_{n}^{-1} \phi_{n,\alpha}^{h} \|   \\
			\leq& C(\theta) k_{\rm{max}}^{3} \int_{t_{n-1}}^{t_{n+1}} \| u_{ttt} \|^{2} dt
			+ \frac{1}{16} \| \widehat{k}_{n}^{-1} \phi_{n,\alpha}^{h} \|^{2}.  \notag 
		\end{align}
		\begin{align}
			\label{eq:tau-H1-term2}
			\nu \big( \nabla (u_{n,\beta} - u(t_{n,\beta})), \nabla \widehat{k}_{n}^{-1} \phi_{n,\alpha}^{h} \big) 
			=& \nu \big( \Delta (u_{n,\beta} - u(t_{n,\beta})), \widehat{k}_{n}^{-1} \phi_{n,\alpha}^{h} \big) \\
			\leq& C(\theta) \nu^{2} k_{\rm{max}}^{3} \int_{t_{n-1}}^{t_{n+1}} \| u_{tt} \|_{2}^{2} dt 
			+ \frac{1}{16} \| \widehat{k}_{n}^{-1} \phi_{n,\alpha}^{h} \|^{2}. \notag 
		\end{align}
		\begin{align}
		\label{eq:tau-H1-term5}
		(f(t_{n,\beta}) - f_{n,\beta},\widehat{k}_{n}^{-1} \phi_{n,\alpha}^{h}) 
		\leq& \| f(t_{n,\beta}) - f_{n,\beta} \| \| \widehat{k}_{n}^{-1} \phi_{n,\alpha}^{h} \|   \\
		\leq& C(\theta) k_{\rm{max}}^{3} \int_{t_{n-1}}^{t_{n+1}} \| f_{tt} \|^{2} dt 
		+ \frac{1}{16} \| \widehat{k}_{n}^{-1} \phi_{n,\alpha}^{h} \|^{2}.   \notag 
		\end{align}
		\begin{confidential}
			\color{darkblue}
			\begin{align*}
				&b \big( \widetilde{u}_{n}, u_{n,\beta}, \widehat{k}_{n}^{-1} \phi_{n,\alpha}^{h} \big)
				- b ( u(t_{n,\beta}), u(t_{n,\beta}), \widehat{k}_{n}^{-1} \phi_{n,\alpha}^{h} ) \\
				=& b \big( \widetilde{u}_{n}, u_{n,\beta}, \widehat{k}_{n}^{-1} \phi_{n,\alpha}^{h} \big)
				- b \big( u(t_{n,\beta}), u_{n,\beta}, \widehat{k}_{n}^{-1} \phi_{n,\alpha}^{h} \big) 
				+ b \big( u(t_{n,\beta}), u_{n,\beta}, \widehat{k}_{n}^{-1} \phi_{n,\alpha}^{h} \big)
				- b ( u(t_{n,\beta}), u(t_{n,\beta}), \widehat{k}_{n}^{-1} \phi_{n,\alpha}^{h} )
			\end{align*}
			\normalcolor
		\end{confidential}
		By \eqref{eq:b-bound1}, \eqref{eq:b-bound3} and \eqref{eq:DLN-Consistency1}
		\begin{align}
			\label{eq:tau-nonlinear-H1}
			&b \big( \widetilde{u}_{n}, u_{n,\beta}, \widehat{k}_{n}^{-1} \phi_{n,\alpha}^{h} \big)
			- b ( u(t_{n,\beta}), u(t_{n,\beta}), \widehat{k}_{n}^{-1} \phi_{n,\alpha}^{h} )  \\
			=& b \big( \widetilde{u}_{n} - u(t_{n,\beta}), u_{n,\beta}, \widehat{k}_{n}^{-1} \phi_{n,\alpha}^{h} \big)
			+ b \big( u(t_{n,\beta}), u_{n,\beta} - u(t_{n,\beta}), \widehat{k}_{n}^{-1} \phi_{n,\alpha}^{h} \big) \notag \\
			\leq& C \| \widetilde{u}_{n} - u(t_{n,\beta}) \|_{1}
			\| u_{n,\beta} \|_{2} \| \widehat{k}_{n}^{-1} \phi_{n,\alpha}^{h} \| 
			+ C \| u(t_{n,\beta}) \|_{2} \| u_{n,\beta} - u(t_{n,\beta}) \|_{1}
			\| \widehat{k}_{n}^{-1} \phi_{n,\alpha}^{h} \|    \notag \\
			\leq& C(\theta) k_{\rm{max}}^{3} \big( \| |u| \|_{\infty,2}^{2} 
			+ \| |u| \|_{\infty,2,\beta}^{2}\big) \int_{t_{n-1}}^{t_{n+1}} \| u_{tt} \|_{1}^{2} dt 
			+ \frac{1}{16} \| \widehat{k}_{n}^{-1} \phi_{n,\alpha}^{h} \|^{2}. \notag 
		\end{align}
		\begin{confidential}
			\color{darkblue}
			\begin{align*}
				&  b \big( \widetilde{u}_{n}^{h}, u_{n,\beta}^{h}, \widehat{k}_{n}^{-1} \phi_{n,\alpha}^{h} \big)
				- b \big( \widetilde{u}_{n}, u_{n,\beta}, \widehat{k}_{n}^{-1} \phi_{n,\alpha}^{h} \big)\\
				=& b \big( \widetilde{u}_{n}^{h}, u_{n,\beta}^{h}, \widehat{k}_{n}^{-1} \phi_{n,\alpha}^{h} \big) 
				- b \big( \widetilde{u}_{n}, u_{n,\beta}^{h}, \widehat{k}_{n}^{-1} \phi_{n,\alpha}^{h} \big) 
				+ b \big( \widetilde{u}_{n}, u_{n,\beta}^{h}, \widehat{k}_{n}^{-1} \phi_{n,\alpha}^{h} \big) 
				- b \big( \widetilde{u}_{n}, u_{n,\beta}, \widehat{k}_{n}^{-1} \phi_{n,\alpha}^{h} \big) \\
				=& b \big( \widetilde{e}_{n}, u_{n,\beta}^{h}, \widehat{k}_{n}^{-1} \phi_{n,\alpha}^{h} \big)
				+ b \big( \widetilde{u}_{n}, e_{n,\beta}, \widehat{k}_{n}^{-1} \phi_{n,\alpha}^{h} \big) \\
				=& b \big( \widetilde{e}_{n}, u_{n,\beta}^{h}, \widehat{k}_{n}^{-1} \phi_{n,\alpha}^{h} \big) 
				- b \big( \widetilde{e}_{n}, u_{n,\beta}, \widehat{k}_{n}^{-1} \phi_{n,\alpha}^{h} \big)
				+ b \big( \widetilde{e}_{n}, u_{n,\beta}, \widehat{k}_{n}^{-1} \phi_{n,\alpha}^{h} \big)
				+ b \big( \widetilde{u}_{n}, e_{n,\beta}, \widehat{k}_{n}^{-1} \phi_{n,\alpha}^{h} \big) \\
				=& b \big( \widetilde{e}_{n}, u_{n,\beta}, \widehat{k}_{n}^{-1} \phi_{n,\alpha}^{h} \big)
				+ b \big( \widetilde{u}_{n}, e_{n,\beta}, \widehat{k}_{n}^{-1} \phi_{n,\alpha}^{h} \big)
				+ b \big( \widetilde{e}_{n}, e_{n,\beta}, \widehat{k}_{n}^{-1} \phi_{n,\alpha}^{h} \big) 
			\end{align*}
			\normalcolor
		\end{confidential}
		For non-linear terms
		\begin{align*}
			& b \big( \widetilde{u}_{n}^{h}, u_{n,\beta}^{h}, \widehat{k}_{n}^{-1} \phi_{n,\alpha}^{h} \big)
			- b \big( \widetilde{u}_{n}, u_{n,\beta}, \widehat{k}_{n}^{-1} \phi_{n,\alpha}^{h} \big)
			\\
			= & b \big( \widetilde{e}_{n}^{u}, u_{n,\beta}, \widehat{k}_{n}^{-1} \phi_{n,\alpha}^{h} \big)
			+ b \big( \widetilde{u}_{n}, e_{n,\beta}^{u}, \widehat{k}_{n}^{-1} \phi_{n,\alpha}^{h} \big)
			+ b \big( \widetilde{e}_{n}^{u}, e_{n,\beta}^{u}, \widehat{k}_{n}^{-1} \phi_{n,\alpha}^{h} \big).
		\end{align*}
		By \eqref{eq:b-bound1}, \eqref{eq:b-bound3} and inverse inequality in \eqref{eq:inverse-estimator}
		\begin{align*}
			b \big( \widetilde{e}_{n}^{u}, u_{n,\beta}, \widehat{k}_{n}^{-1} \phi_{n,\alpha}^{h} \big) 
			=& b \big( \widetilde{\phi}_{n}^{h}, u_{n,\beta}, \widehat{k}_{n}^{-1} \phi_{n,\alpha}^{h} \big)
			- b \big( \widetilde{\eta}_{n}, u_{n,\beta}, \widehat{k}_{n}^{-1} \phi_{n,\alpha}^{h} \big) 
			\notag \\
			\leq& C  \| \widetilde{\phi}_{n}^{h} \|_{1} \| u_{n,\beta} \|_{2} 
			\| \widehat{k}_{n}^{-1} \phi_{n,\alpha}^{h} \|
			+ C  \| \widetilde{\eta}_{n} \|_{1} \| u_{n,\beta} \|_{2} 
			\| \widehat{k}_{n}^{-1} \phi_{n,\alpha}^{h} \|,  \notag \\
			b \big( \widetilde{u}_{n}, e_{n,\beta}^{u}, \widehat{k}_{n}^{-1} \phi_{n,\alpha}^{h} \big)
			\leq& C \| \widetilde{u}_{n} \|_{2} \| e_{n,\beta}^{u} \|_{1} 
			\| \widehat{k}_{n}^{-1} \phi_{n,\alpha}^{h} \|,  \notag \\
			b \big( \widetilde{e}_{n}, e_{n,\beta}^{u}, \widehat{k}_{n}^{-1} \phi_{n,\alpha}^{h} \big)
			=& b \big( \widetilde{\phi}_{n}^{h}, e_{n,\beta}^{u}, \widehat{k}_{n}^{-1} \phi_{n,\alpha}^{h} \big)
			- b \big( \widetilde{\eta}_{n}, e_{n,\beta}^{u}, \widehat{k}_{n}^{-1} \phi_{n,\alpha}^{h} \big)
			\notag \\
			\leq& C h^{-\frac{1}{2}} \| \widetilde{\phi}_{n}^{h} \|_{1} \| e_{n,\beta}^{u} \|_{1} 
			\| \widehat{k}_{n}^{-1} \phi_{n,\alpha}^{h} \|
			+ C h^{-\frac{1}{2}} \| \widetilde{\eta}_{n} \|_{1} \| e_{n,\beta}^{u} \|_{1} 
			\| \widehat{k}_{n}^{-1} \phi_{n,\alpha}^{h} \|.
		\end{align*}
		Thus 
		\begin{align}
			\label{eq:non-linear-H1-bound}
			&b \big( \widetilde{u}_{n}^{h}, u_{n,\beta}^{h}, \widehat{k}_{n}^{-1} \phi_{n,\alpha}^{h} \big)
			- b \big( \widetilde{u}_{n}, u_{n,\beta}, \widehat{k}_{n}^{-1} \phi_{n,\alpha}^{h} \big) \\
			\leq& C \big( \| u_{n,\beta} \|_{2} + h^{-1/2} \| e_{n,\beta}^{u} \|_{1} \big) 
			\| \widetilde{\phi}_{n}^{h} \|_{1} \| \widehat{k}_{n}^{-1} \phi_{n,\alpha}^{h} \| 
			+ C  \| \widetilde{\eta}_{n} \|_{1} \| u_{n,\beta} \|_{2} 
			\| \widehat{k}_{n}^{-1} \phi_{n,\alpha}^{h} \|  \notag \\
			&+ C \| \widetilde{u}_{n} \|_{2} \| e_{n,\beta}^{u} \|_{1} 
			\| \widehat{k}_{n}^{-1} \phi_{n,\alpha}^{h} \| 
			+ C h^{-1/2} \| \widetilde{\eta}_{n} \|_{1} \| e_{n,\beta}^{u} \|_{1} 
			\| \widehat{k}_{n}^{-1} \phi_{n,\alpha}^{h} \|.     \notag 
		\end{align}
		By Cauchy-Schwarz inequality, Young's inequality, Poincar$\acute{\rm{e}}$ inequality, \eqref{eq:approx-thm}, \eqref{eq:Stoke-Approx} and \eqref{eq:DLN-Consistency1}
		\begin{align}
			\label{eq:non-linear-H1-term1}
			& C \big( \| u_{n,\beta} \|_{2} + h^{-1/2} \| e_{n,\beta}^{u} \|_{1} \big) 
			\| \widetilde{\phi}_{n}^{h} \|_{1} \| \widehat{k}_{n}^{-1} \phi_{n,\alpha}^{h} \| \\
			\leq & C(\theta) \big( \| u_{n,\beta} \|_{2}+ h^{-1/2} \| e_{n,\beta}^{u} \|_{1} \big)
			\big( \| \nabla \phi_{n}^{h} \| + \| \nabla \phi_{n-1}^{h} \| \big)
			+ \frac{1}{64} \| \widehat{k}_{n}^{-1} \phi_{n,\alpha}^{h} \|^{2} \notag \\
			\leq & C(\theta) \big( \| u_{n,\beta} \|_{2}^{2} + h^{-1}\| e_{n,\beta} \|_{1}^{2} \big)
			\big( \| \nabla \phi_{n}^{h} \|^{2} + \| \nabla \phi_{n-1}^{h} \|^{2} \big)
			+ \frac{1}{64} \| \widehat{k}_{n}^{-1} \phi_{n,\alpha}^{h} \|^{2}, \notag
		\end{align}
		\begin{align}
			\label{eq:non-linear-H1-term2}
			&C  \| \widetilde{\eta}_{n} \|_{1} \| u_{n,\beta} \|_{2}  
			\| \widehat{k}_{n}^{-1} \phi_{n,\alpha}^{h} \|  \\
			\leq& C(\theta) \| |u| \|_{\infty,2}^{2} \big( \frac{h^{2s+2}}{\nu^{2}} 
			\| \widetilde{p}_{n} \|_{s+1}^{2} + h^{2r} \| \widetilde{u}_{n} \|_{r+1}^{2} \big) 
			+ \frac{1}{64} \| \widehat{k}_{n}^{-1} \phi_{n,\alpha}^{h} \|^{2} \notag \\
			\leq& C(\theta) \| |u| \|_{\infty,2}^{2} \Big[ \frac{h^{2s+2}}{\nu^{2}} 
			\big( \| \widetilde{p}_{n} - p(t_{n,\beta}) \|_{s+1}^{2} + \| p(t_{n,\beta})\|_{s+1}^{2} \big) \notag \\
			&\qquad \qquad \qquad + h^{2r} \big( \| \widetilde{u}_{n} - u(t_{n,\beta}) \|_{r+1}^{2} + \| u(t_{n,\beta}) \|_{r+1}^{2} \big) \Big]  
			+ \frac{1}{64} \| \widehat{k}_{n}^{-1} \phi_{n,\alpha}^{h} \|^{2}  \notag \\ 
			\leq& C(\theta) \| |u| \|_{\infty,2}^{2} \Big[ \frac{h^{2s\!+\!2}}{\nu^{2}} 
			\Big( C(\theta) k_{\rm{max}}^{3} \int_{t_{n-1}}^{t_{n+1}} \| p_{tt} \|_{s\!+\!1}^{2} dt 
			+ \| p(t_{n,\beta})\|_{s\!+\!1}^{2}  \Big) \notag \\
			&\qquad \qquad \qquad + h^{2r} \Big( C(\theta) k_{\rm{max}}^{3} \int_{t_{n\!-\!1}}^{t_{n\!+\!1}} \| u_{tt} \|_{r\!+\!1}^{2} dt + \| u(t_{n,\beta})\|_{r\!+\!1}^{2} \Big) \Big]  
			+ \frac{1}{64} \| \widehat{k}_{n}^{-1} \phi_{n,\alpha}^{h} \|^{2}, \notag 
		\end{align}
		\begin{align}
			\label{eq:non-linear-H1-term3}
			C \| \widetilde{u}_{n} \|_{2} \| e_{n,\beta}^{u} \|_{1} 
			\| \widehat{k}_{n}^{-1} \phi_{n,\alpha}^{h} \|
			\leq& C \| \widetilde{u}_{n} \|_{2}^{2} \| e_{n,\beta}^{u} \|_{1}^{2}
			+ \frac{1}{64} \| \widehat{k}_{n}^{-1} \phi_{n,\alpha}^{h} \|^{2} \\
			\leq& C \| |u| \|_{\infty,2}^{2} \| e_{n,\beta}^{u} \|_{1}^{2}
			+ \frac{1}{64} \| \widehat{k}_{n}^{-1} \phi_{n,\alpha}^{h} \|^{2}, \notag 
		\end{align}
		\begin{align}
			\label{eq:non-linear-H1-term4}
			& C h^{-1/2} \| \widetilde{\eta}_{n} \|_{1} \| e_{n,\beta}^{u} \|_{1} 
			\| \widehat{k}_{n}^{-1} \phi_{n,\alpha}^{h} \|  \\
			\leq& C h^{-1} \| \widetilde{\eta}_{n} \|_{1}^{2} \| e_{n,\beta}^{u} \|_{1}^{2}
			+ \frac{1}{64} \| \widehat{k}_{n}^{-1} \phi_{n,\alpha}^{h} \|^{2}  \notag \\
			\leq& C \big( \frac{h^{2s+1}}{\nu^{2}} \| |p| \|_{\infty,s+1}^{2} 
			+ h^{2r-1} \| |u| \|_{\infty,r+1}^{2} \big) \| e_{n,\beta}^{u} \|_{1}^{2}
			+ \frac{1}{64} \| \widehat{k}_{n}^{-1} \phi_{n,\alpha}^{h} \|^{2}. \notag 
		\end{align}
		\begin{confidential}
			\color{darkblue}
			\begin{align*}
				& C \big( \| \widetilde{\phi}_{n}^{h} \|_{1} \| u_{n,\beta} \|_{2}
				+ h^{-1/2} \| \widetilde{\phi}_{n}^{h} \|_{1} \| e_{n,\beta} \|_{1} \big) 
				\| \widehat{k}_{n}^{-1} \phi_{n,\alpha}^{h} \| \\
				\leq & C \big( \| \widetilde{\phi}_{n}^{h} \|_{1}^{2} \| u_{n,\beta} \|_{2}^{2}
				+ h^{-1} \| \widetilde{\phi}_{n}^{h} \|_{1}^{2} \| e_{n,\beta} \|_{1}^{2} \big)
				+ \frac{1}{64} \| \widehat{k}_{n}^{-1} \phi_{n,\alpha}^{h} \|^{2} \\
				\leq & C \big( \| \nabla \widetilde{\phi}_{n}^{h} \|^{2} \| u_{n,\beta} \|_{2}^{2}
				+ h^{-1} \| \nabla \widetilde{\phi}_{n}^{h} \|^{2} \| e_{n,\beta} \|_{1}^{2} \big)
				+ \frac{1}{64} \| \widehat{k}_{n}^{-1} \phi_{n,\alpha}^{h} \|^{2} \\
				\leq & C(\theta) \big( \| u_{n,\beta} \|_{2}^{2} + h^{-1}\| e_{n,\beta} \|_{1}^{2} \big)
				\big( \| \nabla \phi_{n}^{h} \|^{2} + \| \nabla \phi_{n-1}^{h} \|^{2} \big)
				+ \frac{1}{64} \| \widehat{k}_{n}^{-1} \phi_{n,\alpha}^{h} \|^{2}.
			\end{align*}
			\begin{align*}
				&C  \| \widetilde{\eta}_{n} \|_{1} \| u_{n,\beta} \|_{2} 
				\| \widehat{k}_{n}^{-1} \phi_{n,\alpha}^{h} \|  
				\leq C  \| \widetilde{\eta}_{n} \|_{1}^{2} \| u_{n,\beta} \|_{2}^{2}
				+ \frac{1}{64} \| \widehat{k}_{n}^{-1} \phi_{n,\alpha}^{h} \|^{2} \\
				\leq& C(\theta) \| |u| \|_{\infty,2}^{2} \big( \frac{h^{2s+2}}{\nu^{2}} 
				\| \widetilde{p}_{n} \|_{s+1}^{2} + h^{2r} \| \widetilde{u}_{n} \|_{r+1}^{2} \big) 
				+ \frac{1}{64} \| \widehat{k}_{n}^{-1} \phi_{n,\alpha}^{h} \|^{2} \\
				\leq& C(\theta) \| |u| \|_{\infty,2}^{2} \Big[ \frac{h^{2s+2}}{\nu^{2}} 
				\big( \| \widetilde{p}_{n} - p(t_{n,\beta}) \|_{s+1}^{2} + \| p(t_{n,\beta})\|_{s+1}^{2} \big) \\
				&+ h^{2r} \big( \| \widetilde{u}_{n} - u(t_{n,\beta}) \|_{r+1}^{2} + \| u(t_{n,\beta}) \|_{r+1}^{2} \big) \Big]  
				+ \frac{1}{64} \| \widehat{k}_{n}^{-1} \phi_{n,\alpha}^{h} \|^{2} \\
				\leq& C(\theta) \| |u| \|_{\infty,2}^{2} \Big[ \frac{h^{2s+2}}{\nu^{2}} 
				\Big( C(\theta) k_{\rm{max}}^{3} \int_{t_{n-1}}^{t_{n+1}} \| p_{tt} \|_{s+1}^{2} dt 
				+ \| p(t_{n,\beta})\|_{s+1}^{2}  \Big) \\
				&+ h^{2r} \big( C(\theta) k_{\rm{max}}^{3} \int_{t_{n-1}}^{t_{n+1}} \| u_{tt} \|_{r+1}^{2} dt + \| u(t_{n,\beta})\|_{r+1}^{2} \big) \Big]  
				+ \frac{1}{64} \| \widehat{k}_{n}^{-1} \phi_{n,\alpha}^{h} \|^{2}
			\end{align*}
			\begin{align*}
				C \| \widetilde{u}_{n} \|_{2} \| e_{n,\beta} \|_{1} 
				\| \widehat{k}_{n}^{-1} \phi_{n,\alpha}^{h} \|
				\leq& C \| \widetilde{u}_{n} \|_{2}^{2} \| e_{n,\beta} \|_{1}^{2}
				+ \frac{1}{64} \| \widehat{k}_{n}^{-1} \phi_{n,\alpha}^{h} \|^{2} \\
				\leq& C \| |u| \|_{\infty,2}^{2} \| e_{n,\beta} \|_{1}^{2}
				+ \frac{1}{64} \| \widehat{k}_{n}^{-1} \phi_{n,\alpha}^{h} \|^{2}.
			\end{align*}
			\begin{align*}
				C h^{-1/2} \| \widetilde{\eta}_{n} \|_{1} \| e_{n,\beta} \|_{1} 
				\| \widehat{k}_{n}^{-1} \phi_{n,\alpha}^{h} \| 
				\leq& C h^{-1} \| \widetilde{\eta}_{n} \|_{1}^{2} \| e_{n,\beta} \|_{1}^{2}
				+ \frac{1}{64} \| \widehat{k}_{n}^{-1} \phi_{n,\alpha}^{h} \|^{2}  \\
				\leq& C h^{-1} \big( \frac{h^{2s+2}}{\nu^{2}} \| \widetilde{p}_{n} \|_{s+1}^{2} 
				+ h^{2r} \| \widetilde{u}_{n} \|_{r+1}^{2} \big) \| e_{n,\beta} \|_{1}^{2}
				+ \frac{1}{64} \| \widehat{k}_{n}^{-1} \phi_{n,\alpha}^{h} \|^{2} \\
				\leq& C \big( \frac{h^{2s+1}}{\nu^{2}} \| |p| \|_{\infty,s+1}^{2} 
				+ h^{2r-1} \| |u| \|_{\infty,r+1}^{2} \big) \| e_{n,\beta} \|_{1}^{2}
				+ \frac{1}{64} \| \widehat{k}_{n}^{-1} \phi_{n,\alpha}^{h} \|^{2}.
			\end{align*}
			\normalcolor
		\end{confidential}
		By \eqref{eq:non-linear-H1-term1}, \eqref{eq:non-linear-H1-term2}, 
		\eqref{eq:non-linear-H1-term3} and \eqref{eq:non-linear-H1-term4}, \eqref{eq:non-linear-H1-bound} becomes 
		\begin{align}
			\label{eq:non-linear-H1-bound2}
			&b \big( \widetilde{u}_{n}^{h}, u_{n,\beta}^{h}, \widehat{k}_{n}^{-1} \phi_{n,\alpha}^{h} \big)
			- b \big( \widetilde{u}_{n}, u_{n,\beta}, \widehat{k}_{n}^{-1} \phi_{n,\alpha}^{h} \big)  \\
			\leq& C(\theta) \big( \| u_{n,\beta} \|_{2}^{2} + h^{-1}\| e_{n,\beta}^{u} \|_{1}^{2} \big)
			\big( \| \nabla \phi_{n}^{h} \|^{2} + \| \nabla \phi_{n-1}^{h} \|^{2} \big) 
			+ C \| |u| \|_{\infty,2}^{2} \| e_{n,\beta}^{u} \|_{1}^{2} \notag \\
			&+ C \big( \frac{h^{2s+1}}{\nu^{2}} \| |p| \|_{\infty,s+1}^{2} 
			+ h^{2r-1} \| |u| \|_{\infty,r+1}^{2} \big) \| e_{n,\beta} \|_{1}^{2}
			+ \frac{1}{16} \| \widehat{k}_{n}^{-1} \phi_{n,\alpha}^{h} \|^{2} \notag \\
			&+ C(\theta) \| |u| \|_{\infty,2}^{2} \Big[ \frac{h^{2s\!+\!2}}{\nu^{2}} 
			\Big( k_{\rm{max}}^{3} \int_{t_{n-1}}^{t_{n+1}} \| p_{tt} \|_{s\!+\!1}^{2} dt 
			+ \| p(t_{n,\beta})\|_{s\!+\!1}^{2}  \Big) \notag \\
			& \qquad \qquad \qquad \qquad \qquad \qquad 
			+ h^{2r} \Big( k_{\rm{max}}^{3} \int_{t_{n\!-\!1}}^{t_{n\!+\!1}} \| u_{tt} \|_{r\!+\!1}^{2} dt + \| u(t_{n,\beta})\|_{r\!+\!1}^{2} \Big) \Big]. 
			\notag 
		\end{align}
		We use integration by parts and \eqref{eq:DLN-Consistency1} in Lemma \ref{lemma:DLN-consistency}
		\begin{confidential}
			\color{darkblue}
			\begin{align*}
				\big(p_{n,\beta} - p(t_{n,\beta}), \nabla \cdot \widehat{k}_{n}^{-1} \phi_{n,\alpha}^{h} \big) 
				=& - \big(\nabla (p_{n,\beta} - p(t_{n,\beta})), \widehat{k}_{n}^{-1} \phi_{n,\alpha}^{h} \big) \\
				\leq & C \big\| \nabla (p_{n,\beta} - p(t_{n,\beta})) \big\|^{2}
				+ \frac{1}{16} \| \widehat{k}_{n}^{-1} \phi_{n,\alpha}^{h} \|^{2} \\
				\leq & C(\theta) (k_{n} + k_{n-1})^3 \int_{t_{n-1}}^{t_{n+1}} \| \nabla p_{tt} \|^{2} dt 
				+ \frac{1}{16} \| \widehat{k}_{n}^{-1} \phi_{n,\alpha}^{h} \|^{2}.
			\end{align*}
			\normalcolor
		\end{confidential}
		\begin{align}
			\label{eq:error-H1-pressure}
			\big(p_{n,\beta} - p(t_{n,\beta}), \nabla \cdot \widehat{k}_{n}^{-1} \phi_{n,\alpha}^{h} \big)
			\leq& C \big\| \nabla (p_{n,\beta} - p(t_{n,\beta})) \big\|^{2}
			+ \frac{1}{16} \| \widehat{k}_{n}^{-1} \phi_{n,\alpha}^{h} \|^{2} \\
			\leq& C(\theta) k_{\rm{max}}^3 \int_{t_{n-1}}^{t_{n+1}} \| \nabla p_{tt} \|^{2} dt 
			+ \frac{1}{16} \| \widehat{k}_{n}^{-1} \phi_{n,\alpha}^{h} \|^{2}.   \notag 
		\end{align}
		We combine \eqref{eq:error-H1-diff-eta-phi}, \eqref{eq:tau-H1-term1}, \eqref{eq:tau-H1-term2}, 
		\eqref{eq:tau-H1-term5}, \eqref{eq:tau-nonlinear-H1}, \eqref{eq:non-linear-H1-bound2}, 
		\eqref{eq:error-H1-pressure} and sum \eqref{eq:error-H1-1-DLN} over $n$ from $1$ to $N-1$
		\begin{confidential}
			\color{darkblue}
			\begin{align*}
				&\nu \Big(\begin{Vmatrix}
					\nabla {\phi_{n+1}^{h}} \\
					\nabla {\phi_{n}^{h}}
				\end{Vmatrix}
				_{G(\theta)}^{2} -
				\begin{Vmatrix}
					\nabla {\phi_{n}^{h}} \\
					\nabla {\phi_{n-1}^{h}}
				\end{Vmatrix}%
				_{G(\theta)}^{2} 
				+ \Big\| \nabla \big(\sum_{\ell=0}^{2} \gamma_{\ell}^{(n)} \phi_{n\!-\!1\!+\!\ell}^{h} \big) \Big\| ^{2} \Big) 
				+ \frac{\widehat{k}_{n}}{2} \Big\| \widehat{k}_{n}^{-1} \phi_{n,\alpha}^{h} \Big\|^{2} \\
				\leq& C(\theta) \big( \widehat{k}_{n} \| u_{n,\beta} \|_{2}^{2} 
				+ \frac{1}{h \nu} \nu \widehat{k}_{n} \| e_{n,\beta} \|_{1}^{2} \big)
				\big( \| \nabla \phi_{n}^{h} \|^{2} + \| \nabla \phi_{n-1}^{h} \|^{2} \big) 
				+ \frac{C}{\nu} \| |u| \|_{\infty,2}^{2} \nu \widehat{k}_{n} \| e_{n,\beta} \|_{1}^{2} \\
				+& C \big( \frac{h^{2s+1}}{\nu^{3}} \| |p| \|_{\infty,s+1}^{2} 
				+ \frac{h^{2r-1}}{\nu} \| |u| \|_{\infty,r+1}^{2} \big) \nu \widehat{k}_{n}\| e_{n,\beta} \|_{1}^{2}  \\
				+& C(\theta) \| |u| \|_{\infty,2}^{2} \Big[ \frac{h^{2s\!+\!2}}{\nu^{2}} 
				\Big( k_{\rm{max}}^{4} \int_{t_{n-1}}^{t_{n+1}} \| p_{tt} \|_{s\!+\!1}^{2} dt 
				+ (k_{n} + k_{n-1}) \| p(t_{n,\beta})\|_{s\!+\!1}^{2}  \Big) \notag \\
				& \qquad \qquad \qquad \qquad  
				+ h^{2r} \Big( k_{\rm{max}}^{4} \int_{t_{n\!-\!1}}^{t_{n\!+\!1}} \| u_{tt} \|_{r\!+\!1}^{2} dt + (k_{n} + k_{n-1}) \| u(t_{n,\beta})\|_{r\!+\!1}^{2} \Big) \Big] \\
				+& C(\theta) \Big( h^{2r+2} 
				\int_{t_{n-1}}^{t_{n+1}} \| u_{t} \|_{r+1}^{2} dt 
				+ \frac{h^{2s+4}}{\nu^{2}} \int_{t_{n-1}}^{t_{n+1}} \| p_{t} \|_{s+1}^{2} dt \Big)
				+ C(\theta) k_{\rm{max}}^{4} \int_{t_{n-1}}^{t_{n+1}} \| u_{ttt} \|^{2} dt \\
				+& C(\theta) \nu^{2} k_{\rm{max}}^{4} \int_{t_{n-1}}^{t_{n+1}} \| u_{tt} \|_{2}^{2} dt
				+ C(\theta) k_{\rm{max}}^{4} \int_{t_{n-1}}^{t_{n+1}} \| f_{tt} \|^{2} dt
				+ C(\theta) k_{\rm{max}}^{4} \| |u| \|_{\infty,2}^{2} 
				\int_{t_{n-1}}^{t_{n+1}} \| u_{tt} \|_{1}^{2} dt \\
				+& C(\theta) k_{\rm{max}}^{4} \int_{t_{n-1}}^{t_{n+1}} \| \nabla p_{tt} \|^{2} dt.
			\end{align*}
			\begin{align*}
				&\begin{Vmatrix}
					\nabla {\phi_{N}^{h}} \\
					\nabla {\phi_{N-1}^{h}}
				\end{Vmatrix}
				_{G(\theta)}^{2} -
				\begin{Vmatrix}
					\nabla {\phi_{1}^{h}} \\
					\nabla {\phi_{0}^{h}}
				\end{Vmatrix}%
				_{G(\theta)}^{2} 
				+ \sum_{n=1}^{N-1} \Big\| \nabla \big(\sum_{\ell=0}^{2} \gamma_{\ell}^{(n)} \phi_{n\!-\!1\!+\!\ell}^{h} \big) \Big\| ^{2} 
				+ \sum_{n=1}^{N-1} \frac{\widehat{k}_{n}}{2\nu} \Big\| \widehat{k}_{n}^{-1} \phi_{n,\alpha}^{h} \Big\|^{2} \\
				\leq& \frac{C(\theta)}{\nu} \sum_{n=1}^{N-1}\big( \widehat{k}_{n} \| u_{n,\beta} \|_{2}^{2} 
				+ \frac{1}{h \nu} \nu \widehat{k}_{n} \| e_{n,\beta} \|_{1}^{2} \big)
				\big( \| \nabla \phi_{n}^{h} \|^{2} + \| \nabla \phi_{n-1}^{h} \|^{2} \big)
				+ \frac{C}{\nu^{2}} \| |u| \|_{\infty,2}^{2} \sum_{n=1}^{N-1} 
				\nu \widehat{k}_{n} \| e_{n,\beta} \|_{1}^{2} \\
				+& C \big( \frac{h^{2s+1}}{\nu^{4}} \| |p| \|_{\infty,s+1}^{2} 
				+ \frac{h^{2r-1}}{\nu^{2}} \| |u| \|_{\infty,r+1}^{2} \big) 
				\sum_{n=1}^{N-1} \nu \widehat{k}_{n}\| e_{n,\beta} \|_{1}^{2}  \\
				+& C(\theta) \| |u| \|_{\infty,2}^{2} \Big[ \frac{h^{2s\!+\!2}}{\nu^{3}} 
				\Big( k_{\rm{max}}^{4} \sum_{n=1}^{N-1} \int_{t_{n-1}}^{t_{n+1}} \| p_{tt} \|_{s\!+\!1}^{2} dt 
				+ \sum_{n=1}^{N-1} (k_{n} + k_{n-1}) \| p(t_{n,\beta})\|_{s\!+\!1}^{2}  \Big) \notag \\
				& \qquad \qquad \qquad \qquad  
				+ \frac{h^{2r}}{\nu} \Big( k_{\rm{max}}^{4} \sum_{n=1}^{N-1} \int_{t_{n\!-\!1}}^{t_{n\!+\!1}} \| u_{tt} \|_{r\!+\!1}^{2} dt + \sum_{n=1}^{N-1} (k_{n} + k_{n-1}) \| u(t_{n,\beta})\|_{r\!+\!1}^{2} \Big) \Big] \\ 
				+& C(\theta) \Big( \frac{h^{2r+2}}{\nu} \sum_{n=1}^{N-1} 
				\int_{t_{n-1}}^{t_{n+1}} \| u_{t} \|_{r+1}^{2} dt 
				+ \frac{h^{2s+4}}{\nu^{3}} 
				\sum_{n=1}^{N-1}\int_{t_{n-1}}^{t_{n+1}} \| p_{t} \|_{s+1}^{2} dt \Big)
				+ \frac{C(\theta)}{\nu} k_{\rm{max}}^{4} \sum_{n=1}^{N-1} 
				\int_{t_{n-1}}^{t_{n+1}} \| u_{ttt} \|^{2} dt \\
				+& C(\theta) \nu k_{\rm{max}}^{4} \sum_{n=1}^{N-1} 
				\int_{t_{n-1}}^{t_{n+1}} \| u_{tt} \|_{2}^{2} dt
				+ \frac{C(\theta)}{\nu} k_{\rm{max}}^{4} \sum_{n=1}^{N-1} \int_{t_{n-1}}^{t_{n+1}} \| f_{tt} \|^{2} dt
				+ \frac{C(\theta)}{\nu} k_{\rm{max}}^{4} \| |u| \|_{\infty,2}^{2} 
				\sum_{n=1}^{N-1} \int_{t_{n-1}}^{t_{n+1}} \| u_{tt} \|_{1}^{2} dt \\
				+& \frac{C(\theta)}{\nu} k_{\rm{max}}^{4} 
				\sum_{n=1}^{N-1} \int_{t_{n-1}}^{t_{n+1}} \| \nabla p_{tt} \|^{2} dt.
			\end{align*}
			\normalcolor
		\end{confidential}
		\begin{align}
			\label{eq:phi-estimator-H1}
			&\| \nabla \phi_{N}^{h} \|^{2} 
			+ \frac{C(\theta)}{\nu} \sum_{n=1}^{N-1} \widehat{k}_{n} \Big\| \widehat{k}_{n}^{-1} \phi_{n,\alpha}^{h} \Big\|^{2} \\
			\leq & \frac{C(\theta)}{\nu} \sum_{n=1}^{N-1}\big( \widehat{k}_{n} \| u_{n,\beta} \|_{2}^{2} 
			+ \frac{1}{h \nu} \nu \widehat{k}_{n} \| e_{n,\beta} \|_{1}^{2} \big)
			\big( \| \nabla \phi_{n}^{h} \|^{2} + \| \nabla \phi_{n-1}^{h} \|^{2} \big) \notag \\
			+& \frac{C(\theta)}{\nu^{2}} \| |u| \|_{\infty,2}^{2} \sum_{n=1}^{N-1} 
			\nu \widehat{k}_{n} \| e_{n,\beta} \|_{1}^{2}                  \notag \\
			+& C(\theta) \big( \frac{h^{2s+1}}{\nu^{4}} \| |p| \|_{\infty,s+1}^{2} 
			+ \frac{h^{2r-1}}{\nu^{2}} \| |u| \|_{\infty,r+1}^{2} \big) 
			\sum_{n=1}^{N-1} \nu \widehat{k}_{n}\| e_{n,\beta} \|_{1}^{2}  \notag \\
			+& \!C(\!\theta\!) \!\| |u| \|_{\infty\!,\!2}^{2} \Big[\! \frac{h^{2s\!+\!2}}{\nu^{3}} 
			\Big(\! k_{\rm{max}}^{4} \| p_{tt} \|_{2\!,\!s\!+\!1}^{2}  
			\!+\! \| |p| \|_{2\!,\!s\!+\!1\!,\!\beta}^{2} \! \Big) 
			\!+\! \frac{h^{2r}}{\nu} \Big(\! k_{\rm{max}}^{4} \| u_{tt} \|_{2\!,\!r\!+\!1}^{2}
			\!+\! \| |u| \|_{2\!,\!r\!+\!1\!,\!\beta}^{2} \!\Big) \!\Big] \notag \\ 
			+& C(\theta) \Big( \frac{h^{2r+2}}{\nu} \| u_{t} \|_{2,r+1}^{2}
			+ \frac{h^{2s+4}}{\nu^{3}}  \| p_{t} \|_{2,s+1}^{2}  \Big)
			+ \frac{C(\theta)}{\nu} k_{\rm{max}}^{4} \| u_{ttt} \|_{2,0}^{2} \notag \\
			+& C(\theta) \nu k_{\rm{max}}^{4} \| u_{tt} \|_{2,2}^{2} 
			+ \frac{C(\theta)}{\nu} k_{\rm{max}}^{4} \| f_{tt} \|_{2,0}^{2} 
			+ \frac{C(\theta)}{\nu} k_{\rm{max}}^{4} \| |u| \|_{\infty,2}^{2} 
			\| u_{tt} \|_{2,1}^{2} \notag \\
			+& \frac{C(\theta)}{\nu} k_{\rm{max}}^{4} \| \nabla p_{tt} \|_{2,0}^{2}
			+ C(\theta) \big( \| \nabla \phi_{1}^{h} \|^{2} + \| \nabla \phi_{0}^{h} \|^{2} \big). \notag  
		\end{align}
		Since 
		\begin{align*}
			\sum_{n=1}^{N-1} \nu \widehat{k}_{n} \| e_{n,\beta}^{u} \|_{1}^{2}
			=& \nu \sum_{n=1}^{N-1} \widehat{k}_{n} \| e_{n,\beta}^{u} \|^{2}
			+ \nu \sum_{n=1}^{N-1} \widehat{k}_{n} \| \nabla e_{n,\beta}^{u} \|^{2} \\
			\leq& C(\theta) \nu T \max_{0 \leq n \leq M} \| e_{n}^{u} \|^{2} 
			+ \nu \sum_{n=1}^{N-1} \widehat{k}_{n} \| \nabla e_{n,\beta}^{u} \|^{2},
		\end{align*}
		we use \eqref{eq:error-L2-inf-final} in the proof of Theorem \ref{thm:error-velocity-L2} to obtain
		\begin{align}
			\label{eq:error-H1-beta}
			\sum_{n=1}^{N-1} 
			\nu \widehat{k}_{n} \| e_{n,\beta}^{u} \|_{1}^{2}
			\leq& \big( C(\theta) \nu T + 1 \big)
				F_{2}^{2},
		\end{align}
		where
		\begin{align*}
			F_{2} 
			=& \exp \Big[ \frac{C(\theta)}{\nu} \big( k_{\rm{max}}^{4} \| u_{tt} \|_{2,2}^{2} 
			+ \| |u| \|_{2,2,\beta}^{2} \big) \Big] 
			\sqrt{F_{1}} + Ch^{r} \| |u| \|_{\infty,r} 
			\notag \\
			& + C(\theta) \sqrt{\nu} h^{r} \Big( k_{\rm{max}}^{2} \| u_{tt} \|_{2,r+1}
			+ \| |u| \|_{2,r+1,\beta} \Big),
		\end{align*}
		and $F_{1}$ is in \eqref{eq:def_F1}.
		We apply \eqref{eq:error-H1-beta} and discrete Gronwall inequality to \eqref{eq:phi-estimator-H1} 
		\begin{confidential}
			\color{darkblue}
			\begin{align*}
				&\| \nabla \phi_{N}^{h} \|^{2}   
				+ \frac{C(\theta)}{\nu} \sum_{n=1}^{N-1} \widehat{k}_{n} 
				\Big\| \widehat{k}_{n}^{-1} \phi_{n,\alpha}^{h} \Big\|^{2} \\
				\leq & \frac{C(\theta)}{\nu} 
				\Big[ \big( \widehat{k}_{N-1} \| u_{N-1,\beta} \|_{2}^{2} 
				+ \frac{1}{h \nu} \nu \widehat{k}_{N-1} \| e_{N-1,\beta}^{u} \|_{1}^{2}  \big) 
				\| \nabla \phi_{N-1}^{h} \|^{2} \\
				+& \sum_{n=2}^{N-2}  \big(\widehat{k}_{n+1} \| u_{n+1,\beta} \|_{2}^{2} 
				+ \widehat{k}_{n} \| u_{n,\beta} \|_{2}^{2} 
				+ \frac{1}{h \nu} \nu \widehat{k}_{n+1} \| e_{n+1,\beta}^{u} \|_{1}^{2} 
				+ \frac{1}{h \nu} \nu \widehat{k}_{n} \| e_{n,\beta}^{u} \|_{1}^{2} \big) 
				\| \nabla \phi_{n}^{h} \|^{2} \\
				+& \big( \widehat{k}_{1} \| u_{1,\beta} \|_{2}^{2} 
				+ \frac{1}{h \nu} \nu \widehat{k}_{1} \| e_{1,\beta}^{u} \|_{1}^{2}  \big) 
				\| \nabla \phi_{0}^{h} \|^{2} \Big] 
				+ \frac{C(\theta)(1 + \nu T)}{\nu^{2}} \| |u| \|_{\infty,2}^{2} 
				F_{2}^{2} 	\\
				+& C(\theta)(1 + \nu T) \big( \frac{h^{2s+1}}{\nu^{4}} \| |p| \|_{\infty,s+1}^{2} 
				+ \frac{h^{2r-1}}{\nu^{2}} \| |u| \|_{\infty,r+1}^{2} \big) 
				F_{2}^{2} \\
				+& \!C(\!\theta\!) \!\| |u| \|_{\infty\!,\!2}^{2} \Big[\! \frac{h^{2s\!+\!2}}{\nu^{3}} 
				\Big(\! k_{\rm{max}}^{4} \| p_{tt} \|_{2\!,\!s\!+\!1}^{2}  
				\!+\! \| |p| \|_{2\!,\!s\!+\!1\!,\!\beta}^{2} \! \Big) 
				\!+\! \frac{h^{2r}}{\nu} \Big(\! k_{\rm{max}}^{4} \| u_{tt} \|_{2\!,\!r\!+\!1}^{2}
				\!+\! \| |u| \|_{2\!,\!r\!+\!1\!,\!\beta}^{2} \!\Big) \!\Big] \notag \\ 
				+& C(\theta) \Big( \frac{h^{2r+2}}{\nu} \| u_{t} \|_{2,r+1}^{2}
				+ \frac{h^{2s+4}}{\nu^{3}}  \| p_{t} \|_{2,s+1}^{2}  \Big)
				+ \frac{C(\theta)}{\nu} k_{\rm{max}}^{4} \| u_{ttt} \|_{2,0}^{2} \notag \\
				+& C(\theta) \nu k_{\rm{max}}^{4} \| u_{tt} \|_{2,2}^{2} 
				+ \frac{C(\theta)}{\nu} k_{\rm{max}}^{4} \| f_{tt} \|_{2,0}^{2} 
				+ \frac{C(\theta)}{\nu} k_{\rm{max}}^{4} \| |u| \|_{\infty,2}^{2} 
				\| u_{tt} \|_{2,1}^{2} \notag \\
				+& \frac{C(\theta)}{\nu} k_{\rm{max}}^{4} \| \nabla p_{tt} \|_{2,0}^{2} 
				+ C(\theta) \big( \| \nabla \phi_{1}^{h} \|^{2} + \| \nabla \phi_{0}^{h} \|^{2} \big).
			\end{align*}
			\normalcolor
		\end{confidential}
		\begin{align}
			\label{eq:phi-H1-diff-L2}
			&\| \nabla \phi_{N}^{h} \|^{2}  
			+ \frac{C(\theta)}{\nu} \sum_{n=1}^{N-1} \widehat{k}_{n} 
			\Big\| \widehat{k}_{n}^{-1} \phi_{n,\alpha}^{h} \Big\|^{2}        \\
			\leq&  \exp \Big[ \frac{C(\theta)}{\nu}  
			\big( k_{\rm{max}}^{4} \| u_{tt} \|_{2,2}^{2} 
			+ \| |u|\|_{2,2,\beta}^{2} + \frac{ C(\theta) \nu T + 1 }{h \nu} 
			F_{2}^{2} 	\big)   \Big] F_{3},               \notag 
		\end{align}
		where
		\begin{align*}
			F_{3} 
			&= \frac{C(\theta)(1 \!+\! \nu T)}{\nu^{2}} \| |u| \|_{\infty,2}^{2} 
			F_{2}^{2} 	
			\!+\! C(\theta)(1 \!+\! \nu T) \big( \frac{h^{2s\!+\!1}}{\nu^{4}} \| |p| \|_{\infty,s\!+\!1}^{2} 
			+ \frac{h^{2r\!-\!1}}{\nu^{2}} \| |u| \|_{\infty,r\!+\!1}^{2} \big) 
			F_{2}^{2} \\
			+& \!C(\!\theta\!) \!\| |u| \|_{\infty\!,\!2}^{2} \Big[\! \frac{h^{2s\!+\!2}}{\nu^{3}} 
			\Big(\! k_{\rm{max}}^{4} \| p_{tt} \|_{2\!,\!s\!+\!1}^{2}  
			\!+\! \| |p| \|_{2\!,\!s\!+\!1\!,\!\beta}^{2} \! \Big) 
			\!+\! \frac{h^{2r}}{\nu} \Big(\! k_{\rm{max}}^{4} \| u_{tt} \|_{2\!,\!r\!+\!1}^{2}
			\!+\! \| |u| \|_{2\!,\!r\!+\!1\!,\!\beta}^{2} \!\Big) \!\Big] \notag \\ 
			+& C(\theta) \Big( \frac{h^{2r+2}}{\nu} \| u_{t} \|_{2,r+1}^{2}
			+ \frac{h^{2s+4}}{\nu^{3}}  \| p_{t} \|_{2,s+1}^{2}  \Big)
			+ \frac{C(\theta)}{\nu} k_{\rm{max}}^{4} \| u_{ttt} \|_{2,0}^{2} \notag \\
			+& C(\theta) \nu k_{\rm{max}}^{4} \| u_{tt} \|_{2,2}^{2} 
			+ \frac{C(\theta)}{\nu} k_{\rm{max}}^{4} \| f_{tt} \|_{2,0}^{2} 
			+ \frac{C(\theta)}{\nu} k_{\rm{max}}^{4} \| |u| \|_{\infty,2}^{2} 
			\| u_{tt} \|_{2,1}^{2} \notag \\
			+& \frac{C(\theta)}{\nu} k_{\rm{max}}^{4} \| \nabla p_{tt} \|_{2,0}^{2}
			+ C(\theta) \big( \| \nabla \phi_{1}^{h} \|^{2} + \| \nabla \phi_{0}^{h} \|^{2} \big).
		\end{align*}
		By the time-diameter condition in \eqref{eq:time-h-limit}, $h^{-1} F_{2}^{2}$ is bounded.
		Thus 
		\begin{align}
			\label{eq:error-grad-velocity}
			 \max_{0 \leq n \leq N} \| \nabla e_{n}^{u} \|  
			\leq& \max_{0 \leq n \leq N} \| \nabla \eta_{n} \| 
			+ \max_{0 \leq n \leq N} \| \nabla \phi_{n}^{h} \|   \\
			\leq& Ch^{r} \| |u| \|_{\infty,r+1} \!+\! \frac{Ch^{s+1}}{\nu} \| |p| \|_{\infty,s+1}  \notag \\
			+& \exp \Big[ \frac{C(\theta)}{\nu}  \big( k_{\rm{max}}^{4} \| u_{tt} \|_{2,2}^{2} 
			+ \| |u|\|_{2,2,\beta}^{2} + \frac{ C(\theta) \nu T + 1 }{h \nu} F_{2}^{2} 	\big)   \Big] \sqrt{F_{3}}.   \notag 
		\end{align}
		Combining Theorem \ref{thm:error-velocity-L2} and \eqref{eq:error-grad-velocity}, 
		we have \eqref{eq:error-H1-conclusion}. 
		By \eqref{eq:approx-thm}, \eqref{eq:Stoke-Approx} and H$\ddot{\rm{o}}$lder's inequality 
		\begin{confidential}
			\color{darkblue}
			\begin{align*}
				&\sum_{n=1}^{N-1} \frac{\widehat{k}_{n}}{\nu} 
				\Big\| \widehat{k}_{n}^{-1} \eta_{n,\alpha} \Big\|^{2}
				= \sum_{n=1}^{N-1} \frac{1}{\nu \widehat{k}_{n}} \| \eta_{n,\alpha} \|^{2} 
				\leq \sum_{n=1}^{N-1} \frac{C}{\nu \widehat{k}_{n}} 
				\Big( \frac{h^{2s+2}}{\nu^{2}} \| p_{n,\alpha} \|_{s+1}^{2} 
				+ h^{2r} \| u_{n,\alpha}\|_{r}^{2}  \Big) \\
				\leq& \sum_{n=1}^{N-1} \frac{C}{\nu \widehat{k}_{n}}
				\Big[ \frac{h^{2s+2}}{\nu^{2}} C(\theta)(k_{n} + k_{n-1}) \int_{t_{n-1}}^{t_{n+1}} \| p_{t} \|_{s+1}^{2} dt 
				+ h^{2r} C(\theta)(k_{n} + k_{n-1}) \int_{t_{n-1}}^{t_{n+1}} \| u_{t} \|_{r}^{2} dt  \Big] \\
				\leq& \sum_{n=1}^{N-1} \frac{C(\theta)}{\nu} \Big[ \frac{h^{2s+2}}{\nu^{2}} 
				\int_{t_{n-1}}^{t_{n+1}} \| p_{t} \|_{s+1}^{2} dt 
				+ h^{2r} \int_{t_{n-1}}^{t_{n+1}} \| u_{t} \|_{r}^{2} dt  \Big]
				= \frac{C(\theta)}{\nu} 
				\Big( \frac{h^{2s+2}}{\nu^{2}} \| p_{t} \|_{2,s+1}^{2} 
				+ h^{2r} \| u_{t} \|_{2,r}^{2}   \Big)
			\end{align*}
			\begin{align*}
				\sum_{n=1}^{N-1} \frac{\widehat{k}_{n}}{\nu} 
				\Big\| \widehat{k}_{n}^{-1} e_{n,\alpha} \Big\|^{2} 
				\leq& 2 \sum_{n=1}^{N-1} \frac{\widehat{k}_{n}}{\nu} 
				\Big\| \widehat{k}_{n}^{-1} \eta_{n,\alpha} \Big\|^{2}
				+ 2 \sum_{n=1}^{N-1} \frac{\widehat{k}_{n}}{\nu} 
				\Big\| \widehat{k}_{n}^{-1} \phi_{n,\alpha}^{h} \Big\|^{2} \\
				\leq& \frac{C(\theta)}{\nu} \Big( \frac{h^{2s+2}}{\nu^{2}} \| p_{t} \|_{2,s+1}^{2} 
				+ h^{2r} \| u_{t} \|_{2,r}^{2}  \Big) \\
				+& \exp \Big[ \frac{C(\theta)}{\nu} \! 
				\big( k_{\rm{max}}^{4} \| u_{tt} \|_{2,2}^{2} 
				\!+\! \| |u|\|_{2,2,\beta}^{2} \!+\! \frac{ C(\!\theta\!) \nu T \!+\! 1 }{h \nu} 
				F_{2}^{2} \big)  \Big] \cdot F_{3}
			\end{align*}
			\normalcolor
		\end{confidential}
		\begin{align}
			\label{eq:diff-eta-L2}
			\sum_{n=1}^{N-1} \frac{\widehat{k}_{n}}{\nu} 
			\Big\| \widehat{k}_{n}^{-1} \eta_{n,\alpha} \Big\|^{2}
			\leq& \sum_{n=1}^{N-1} \frac{C}{\nu \widehat{k}_{n}} 
			\Big( \frac{h^{2s+2}}{\nu^{2}} \Big\| \sum_{\ell=0}^{2} \alpha_{\ell} p_{n-1+\ell} \Big\|_{s+1}^{2} 
			+ h^{2r} \Big\| \sum_{\ell=0}^{2} \alpha_{\ell} u_{n-1+\ell} \Big\|_{r}^{2} \Big) \notag \\
			\leq& \sum_{n=1}^{N-1} \frac{C(\theta)}{\nu} \Big[ \frac{h^{2s+2}}{\nu^{2}} 
			\int_{t_{n-1}}^{t_{n+1}} \| p_{t} \|_{s+1}^{2} dt 
			+ h^{2r} \int_{t_{n-1}}^{t_{n+1}} \| u_{t} \|_{r}^{2} dt  \Big] \notag \\
			=& \frac{C(\theta)}{\nu} \Big( \frac{h^{2s+2}}{\nu^{2}} \| p_{t} \|_{2,s+1}^{2} 
			+ h^{2r} \| u_{t} \|_{2,r}^{2} \Big).      
		\end{align}
		By \eqref{eq:phi-H1-diff-L2} and \eqref{eq:diff-eta-L2} and triangle inequality
		\begin{align}
			\label{eq:error-diff-time-H1}
			\sum_{n=1}^{N-1} \frac{\widehat{k}_{n}}{\nu} 
			\Big\| \widehat{k}_{n}^{-1} e_{n,\alpha} \Big\|^{2} 
			\leq&\! \frac{C(\theta)}{\nu} \!\Big(\! \frac{h^{2s\!+\!2}}{\nu^{2}} 
			\| p_{t} \|_{2,s\!+\!1}^{2} 
			\!+\! h^{2r} \| u_{t} \|_{2,r}^{2}  \!\Big)  \\
			+& \exp \Big[\! \frac{C(\theta)}{\nu} \! 
			\big(\! k_{\rm{max}}^{4} \| u_{tt} \|_{2,2}^{2} 
			\!+\! \| |u|\|_{2,2,\beta}^{2} \!+\! \frac{ C(\!\theta\!) \nu T \!+\! 1 }{h \nu} 
			F_{2}^{2} 
			\!\big)  \! \Big]  F_{3},      \notag
		\end{align}
		which implies \eqref{eq:error-diff-L2-conclusion}.
	\end{proof}

	\begin{theorem}
		Suppose the velocity $u \in X$ and pressure $p \in Q$ of the NSE in \eqref{eq:NSE} satisfy
		\begin{gather*}
			u \in \ell^{\infty}(0,N;H^{r+1}) \cap \ell^{\infty}(0,N;H^2) \cap \ell^{\infty,\beta}(0,N;H^{1})
			\cap \ell^{2,\beta}(0,N;H^{r+1} \cap H^{2}), \\
			u_{t} \in L^{2}(0,T;H^{r+1}),  \ \ 
			u_{tt} \in L^{2}(0,T;H^{r+1} \cap H^{2}), \ \ 
			u_{ttt} \in L^{2}(0,T;X' \cap L^{2}), \\
			p \!\in\! \ell^{\infty}(0,N;H^{s\!+\!1}) \!\cap\! \ell^{2,\beta}(0,N;H^{\!s+\!1}), \
			p_{t} \!\in\! L^{2}(0,T;H^{s\!+\!1}), \ p_{tt} \!\in\! L^{2}(0,T;H^{\!s+\!1} \!\cap\! H^{1}),
		\end{gather*} 
		and the 
		body force $f \in L^{2}(0,T;X'\cap L^{2})$, then under the time step bounds in \eqref{eq:ratio-limit} and the time-diameter condition in \eqref{eq:time-h-limit},
		the pressure component by the algorithm in \eqref{eq:DLN-Semi-Alg} satisfy 
		\begin{align}
			\Big( \sum_{n=1}^{N-1} \widehat{k}_{n} \| p_{n,\beta} - p_{n,\beta}^{h} \|^{2} \Big)^{1/2}
			\leq \mathcal{O} \big( h^{r},h^{s+1}, k_{\rm{max}}^{2} \big). \label{eq:pressure-L2-conclusion} 
		\end{align}
		Moreover for constant time-stepping DLN algorithm with parameter $\theta \in (0,1)$ and constant time step $k$, we have 
		\begin{align}
			\sum_{n=0}^{N} k \| p_{n} - p_{n}^{h} \| \leq \mathcal{O} \big( h^{r},h^{s+1}, k^{2} \big). 
			\label{eq:pressure-L1-conclusion} 
		\end{align}
	\end{theorem}
	\begin{proof}
		Let $(P_{S}^{(u)} u_{n}, P_{S}^{(p)} p_{n})$ be Stokes projection of $(u_n,p_n)$ 
		onto $V^{h} \times Q^{h}$. We set
		\begin{gather*}
			\phi_{n}^{h} = u_{n}^{h} - P_{S}^{(u)} u_{n}, \ \ \ \eta_{n} = u_{n} - P_{S}^{(u)} u_{n}, \ \ \ 
			e_{n}^{u} = \phi_{n}^{h} - \eta_{{n}} \\
			e_{n,\alpha}^{u} = \sum_{\ell=0}^{2}{\alpha_{\ell}} e_{n-1+\ell}^{u}, \ \ \ 
			\widetilde{e}_{n}^{u} = \widetilde{u}_{n}^{h} - \widetilde{u}_{n}.
		\end{gather*}
		We let $v^{h} \in X^{h}$ in \eqref{eq:NSE-exact} and subtract \eqref{eq:NSE-exact} 
		from the first equation of \eqref{eq:DLN-Semi-Alg}
		\begin{confidential}
			\color{darkblue}
			\begin{gather*}
			\frac{1}{\widehat{k}_{n}} \Big( \sum_{\ell=0}^{2}{\alpha_{\ell}} u_{n-1+\ell}^{h}, v^{h} \Big)
			- \frac{1}{\widehat{k}_{n}} \Big( \sum_{\ell=0}^{2}{\alpha_{\ell}} u_{n-1+\ell}, v^{h} \Big)
			+ b \big( \widetilde{u}_{n}^{h}, u_{n,\beta}^{h}, v^{h} \big)
			- b \big( \widetilde{u}_{n}, u_{n,\beta}, v^{h} \big) \\
			+ \nu \big( \nabla u_{n,\beta}^{h}, \nabla v^{h} \big) 
			- \nu \big( \nabla u_{n,\beta}, \nabla v^{h} \big) 
			=  (p_{n,\beta}^{h} - p_{n,\beta} + p_{n,\beta} -p(t_{n,\beta}), \nabla \cdot v^{h}) - \tau_{n}(v^{h})
			\end{gather*}
			\begin{align*}
			& \big( \widehat{k}_{n}^{-1} e_{n,\alpha}^{u}, v^{h} \big) 
			+ \nu \big( \nabla e_{n,\beta}^{u}, \nabla v^{h} \big) 
			+ b \big( \widetilde{u}_{n}^{h}, u_{n,\beta}^{h}, v^{h} \big)
			- b \big( \widetilde{u}_{n}, u_{n,\beta}, v^{h} \big) \\
			=& - (p(t_{n,\beta}) - p_{n,\beta} + p_{n,\beta} -q_{n}^{h} + q_{n}^{h} - p_{n,\beta}^{h} , \nabla \cdot v^{h}) - \tau_{n}(v^{h}),
			\end{align*}
			\normalcolor
		\end{confidential}
		\begin{align}
			\label{eq:Diff-exact-DLN-pressure}
			( p_{n,\beta}^{h} - q_{n}^{h}, \nabla \cdot v^{h} )
			=& \big( \widehat{k}_{n}^{-1} e_{n,\alpha}^{u} , v^{h} \big) 
			+ \nu \big( \nabla e_{n,\beta}^{u}, \nabla v^{h} \big)  
			+ b \big( \widetilde{u}_{n}^{h}, u_{n,\beta}^{h}, v^{h} \big)
			- b \big( \widetilde{u}_{n}, u_{n,\beta}, v^{h} \big) \notag \\  
			&+ (p(t_{n,\beta}) - p_{n,\beta}, \nabla \cdot v^{h}) 
			+ ( p_{n,\beta} - q_{n}^{h}, \nabla \cdot v^{h} ) + \tau_{n}(v^{h}),
		\end{align}
		where $q^{h}$ is the $L^{2}$ projection of $p_{n,\beta}$ onto $Q^{h}$.
		By \eqref{eq:b-bound1} and Poincar$\acute{\rm{e}}$ inequality
		\begin{align}
			\label{eq:non-linear-bound-pressure}
			& b \big( \widetilde{u}_{n}^{h}, u_{n,\beta}^{h}, v^{h} \big)
			- b \big( \widetilde{u}_{n}, u_{n,\beta}, v^{h} \big)
			\\
			= & b \big( \widetilde{e}_{n}^{u}, u_{n,\beta}, v^{h} \big)
			 + b \big( \widetilde{u}_{n}, e_{n,\beta}^{u}, v^{h} \big)
			 + b \big( \widetilde{e}_{n}^{u}, e_{n,\beta}^{u}, v^{h} \big) \notag \\
			\leq& C \big( \| \widetilde{e}_{n}^{u} \|_{1} \| \nabla u_{n,\beta} \|\!
			+\! \| \widetilde{u}_{n} \|_{1} \| \nabla e_{n,\beta}^{u} \| 
			\!+\! \| \widetilde{e}_{n}^{u} \|_{1} \| \nabla e_{n,\beta}^{u} \| \big)
			\| \nabla v^{h} \|.   \notag 
		\end{align}
		By Cauchy-Schwarz inequality, \eqref{eq:DLN-Consistency1} and \eqref{eq:DLN-Consistency2} 
		in Lemma \ref{lemma:DLN-consistency}
		\begin{align}
			\label{eq:tau-term1-pressure}
			\Big( \frac{\sum_{\ell=0}^{2}{\alpha_{\ell}} u_{n-1+\ell}}{\widehat{k}_{n}}  
			- u(t_{n,\beta}), v^{h} \Big) 
			\leq& \Big\| \frac{\sum_{\ell=0}^{2}{\alpha_{\ell}} u_{n-1+\ell}}{\widehat{k}_{n}} 
			- u(t_{n,\beta}) \Big\|_{-1} \| \nabla v^{h} \|  \\
			\leq& C(\theta) \Big(  k_{\rm{max}}^{3} \int_{t_{n-1}}^{t_{n+1}} \| u_{ttt} \|_{-1}^{2} dt \Big)^{1/2} \| \nabla v^{h} \|,  \notag 
		\end{align}
		\begin{align}
			\label{eq:tau-term2-pressure}
			\nu \big( \nabla (u_{n,\beta} - u(t_{n,\beta})), \nabla v^{h} \big) 
			\!\leq& \! \nu \big\| \nabla (u_{n,\beta} - u(t_{n,\beta})) \big\| \| \nabla v^{h} \|  \\
			\!\leq& \! C(\theta)\nu \Big( k_{\rm{max}}^{3} 
			\int_{t_{n\!-\!1}}^{t_{n\!+\!1}} \| \nabla u_{tt} \| dt \Big)^{1/2}
			\| \nabla v^{h} \|, \notag 
		\end{align}
		\begin{align}
			\label{eq:tau-term3-pressure}
			(f(t_{n,\beta}) \!-\! f_{n,\beta},v^{h}) 
			\!\leq \! \| f(t_{n,\beta}) \!-\! f_{n,\beta} \|_{-1} \| \nabla v^{h} \|   
			\!\leq \! C(\theta) \Big(\! k_{\rm{max}}^{3} 
			\int_{t_{n\!-\!1}}^{t_{n\!+\!1}} \| f_{tt} \|_{-1} dt \! \Big)^{1/2} \| \nabla v^{h} \|
		\end{align}
		By \eqref{eq:b-bound1}, \eqref{eq:DLN-Consistency1} in Lemma \ref{lemma:DLN-consistency} and 
		Poincar$\acute{\rm{e}}$ inequality
		\begin{align}
			\label{eq:tau-term4-pressure}
			&b \big( \widetilde{u}_{n}, u_{n,\beta}, v^{h} \big)
			- b ( u(t_{n,\beta}), u(t_{n,\beta}), v^{h} )  \\
			=& b \big( \widetilde{u}_{n} - u(t_{n,\beta}), u_{n,\beta}, v^{h} \big)
			+ b \big( u(t_{n,\beta}), u_{n,\beta} - u(t_{n,\beta}), v^{h} \big) \notag \\
			\leq& C \| \nabla \big( \widetilde{u}_{n} - u(t_{n,\beta}) \big) \|
			\| \nabla u_{n,\beta} \| \| \nabla v^{h} \| 
			+ C \| \nabla u(t_{n,\beta}) \| \| \nabla \big( u_{n,\beta} - u(t_{n,\beta}) \big) \|
			\| \nabla v^{h} \|    \notag \\
			\leq& C(\theta) \big( \| |u| \|_{\infty,1} 
			+ \| |u| \|_{\infty,1,\beta} \big) 
			\Big( k_{\rm{max}}^{3} 
			\int_{t_{n-1}}^{t_{n+1}} \| \nabla u_{tt} \| dt \Big)^{1/2} \| \nabla v^{h} \|. \notag 
		\end{align}
		We combine \eqref{eq:tau-term1-pressure}, \eqref{eq:tau-term2-pressure}, \eqref{eq:tau-term3-pressure}
		and \eqref{eq:tau-term4-pressure}
		\begin{align}
			\label{eq:tau-bound-pressure}
			\tau_{n}(v^{h}) 
			\leq& C(\theta) k_{\rm{max}}^{3/2} \Big[ \big(\int_{t_{n-1}}^{t_{n+1}} \| u_{ttt} \|_{-1}^{2} dt \big)^{\frac{1}{2}} 
			+ \big( \int_{t_{n\!-\!1}}^{t_{n\!+\!1}} \| \nabla u_{tt} \| dt \big)^{\frac{1}{2}}
			+ \big( \int_{t_{n\!-\!1}}^{t_{n\!+\!1}} \| f_{tt} \|_{-1} dt \big)^{\frac{1}{2}} 
			\notag \\
			&\qquad \qquad \quad + \big( \| |u| \|_{\infty,1} + \| |u| \|_{\infty,1,\beta} \big) 
			\big( \int_{t_{n\!-\!1}}^{t_{n\!+\!1}} \| \nabla u_{tt} \| dt \big)^{1/2} \Big] 
			\| \nabla v^{h} \|.      
		\end{align}
		By \eqref{eq:approx-thm} and \eqref{eq:DLN-Consistency1}
		\begin{align}
			\label{eq:pressure-bound-pressure}
			(p(t_{n,\beta}) \!-\! p_{n,\beta}, \nabla \cdot v^{h})
			\!\leq&\! \sqrt{d} \| p(t_{n,\beta}) \!-\! p_{n,\beta} \| \| \nabla v^{h} \|
			\!\leq\! C(\theta) \Big( k_{\rm{max}}^{3} 
			\int_{t_{n\!-\!1}}^{t_{n\!+\!1}} \| p_{tt} \| dt \Big)^{\frac{1}{2}} \| \nabla v^{h} \|, \notag \\
			( p_{n,\beta} - q_{n}^{h}, \nabla \cdot v^{h} )
			\leq& \sqrt{d} \| p_{n,\beta} - q^{h} \| \| \nabla v^{h} \|
			\leq \frac{Ch^{s+1}}{\nu} \| p_{n,\beta} \|_{s+1} \| \nabla v^{h} \|.
		\end{align}
		By \eqref{eq:inf-sup},\eqref{eq:non-linear-bound-pressure}, \eqref{eq:tau-bound-pressure} 
		and \eqref{eq:pressure-bound-pressure} 
		\begin{align}
			\label{eq:error-L2-pressure}
			\| p_{n,\beta}^{h} - q_{n}^{h} \| 
			\leq & \| \widehat{k}_{n}^{-1} e_{n,\alpha}^{u} \| \!+\! \nu \| \nabla e_{n,\beta}^{u} \|
			\!+\! C(\theta) \| | u| \|_{\infty,1} \| |e_{n}^{u}| \|_{\infty,1}
			\!+\! C(\theta) \| |e_{n}^{u}| \|_{\infty,1}^{2} \\
			+ C&(\theta) k_{\rm{max}}^{3/2} \! \Big[\! \big( \!\int_{t_{n\!-\!1}}^{t_{n\!+\!1}} \| u_{ttt} \|_{-1}^{2} dt \!\big)^{\frac{1}{2}} 
			\!+\! \big(\! \int_{t_{n\!-\!1}}^{t_{n\!+\!1}} \| \nabla u_{tt} \|^{2} dt \!\big)^{\frac{1}{2}}
			\!+\! \big(\! \int_{t_{n\!-\!1}}^{t_{n\!+\!1}} \| f_{tt} \|_{-1}^{2} dt \!\big)^{\frac{1}{2}} 
			\notag \\
			&\qquad \qquad + \big( \| |u| \|_{\infty,1} + \| |u| \|_{\infty,1,\beta} \big) 
			\big( \int_{t_{n\!-\!1}}^{t_{n\!+\!1}} \| \nabla u_{tt} \|^{2} dt \big)^{\frac{1}{2}} \Big] \notag \\
			+ C&(\theta) \Big( k_{\rm{max}}^{3} 
			\int_{t_{n\!-\!1}}^{t_{n\!+\!1}} \| p_{tt} \| dt \Big)^{\frac{1}{2}}
			+ \frac{C h^{s+1}}{\nu} \| p_{n,\beta} \|_{s+1}.   \notag 
		\end{align}

		\begin{confidential}
			\color{darkblue}
			\begin{align*}
				&\frac{|\big(p_{n,\beta}^{h} \!-\! q_{n}^{h}, \nabla \cdot v^{h} \big)|}{\| \nabla v^{h} \|} \\
				\leq& \| \widehat{k}_{n}^{-1} e_{n,\alpha}^{u} \| + \nu \| \nabla e_{n,\beta}^{u} \|
				+ C(\theta) \| | u| \|_{\infty,1} \| |e_{n}^{u}| \|_{\infty,1}
				+ C(\theta) \| |e_{n}^{u}| \|_{\infty,1}^{2} \\
				&+ C(\theta) k_{\rm{max}}^{3/2} \Big[ \big(\int_{t_{n-1}}^{t_{n+1}} \| u_{ttt} \|_{-1}^{2} dt \big)^{\frac{1}{2}} 
				+ \big( \int_{t_{n\!-\!1}}^{t_{n\!+\!1}} \| \nabla u_{tt} \|^{2} dt \big)^{\frac{1}{2}}
				+ \big( \int_{t_{n\!-\!1}}^{t_{n\!+\!1}} \| f_{tt} \|_{-1}^{2} dt \big)^{\frac{1}{2}} 
				\notag \\
				&+ \big( \| |u| \|_{\infty,1} + \| |u| \|_{\infty,1,\beta} \big) 
				\big( \int_{t_{n\!-\!1}}^{t_{n\!+\!1}} \| \nabla u_{tt} \|^{2} dt \big)^{1/2} \Big]  
				+ C(\theta) \Big( k_{\rm{max}}^{3} 
				\int_{t_{n\!-\!1}}^{t_{n\!+\!1}} \| p_{tt} \| dt \Big)^{\frac{1}{2}}
				+ \frac{C h^{s+1}}{\nu} \| p_{n,\beta} \|_{s+1}. 
			\end{align*}
			\normalcolor
		\end{confidential}
		By triangle inequality, \eqref{eq:approx-thm} and \eqref{eq:error-L2-pressure}
		\begin{align}
			\label{eq:error-L2-L2-pressure}
			&\sum_{n=1}^{N-1} \widehat{k}_{n} \| p_{n,\beta} - p_{n,\beta}^{h} \|^{2} \\
			\leq& 2 \sum_{n=1}^{N-1} \widehat{k}_{n} \| p_{n,\beta}^{h} - q_{n}^{h} \|^{2}
			+ 2 \sum_{n=1}^{N-1} \widehat{k}_{n} \| p_{n,\beta} - q_{n}^{h} \|^{2} \notag \\
			\leq& C h^{2s+2}  \sum_{n=1}^{N-1} \widehat{k}_{n} \| p_{n,\beta} \|_{s+1}^{2}
			+ 2 \sum_{n=1}^{N-1} \widehat{k}_{n} \| \widehat{k}_{n}^{-1} e_{n,\alpha}^{u} \|^{2}
			+ 2 \sum_{n=1}^{N-1} \nu \widehat{k}_{n} \| \nabla e_{n,\beta}^{u} \|^{2}   \notag \\
			+& C(\theta) T \| | u| \|_{\infty,1}^{2} \| |e_{n}^{u}| \|_{\infty,1}^{2}
			+ C(\theta) T \| |e_{n}^{u}| \|_{\infty,1}^{2} \notag \\
			+& C(\theta) k_{\rm{max}}^{4}  \Big[ \| u_{ttt} \|_{2,-1}^{2}  
			+ \| \nabla u_{tt} \|_{2,0}^{2}
			+ \| f_{tt} \|_{2,-1}^{2}  + \big( \| |u| \|_{\infty,1}^{2} + \| |u| \|_{\infty,1,\beta}^{2} \big) 
			\| \nabla u_{tt} \|_{2,0}^{2}  \Big] \notag \\
			+& C(\theta) k_{\rm{max}}^{4} \| p_{tt} \|_{2,0}^{2} 
			+ \frac{Ch^{2s+2}}{\nu^{2}} \sum_{n=1}^{N-1} \widehat{k}_{n} \| p_{n,\beta} \|_{s+1}^{2}. \notag 
		\end{align}
		By \eqref{eq:DLN-Consistency1} in Lemma \ref{lemma:DLN-consistency},
		\eqref{eq:error-L2-inf-final}, \eqref{eq:error-grad-velocity}, \eqref{eq:error-diff-time-H1} 
		and \eqref{eq:error-L2-L2-pressure}
		\begin{confidential}
			\color{darkblue}
			\begin{align*}
				&\sum_{n=1}^{N-1} \widehat{k}_{n} \| p_{n,\beta} - p_{n,\beta}^{h} \|^{2} \\
				\leq& C(\theta) h^{2s+2} k_{\rm{max}}^{4}  \sum_{n=1}^{N-1} \int_{t_{n-1}}^{t_{n+1}} \| p_{tt} \|_{s+1}^{2} dt
				+ C(\theta) h^{2s+2} \sum_{n=1}^{N-1} (k_{n}+k_{n-1}) \| p(t_{n,\beta}) \|_{s+1}^{2} \\
				+& \! C(\theta) \!\Big(\! \frac{h^{2s\!+\!2}}{\nu^{2}} 
				\| p_{t} \|_{2,s\!+\!1}^{2} \!+\! h^{2r} \| u_{t} \|_{2,r}^{2}  \!\Big)  
				+ \nu \exp \Big[\! \frac{C(\theta)}{\nu} \! \big(\! k_{\rm{max}}^{4} \| u_{tt} \|_{2,2}^{2} 
				\!+\! \| |u|\|_{2,2,\beta}^{2} \!+\! \frac{ C(\!\theta\!) \nu T \!+\! 1 }{h \nu} 
				F_{2}^{2} \!\big)  \! \Big]  F_{3}          \\
				+&   
				\exp \Big[ \frac{C(\theta)}{\nu} \big( k_{\rm{max}}^{4} \| u_{tt} \|_{2,2}^{2} 
				+ \| |u| \|_{2,2,\beta}^{2} \big) \Big] F_{1} +\! Ch^{2r} \!\| |u| \|_{\infty,r}^{2}
				\!+\! C(\!\theta\!) \! \nu h^{2r} \!\big( \!k_{\rm{max}}^{4} \!\| u_{tt} \|_{2,r\!+\!1}^{2}
				\!+\! \| |u| \|_{2,r\!+\!1,\beta}^{2} \!\big)  \\
				+& C(\theta) T ( \| | u| \|_{\infty,1}^{2} + 1 ) \Big\{ Ch^{2r} \| |u| \|_{\infty,r+1}^{2} 
				\!+\! \frac{Ch^{2s+2}}{\nu^{2}} \| |p| \|_{\infty,s+1}^{2}  
				+ \exp \Big[ \frac{C(\theta)}{\nu}  \big( k_{\rm{max}}^{4} \| u_{tt} \|_{2,2}^{2} 
				+ \| |u|\|_{2,2,\beta}^{2} + \frac{ C(\theta) \nu T + 1 }{h \nu} F_{2}^{2} 	\big)   \Big] F_{3} \Big\} \\
				+& C(\theta) k_{\rm{max}}^{4}  \Big[ \| u_{ttt} \|_{2,-1}^{2}  
				+ \| \nabla u_{tt} \|_{2,0}^{2}
				+ \| f_{tt} \|_{2,-1}^{2}  + \big( \| |u| \|_{\infty,1} + \| |u| \|_{\infty,1,\beta} \big) 
				\| \nabla u_{tt} \|_{2,0}^{2}  \Big] \notag \\
				+& C(\theta) k_{\rm{max}}^{4} \| p_{tt} \|_{2,0}^{2} 
				+ \frac{Ch^{2s+2}}{\nu^{2}} \sum_{n=1}^{M-1} \widehat{k}_{n} \| p_{n,\beta} \|_{s+1}^{2}.
			\end{align*}
			\normalcolor
		\end{confidential}
		\begin{align}
			\label{eq:error-L2-L2-pressure-final}
			&\sum_{n=1}^{N-1} \widehat{k}_{n} \| p_{n,\beta} - p_{n,\beta}^{h} \|^{2} \\
			\leq& C(\theta) h^{2s\!+\!2} k_{\rm{max}}^{4} \| p_{tt} \|_{2,s\!+\!1}^{2} 
			+ C(\theta) h^{2s\!+\!2} \| |p| \|_{2,s\!+\!1,\beta}^{2}
			\!+\! C(\theta) \!\Big(\! \frac{h^{2s\!+\!2}}{\nu^{2}}  \| p_{t} \|_{2,s\!+\!1}^{2} \!+\! h^{2r} \| u_{t} \|_{2,r}^{2}  \!\Big) 
			\notag \\
			+& \nu \exp \Big[\! \frac{C(\theta)}{\nu} \! \big(\! k_{\rm{max}}^{4} \| u_{tt} \|_{2,2}^{2} 
			\!+\! \| |u|\|_{2,2,\beta}^{2} \!+\! \frac{ C(\!\theta\!) \nu T \!+\! 1 }{h \nu} 
			F_{2}^{2} \!\big)  \! \Big]  F_{3} \notag \\
			+& \exp \!\Big[ \frac{C(\theta)}{\nu} \big( k_{\rm{max}}^{4} \| u_{tt} \|_{2,2}^{2} 
			+ \| |u| \|_{2,2,\beta}^{2} \big) \Big] \! F_{1} \!+\! Ch^{2r} \!\| |u| \|_{\infty,r}^{2}  \notag \\
			+& C(\!\theta\!)  \nu h^{2r} \!\big( \!k_{\rm{max}}^{4} \!\| u_{tt} \|_{2,r\!+\!1}^{2}
			\!+\! \| |u| \|_{2,r\!+\!1,\beta}^{2} \!\big)   \notag \\
			+& C(\theta) T ( \| | u| \|_{\infty,1}^{2} + 1 ) \Big\{ Ch^{2r} \| |u| \|_{\infty,r+1}^{2} 
			+ \frac{Ch^{2s+2}}{\nu^{2}} \| |p| \|_{\infty,s+1}^{2}  \notag \\
			& \qquad \qquad \qquad \qquad \quad + \exp \Big[ \frac{C(\theta)}{\nu}  \big( k_{\rm{max}}^{4} \| u_{tt} \|_{2,2}^{2} 
			+ \| |u|\|_{2,2,\beta}^{2} + \frac{ C(\theta) \nu T + 1 }{h \nu} F_{2}^{2} 	\big)   \Big] F_{3} \Big\} \notag \\
			&+ C(\theta) k_{\rm{max}}^{4}  \Big[ \| u_{ttt} \|_{2,-1}^{2}  
			+ \| \nabla u_{tt} \|_{2,0}^{2}
			+ \| f_{tt} \|_{2,-1}^{2}  + \big( \| |u| \|_{\infty,1} + \| |u| \|_{\infty,1,\beta} \big) 
			\| \nabla u_{tt} \|_{2,0}^{2}  \Big] \notag \\
			&+ C(\theta) k_{\rm{max}}^{4} \| p_{tt} \|_{2,0}^{2} 
			+ \frac{C(\theta)h^{2s+2}}{\nu^{2}} \big( k_{\rm{max}}^{4} \| p_{tt} \|_{2,s+1}^{2} + \| |p| \|_{2,s+1,\beta}^{2} \big),
			\notag
		\end{align}
		which implies \eqref{eq:pressure-L2-conclusion}.
		For constant step case, $\varepsilon_{n} = 0$ and 
		the coefficients $\{ \beta_{\ell}^{(n)} \}_{\ell=0}^{2}$ are
		\begin{gather*}
			\beta _{2} = \frac{1}{4} (2 + \theta - \theta^2 ), \ \ \ 
			\beta _{1} = \frac{1}{2}\theta^2, \ \ \ 
			\beta _{0} = \frac{1}{4} (2 -\theta - \theta^2 ).
		\end{gather*}
		\begin{confidential}
			\color{darkblue}
			\begin{align*}
				\| p_{n,\beta} - p_{n,\beta}^{h} \| 
				=& \| \beta_{2} (p_{n+1} - p_{n+1}^{h}) + \beta_{1} (p_{n} - p_{n}^{h}) + 
				    \beta_{0} (p_{n-1} - p_{n-1}^{h}) \| \\
				\geq& \beta_{2} \| p_{n+1} - p_{n+1}^{h} \| 
				- \beta_{1} \| p_{n} - p_{n}^{h} \|
				- \beta_{0} \| p_{n-1} - p_{n-1}^{h} \|  \\
				=& (\beta_{2} - \beta_{1} - \beta_{0}) \| p_{n+1} - p_{n+1}^{h} \|
				+ \beta_{1} \big( \| p_{n+1} - p_{n+1}^{h} \| - \| p_{n} - p_{n}^{h} \| \big) 
				+ \beta_{0} \big( \| p_{n+1} - p_{n+1}^{h} \| - \| p_{n-1} - p_{n-1}^{h} \| \big) \\
				=& \frac{\theta}{2}(1-\theta) \| p_{n+1} - p_{n+1}^{h} \|
				+ \beta_{1} \big( \| p_{n+1} - p_{n+1}^{h} \| - \| p_{n} - p_{n}^{h} \| \big) \\
				&+ \beta_{0} \big( \| p_{n+1} - p_{n+1}^{h} \| - \| p_{n} - p_{n}^{h} \|
				+ \| p_{n} - p_{n}^{h} \| - \| p_{n-1} - p_{n-1}^{h} \| \big) \\
				=& \frac{\theta}{2}(1-\theta) \| p_{n+1} - p_{n+1}^{h} \|
				+ (\beta_{1} + \beta_{0}) \big( \| p_{n+1} - p_{n+1}^{h} \| - \| p_{n} - p_{n}^{h} \| \big) 
				+ \beta_{0} \big( \| p_{n} - p_{n}^{h} \| - \| p_{n-1} - p_{n-1}^{h} \| \big)
			\end{align*}
			\normalcolor
		\end{confidential}
		\begin{align*}
			\| p_{n,\beta} - p_{n,\beta}^{h} \| 
			\geq& \beta_{2} \| p_{n+1} - p_{n+1}^{h} \| - \beta_{1} \| p_{n} - p_{n}^{h} \|
			    - \beta_{0} \| p_{n-1} - p_{n-1}^{h} \|  \\
			=& \frac{\theta}{2}(1-\theta) \| p_{n+1} - p_{n+1}^{h} \|
			+ (\beta_{1} + \beta_{0}) \big( \| p_{n+1} - p_{n+1}^{h} \| - \| p_{n} - p_{n}^{h} \| \big) \\
			&+ \beta_{0} \big( \| p_{n} - p_{n}^{h} \| - \| p_{n-1} - p_{n-1}^{h} \| \big).	
		\end{align*}
		By triangle inequality and \eqref{eq:approx-thm}
		\begin{align}
			\label{eq:error-const-L1-pressure-1}
			\sum_{n=0}^{N} k \| p_{n} - p_{n}^{h} \|
			\leq& \sum_{n=1}^{N-1} k \| p_{n,\beta} - p_{n,\beta}^{h} \| 
			+ C(\theta) k \big( \| p_{1} - p_{1}^{h} \| + \| p_{0} - p_{0}^{h} \| \big) \\
			\leq& \sum_{n=1}^{N-1} k \big( C h^{s+1} \| p_{n,\beta} \|_{s+1} 
			\!+\! \| p_{n,\beta}^{h} - q_{n}^{h} \| \big) 
			\!+\! C(\theta) k \big( \| p_{1} \!-\! p_{1}^{h} \| \!+\! \| p_{0} \!-\! p_{0}^{h} \| \big). \notag 
 		\end{align}
\begin{confidential}
 			\color{darkblue}
 			\begin{align*}
 				C h^{s+1} \sum_{n=1}^{N-1} k \| p_{n,\beta} \|_{s+1} 
 				\leq& C h^{s+1} \Big(  \sum_{n=1}^{N-1} k \| p_{n,\beta} - p(t_{n,\beta}) \|_{s+1} 
 				                                + \sum_{n=1}^{N-1} k \| p(t_{n,\beta}) \|_{s+1}   \Big)       \\
 				\leq& C h^{s+1} \Big\{ C(\theta) \sum_{n=1}^{N-1} k^{5/2} \big( \int_{t_{n-1}}^{t_{n+1}}  \| p_{tt} \|_{s+1}^{2} dt \big)^{1/2} + \big[ (N-1)k \big( \sum_{n=1}^{N-1} k \| p(t_{n,\beta}) \|_{s+1}^{2} \big) \big]^{1/2}  \Big\}  \\
 				\leq& C h^{s+1} \Big( C(\theta) \Big[ (N-1)k \Big( \sum_{n=1}^{N-1} k^{2}\int_{t_{n-1}}^{t_{n+1}}  \| p_{tt} \|_{s+1}^{2} dt \Big) \Big]^{1/2} + \sqrt{T} \| |p|\|_{2,s+1,\beta} \Big) \\
 				\leq& C(\theta) \sqrt{T}  h^{s+1} \big(  k^{2}  \| p_{tt} \|_{2,s+1} + \| |p|\|_{2,s+1,\beta}  \big)
 			\end{align*}
 			\normalcolor
\end{confidential}
		By Cauchy-Schwarz inequality and \eqref{eq:DLN-Consistency1} in Lemma \ref{lemma:DLN-consistency}
		\begin{align}
			\label{eq:error-const-L1-pressure-2}
			C h^{s+1} \sum_{n=1}^{N-1} k \| p_{n,\beta} \|_{s+1}
			\leq& C(\theta) \sqrt{T}  h^{s+1} \big(  k^{2}  \| p_{tt} \|_{2,s+1} + \| |p|\|_{2,s+1,\beta}  \big),  \\
			\sum_{n=1}^{N-1} k \| p_{n,\beta}^{h} - q_{n}^{h} \|
			\leq& \sqrt{T} \Big( \sum_{n=1}^{N-1} k \| p_{n,\beta}^{h} - q_{n}^{h} \|^{2} \Big)^{1/2}.     \notag 
		\end{align}
		The bound of term $\sum_{n=1}^{N-1} k \| p_{n,\beta}^{h} - q_{n}^{h} \|^{2}$ is given 
		in \eqref{eq:error-L2-L2-pressure-final}. Thus we combine \eqref{eq:error-L2-L2-pressure-final},
		\eqref{eq:error-const-L1-pressure-1} and \eqref{eq:error-const-L1-pressure-2}
		to obtain 
		\begin{align}
		    &\sum_{n=0}^{N} k \| p_{n} - p_{n}^{h} \| \\
		    \leq& C(\theta) \sqrt{T} h^{s\!+\!1} \big(  k^{2}  \| p_{tt} \|_{2,s\!+\!1} \!+\! \| |p|\|_{2,s\!+\!1,\beta}  \big)
			\!+\! C(\theta) \sqrt{T} \!\Big(\! \frac{h^{s\!+\!1}}{\nu}  \| p_{t} \|_{2,s\!+\!1} \!+\! h^{r} \| u_{t} \|_{2,r}  \!\Big) 
			\notag \\
			+& \sqrt{\nu T} \exp \Big[\! \frac{C(\theta)}{\nu} \! \big(\! k^{4} \| u_{tt} \|_{2,2}^{2} 
			\!+\! \| |u|\|_{2,2,\beta}^{2} \!+\! \frac{ C(\!\theta\!) \nu T \!+\! 1 }{h \nu} 
			F_{2}^{2} \!\big)  \! \Big]  \sqrt{F_{3}}  \notag \\
			+& \exp \!\Big[ \frac{C(\theta)}{\nu} \big( k^{4} \| u_{tt} \|_{2,2}^{2} 
			+ \| |u| \|_{2,2,\beta}^{2} \big) \Big] \! \sqrt{T F_{1}} \!+\! C \sqrt{T} h^{r} \!\| |u| \|_{\infty,r}  \notag \\
			+& C(\!\theta\!)  \sqrt{\nu T } h^{r} \!\big( \!k^{2} \!\| u_{tt} \|_{2,r\!+\!1}
			\!+\! \| |u| \|_{2,r\!+\!1,\beta} \!\big)   \notag \\
			+& C(\theta) T ( \| | u| \|_{\infty,1} + 1 ) \Big\{ Ch^{r} \| |u| \|_{\infty,r+1}
			+ \frac{Ch^{s+1}}{\nu} \| |p| \|_{\infty,s+1}  \notag \\
			& \qquad \qquad \qquad \qquad \quad + \exp \Big[ \frac{C(\theta)}{\nu}  \big( k^{4} \| u_{tt} \|_{2,2}^{2} 
			+ \| |u|\|_{2,2,\beta}^{2} + \frac{ C(\theta) \nu T + 1 }{h \nu} F_{2}^{2} 	\big)   \Big] \sqrt{F_{3}} \Big\} \notag \\
			+& C(\theta) \sqrt{T} k^{2}  \Big[ \| u_{ttt} \|_{2,-1}
			+ \| \nabla u_{tt} \|_{2,0}
			+ \| f_{tt} \|_{2,-1} + \sqrt{\| |u| \|_{\infty,1} \!+\! \| |u| \|_{\infty,1,\beta} }
			\| \nabla u_{tt} \|_{2,0}  \Big] \notag \\
			+& C(\theta) \sqrt{T} k^{2} \| p_{tt} \|_{2,0}
			+ \frac{C(\theta)\sqrt{T} h^{s+1}}{\nu} \big( k^{2} \| p_{tt} \|_{2,s+1} + \| |p| \|_{2,s+1,\beta} \big) \notag \\
			+& C(\theta) k \big( \| p_{1} \!-\! p_{1}^{h} \| \!+\! \| p_{0} \!-\! p_{0}^{h} \| \big), \notag
		\end{align}
		which implies \eqref{eq:pressure-L1-conclusion}.
		
	\end{proof}

	\section{Implementation of Adaptive DLN Algorithm}
	\label{sec:Implement-DLN}
	We present two ways of time adaptivity for the whole family of DLN methods. The first way is to use the local truncation error (LTE) criterion:  we apply the revised AB2 method (herein AB2-like method) to estimate the error of the DLN scheme for NSE and adjust the time step according to the ratio of the required tolerance and the estimator. The second way is to adapt the time step to control the numerical dissipation.  
	\subsection{Local Truncation Error Criterion} \ \\
	Given four previous solutions $u_{n}^{h},u_{n-1}^{h},u_{n-2}^{h},u_{n-3}^{h}$, 
	the AB2-like solution for NSE at time $t_{n+1}$ is 
	\begin{align}
	    \label{eq:AB2-like}
	    u_{n+1}^{h,\tt AB2} &= 
	    \Big[ 1 +  \alpha_2 \frac{ (t_{n+1} - t_n) ( t_{n+1} + t_n - 2 t_{n-1,\beta} )}{2 ( t_{n,\beta} - t_{n-1,\beta}) \widehat{k}_{n-1}}  \Big]
	    u_{n}^{h} 
	    \\
	    & 
	    +  \frac{(t_{n+1} - t_n) }{2 ( t_{n,\beta} - t_{n-1,\beta})} \Big[ \alpha_1  \frac{ ( t_{n+1} + t_n - 2 t_{n-1,\beta} )}{\widehat{k}_{n-1}} 
	    - \alpha_2 \frac{ ( t_{n+1} + t_n - 2t_{n,\beta} ) }{\widehat{k}_{n-2}}  \Big] 
	   u_{n-1}^{h}
	    \notag
	    \\
	    & 
	    + \frac{(t_{n+1} - t_n)}{2 ( t_{n,\beta} - t_{n-1,\beta})} \Big[ \alpha_0 \frac{  ( t_{n+1} + t_n - 2 t_{n-1,\beta} )}{ \widehat{k}_{n-1}}  
	    -  \alpha_1 \frac{  ( t_{n+1} + t_n - 2t_{n,\beta} ) }{ \widehat{k}_{n-2}}  \Big]
	    u_{n-2}^{h}  
	    \notag 
	    \\
	    & 
	    - \alpha_0 \frac{ (t_{n+1} - t_n ) ( t_{n+1} + t_n - 2t_{n,\beta} ) }{2 ( t_{n,\beta} - t_{n-1,\beta}) \widehat{k}_{n-2}}   u_{n-3}^{h}, \notag
	\end{align}
    and the estimators are 
	\begin{align}
	\widehat{T}_{n+1}
	=&
	\frac{ |G^{(n)}|}{ | G^{(n)} + {\cal R}^{(n)} | } \| u_{n+1}^{h,\tt DLN}-u_{n+1}^{h,\tt AB2} \|,     \tag{Absolute estimator}   \label{eq:ESTabs-AB2DLN} \\
	\widehat{T}_{n+1}
	=&
	\frac{ |G^{(n)}|}{ | G^{(n)} + {\cal R}^{(n)} | } \frac{\| u_{n+1}^{h,\tt DLN}-u_{n+1}^{h,\tt AB2} \|}{\| u_{n+1}^{h,\tt DLN} \|},
	\tag{Relative estimator}   \label{eq:ESTrel-AB2DLN}
	\end{align}
	where
	\begin{align}
	\label{eq:EST-LTE-coefficients}
	G^{(n)}
	&=\Big(\frac{1}{2}-\frac{\alpha _{0}}{2\alpha _{2}}\frac{1-\varepsilon
		_{n}}{1+\varepsilon _{n}}\Big)\Big(\beta _{2}^{(n)}-\beta _{0}^{(n)}\frac{%
		1-\varepsilon _{n}}{1+\varepsilon _{n}}\Big)^{2}+\frac{\alpha _{0}}{6\alpha
		_{2}}\Big(\frac{1-\varepsilon _{n}}{1+\varepsilon _{n}}\Big)^{3}-\frac{1}{6}, \\
	{\cal R}^{(n)} 
	&= 
	\frac{1}{12}
	\Big[ 2 + 
	\frac{3(1 - \varepsilon_{n})}{1 + \varepsilon_{n}} \Big( 1 - \beta_{2}^{(n-2)}  
	\frac{1 - \varepsilon_{n-1}}{1 + \varepsilon_{n-1}}  
	+ \beta_{0}^{(n-2)} \frac{1 - \varepsilon_{n-2}}{1 + \varepsilon_{n-2}}  
	\frac{1 - \varepsilon_{n-1}}{1 + \varepsilon_{n-1}}  \Big) \times
	\notag \\
	& 
	\hspace{4.5cm}
	\times \Big( 1 - \beta_{2}^{(n-1)} \frac{1 - \varepsilon_{n}}{1 + \varepsilon_{n}} 
	+ \beta_{0}^{(n-1)} \frac{1 - \varepsilon_{n-1}}{1 + \varepsilon_{n-1}} 
	\frac{1 - \varepsilon_{n}}{1 + \varepsilon_{n}}  \Big) 
	\notag \\
	& \ \ \ \ \ \ \ \ 
	\!+\! \frac{3(1 \!-\! \varepsilon_{n})}{1 \!+\! \varepsilon_{n}} 
	\Big( \! \frac{2}{1 \!+\! \varepsilon_{n}}  \!-\!  \beta_{2}^{(n\!-\!2)} 
	\frac{1 \!-\! \varepsilon_{n\!-\!1}}{1 \!+\! \varepsilon_{n\!-\!1}}  \frac{1 \!-\! \varepsilon_{n}}{1 \!+\! \varepsilon_{n}} 
	\!+\! \beta_{0}^{(n\!-\!2)} \frac{1 \!-\! \varepsilon_{n\!-\!2}}{1 \!+\! \varepsilon_{n\!-\!2}} 
	\frac{1 \!-\! \varepsilon_{n\!-\!1}}{1 \!+\! \varepsilon_{n\!-\!1}}  \frac{1 \!-\! \varepsilon_{n}}{1 \!+\! \varepsilon_{n}}  \! \Big) \! \times
	\notag \\
	& 
	\hspace{7cm}
	\times \Big( - \beta_{2}^{(n-1)}  
	+ \beta_{0}^{(n-1)} \frac{1 - \varepsilon_{n-1}}{1 + \varepsilon_{n-1}} \Big)  \Big]. \notag
	\end{align}
	\begin{confidential}
		\color{darkblue}
		\begin{align*}
		& 
		{\cal R}^{(n)}
		= 
		\frac{1}{12}
		\Big[ 2 + 
		\frac{3}{\tau_n} \Big( 1 - \beta_{2}^{(n-2)} \frac{1}{\tau_{n-1}}
		+ \beta_{0}^{(n-2)} \frac{1}{\tau_{n-2}}\frac{1}{\tau_{n-1}}\Big) \times
		\\
		& 
		\hspace{6.2cm}
		\times \Big( 1 - \beta_{2}^{(n-1)} \frac{1}{\tau_n} 
		+ \beta_{0}^{(n-1)} \frac{1}{\tau_{n-1}}\frac{1}{\tau_n} \Big) 
		\\
		& 
		\hspace{2.cm}
		+ \frac{3}{\tau_n} \Big( 1 + \frac{1}{\tau_n} -  \beta_{2}^{(n-2)} \frac{1}{\tau_{n-1}} \frac{1}{\tau_n}
		+ \beta_{0}^{(n-2)} \frac{1}{\tau_{n-2}}\frac{1}{\tau_{n-1}}\frac{1}{\tau_n} \Big) \times
		\\
		& 
		\hspace{7cm}
		\times \Big( - \beta_{2}^{(n-1)}  
		+ \beta_{0}^{(n-1)} \frac{1}{\tau_{n-1}} \Big)  
		\Big].
		\end{align*}
		\normalcolor
	\end{confidential}
	We refer to \cite{LPT22_Tech} for the derivation of the AB2-like method in \eqref{eq:AB2-like} and the 
	estimator in \eqref{eq:ESTabs-AB2DLN} and \eqref{eq:ESTrel-AB2DLN}.
	We use the step controller proposed by Hairer and Wanner in \cite{HW10_Springer} to adjust the next time step
	\begin{align}
	k_{n+1}=k_{n}\cdot \min \Big\{1.5,\max \Big\{0.2,\kappa \big({\text{Tol}}/{%
		\widehat{T}_{n+1} }\big)^{1/3}\Big\}\Big\},
	\label{eq:StepController-Improved}
	\end{align}
	where $\text{Tol}$ is the required tolerance and the safety factor $\kappa \in (0,1]$ is selected to minimize the number of step rejections. 
	If $\widehat{T}_{n+1} >\text{Tol}$, then the DLN solution is rejected and the current step $k_{n}$ is adjusted by \eqref{eq:StepController-Improved} for recomputing. 
	We summarize the above adaptive DLN algorithm in Algorithm \ref{alg:Adaptivity-AB2-like} 
	\LinesNumberedHidden
	\begin{algorithm}
		\caption{Adaptivity of  DLN (estimator of LTE by AB2-like scheme)}
		\label{alg:Adaptivity-AB2-like}
		\KwIn{tolerance $\text{Tol}$, four previous solutions $u_{n}^{h},u_{n-1}^{h},u_{n-2}^{h},u_{n-3}^{h}$
				and $p_{n}^{h},p_{n-1}^{h},p_{n-2}^{h},p_{n-3}^{h}$,
			    current time step $k_{n}$, three previous time step $k_{n-1},k_{n-2},k_{n-3}$, safety factor $\kappa$, \;}
			compute the DLN solution $u_{n+1}^{h,\tt DLN}$ and $p_{n+1}^{h,\tt DLN}$ 
			by \eqref{eq:DLN-Semi-Alg} \;
			compute the AB2-like solution $u_{n+1}^{h,\tt AB2}$ by \eqref{eq:AB2-like} \;
			use $k_{n},k_{n-1},k_{n-2},k_{n-3}$ to update $\varepsilon_{n},\varepsilon_{n-1},\varepsilon_{n-2}$ \;
			compute $G^{(n)}$, ${\cal R}^{(n)}$ by \eqref{eq:EST-LTE-coefficients} \;
			$\widehat{T}_{n+1} \Leftarrow \frac{ |G^{(n)}|}{ | G^{(n)} + {\cal R}^{(n)} | } 
			\| u_{n+1}^{h,\tt DLN}-u_{n+1}^{h,\tt AB2} \| $
			or $\widehat{T}_{n+1} \Leftarrow \frac{ |G^{(n)}|}{ | G^{(n)} + {\cal R}^{(n)} | } 
			\frac{\| u_{n+1}^{h,\tt DLN} - u_{n+1}^{h,\tt AB2} \|}{\| u_{n+1}^{h} \|}$ \;
			\uIf{$ \widehat{T}_{n+1}  < \rm{Tol}$}
			{
				$u_{n+1}^{h} \Leftarrow u_{n+1}^{h,\tt DLN}$ and 
				$p_{n+1}^{h} \Leftarrow p_{n+1}^{h,\tt DLN}$  \tcp*{accept the result}
				$k_{n\!+\!1} \!\Leftarrow \! k_{n} \cdot \min \big\{ \!1.5, \max \big\{\!0.2, \kappa \big(\!\frac {\text{Tol}}{ \widehat{T}_{n+1}  }\!\big)^{1/3} \big\} \!\big\}$  
				\tcp*{adjust step by \eqref{eq:StepController-Improved}}
			}\Else
			{
				\tt{// adjust current step to recompute}  
				$k_{n} \!\Leftarrow \! k_{n} \cdot \min \big\{ 1.5, \max \big\{0.2, \kappa \big(\frac {\text{Tol}}{ \widehat{T}_{n+1}  }\big)^{1/3} \big\} \big\}$ \;
			}
	\end{algorithm}

	\subsection{Numerical Dissipation Criterion} 
	The algorithm proposed by Capuano, Sanderse, De Angelis and Coppola \cite{CSDC17} calibrates the step size to ensure the ratio of numerical dissipation and viscosity under the required value and its effect on the fully-implicit DLN scheme has been tested in \cite{LPQT21_NMPDE}.
	Given the tolerance $\rm{Tol}$, the maximum time step $k_{\rm{max}}$, the minimum time step $k_{\rm{min}}$, we compute the DLN solution and the ratio of numerical dissipation and viscosity
	\begin{gather*}
		\chi_{n+1} = \mathcal{E}_{n+1}^{\tt ND} / \mathcal{E}_{n+1}^{\tt VD}.
	\end{gather*}
	If $\chi_{n\!+\!1} \!\leq \! \rm{Tol}$, we accept the current solutions and double the time step.
	Otherwise, we halve the time step for recomputing. We summarize the algorithm in Algorithm \ref{alg:Adaptivity-ND}.
	\begin{algorithm}
		\caption{Adaptivity by Numerical Dissipation Criterion}
		\label{alg:Adaptivity-ND}
		\KwIn{tolerance $\text{Tol}$, two previous solutions $u_{n}^{h},u_{n-1}^{h}$ 
		    and $p_{n}^{h},p_{n-1}^{h}$, current time step $k_{n}$, previous time step $k_{n-1}$, \;}
		compute the DLN solution $u_{n+1}^{h,\tt DLN}$ and $p_{n+1}^{h,\tt DLN}$ by \eqref{eq:DLN-Semi-Alg} \;
		compute numerical dissipation $\mathcal{E}_{n+1}^{\tt ND}$ and viscosity $\mathcal{E}_{n+1}^{\tt VD}$ \;
		$\chi_{n+1} \Leftarrow \mathcal{E}_{n+1}^{\tt ND} / \mathcal{E}_{n+1}^{\tt VD}$ \;
		\uIf{$ \chi_{n+1}  < \rm{Tol}$}
		{
			$u_{n+1}^{h} \Leftarrow u_{n+1}^{h,\tt DLN}$ 
			and $p_{n+1}^{h} \Leftarrow p_{n+1}^{h,\tt DLN}$   \tcp*{accept result}
			$k_{n\!+\!1} \!\Leftarrow \!  \min \big\{ 2k_{n}, k_{\rm{max}} \!\big\}$  
			\tcp*{double the step }
		}\Else
		{
			$k_{n} \!\Leftarrow \!  \max \big\{ 0.5k_{n}, k_{\rm{min}} \big\}$
			\tcp*{halve current step to recompute}
		}
	\end{algorithm}
	
	\section{Numerical Tests}
	\label{sec:Num-Test}
	We apply the semi-implicit DLN algorithm in \eqref{eq:DLN-Semi-Alg} with $\theta = 2/3, 2/\sqrt{5}, 1$ for all numerical tests. $\theta = 2/3$ is suggested in \cite{DLN83_SIAM_JNA} to balance the stability and local truncation error. 
	$\theta = 2/\sqrt{5}$ is recommended in \cite{KS05} for stability at infinity (a property similar to $L$-stability).
	$\theta = 1$ reduces to the midpoint rule.
	We use software FreeFem++ and Taylor-Hood ($P2-P1$) finite element space for programming. 
	\subsection{Convergence Test}
	\label{subsec:test-const}
	We use the Taylor-Green benchmark problem \cite{TG1937_RSL} to verify that the semi-implicit DLN algorithm is second-order accurate. In addition, its efficiency over the fully-implicit algorithm can be observed.
	The exact solutions in the domain $\Omega = [0,1]\times[0,1]$ are 
	\begin{gather}
		\label{eq:Taylor-Green-exact}
		u_{1}(x,y,t) = -\cos(\omega \pi x) \sin(\omega \pi y) \exp (- 2 \omega^2 \pi^2 t/\tau ), \\
		u_{2}(x,y,t) = \sin(\omega \pi x) \cos(\omega \pi y) \exp (- 2 \omega^2 \pi^2 t/\tau ),  \notag \\
		p(x,y,t) = - \frac{1}{4} \big( \cos(2 \omega \pi x) + \cos(2 \omega \pi y) \big) 
		\exp (- 4 \omega^2 \pi^2 t/\tau ).    \notag
	\end{gather}
	We set the parameters $\omega = 1$, $\tau = 1/\nu = 100$. The initial value, boundary condition and source function $f$ are determined by the exact solutions in \eqref{eq:Taylor-Green-exact}. 
	We require that the constant time step $k$ and mesh diameter $h$ are the same to satisfy the time-diameter condition in \eqref{eq:time-h-limit}. 
	We simulate the problem over the time interval $[0,1]$.
	The convergence rate $R$ is calculated by 
	\begin{gather*}
		R = \frac{\log\big({\rm{error}}(k,h) / {\rm{error}}(\frac{k}{2},\frac{h}{2})\big)}{\log(2)}.
	\end{gather*}
	The results of the semi-implicit DLN algorithm (with constant time step) are given in \cref{table:Semi-DLN-23-Linf,table:Semi-DLN-23-L2,table:Semi-DLN-25-Linf,table:Semi-DLN-25-L2,table:Semi-DLN-1-Linf,table:Semi-DLN-1-L2}.
	We see that the semi-implicit DLN has third-order convergence in velocity and second-order convergence in pressure for all three $\theta$ values. 
	Hence the semi-implicit DLN scheme has much better performance in the Taylor-Green problem
	than the theories in Subsection \ref{sec:Err-Ana} suggest. 
	Then we apply the fully-implicit DLN scheme to the same problem and use fixed point iteration to solve the non-linear system at each time step. 
	The error and convergence rate are given in \cref{table:Full-DLN-23-Linf,table:Full-DLN-23-L2,table:Full-DLN-25-Linf,table:Full-DLN-25-L2,table:Full-DLN-1-Linf,table:Full-DLN-1-L2}. 
	From the above tables, we observe that the semi-implicit DLN algorithm outperforms the fully-implicit algorithm for all three $\theta$ values in this test problem: 
	The two schemes have almost the same error magnitude but the fully-implicit scheme takes twice the time to finish the simulation since the fully-implicit scheme takes two iterations on average at each time step.
	\begin{table}[ptbh]
    	\centering
    	\caption{$L^{\infty}$-norm of error and rate for semi-implicit DLN scheme ($\theta=2/3$)}
    	\begin{tabular}{lccccccc}
    		\hline
    		\hline
    		$k\!=\!h$  & $\| |u\! - \!u^{h}| \|_{\!\infty\!,\!0\!}$ & $\!R\!$ 
    		& $\| |u \!- \!u^{h}| \|_{\!\infty\!,\!1\!}$   & $\!R\!$  
    		& $\| |p \!-\! p^{h}| \|_{\!\infty\!,\!0\!}$   & $\!R\!$ & Time(s) 
    		\\
    		\hline
    		\hline
    		$1\!/\!16$     & 3.9474e-4   & -      & 4.6605e-2  & -     & 2.7357e-3  & -     & 8.1388
    		\\
    		$1\!/\!32$     & 2.8230e-5   & 3.8056   & 6.4712e-3  & 2.8484  & 7.0563e-4  & 1.9550  & 64.792
    		\\
    		$1\!/\!64$     & 2.1586e-6   & 3.7091   & 8.5120e-4  & 2.9265  & 1.7820e-4  & 1.9854  & 572.01
    		\\
    		$1\!/\!128$    & 1.9754e-7   & 3.4499   & 1.0916e-4  & 2.9630  & 4.4717e-5  & 1.9946  & 4482.08
    		\\
    		$1\!/\!256$    & 2.1700e-8   & 3.1864   & 1.3823e-5  & 2.9814  & 1.1180e-5  & 1.9999  & 36165.2
			\\
    		\hline
    	\end{tabular}
    	\label{table:Semi-DLN-23-Linf}
    \end{table}

	\begin{table}[ptbh]
    	\centering
    	\caption{$L^{2}$-norm of error and rate for semi-implicit DLN scheme ($\theta=2/3$)}
    	\begin{tabular}{lccccccc}
    		\hline
    		\hline
    		$k\!=\!h$  & $\| |u \!-\! u^{h}| \|_{\!2\!,\!0\!}$ & $\!R\!$ 
    		& $\| |u \!-\! u^{h}| \|_{\!2\!,\!1\!}$ & $\!R\!$  
    		& $\| |p \!-\! p^{h}| \|_{\!2\!,\!0\!}$   & $\!R\!$ & Time(s)
    		\\
    		\hline
    		\hline
    		$1\!/\!16$     & 2.3215e-4   & -      & 2.6161e-2  & -     & 1.3270e-3  & -     & 8.1388
    		\\
    		$1\!/\!32$     & 1.6575e-5   & 3.8080   & 3.3174e-3  & 2.9793  & 3.1602e-4  & 2.0700  & 64.792
    		\\
    		$1\!/\!64$     & 1.4593e-6   & 3.5057   & 4.1514e-4  & 2.9984  & 7.6861e-5  & 2.0397  & 572.01
    		\\
    		$1\!/\!128$    & 1.5827e-7   & 3.2048   & 5.1850e-5  & 3.0012  & 1.8945e-5  & 2.0204  & 4482.08
    		\\
    		$1\!/\!256$    & 1.9010e-8   & 3.0576   & 6.4763e-6  & 3.0011  & 4.7026e-6  & 2.0103  & 36165.2
    		\\
    		\hline
    	\end{tabular}
    	\label{table:Semi-DLN-23-L2}
    \end{table}

	\begin{table}[ptbh]
    	\centering
    	\caption{$L^{\infty}$-norm of error and rate for semi-implicit DLN scheme ($\theta=2/\sqrt{5}$)}
    	\begin{tabular}{lccccccc}
    		\hline
    		\hline
    		$k\!=\!h$  & $\| |u \!-\! u^{h}| \|_{\!\infty\!,\!0\!}$ & $\!R\!$ 
    		& $\| |u \!-\! u^{h}| \|_{\!\infty\!,\!1\!}$   & $\!R\!$  
    		& $\| |p \!-\! p^{h}| \|_{\!\infty\!,\!0\!}$   & $\!R\!$ & Time(s) 
    		\\
    		\hline
    		\hline
    		$1\!/\!16$     & 4.0528e-4   & -      & 4.7930e-2  & -     & 3.0208e-3  & -     & 8.3926
    		\\
    		$1\!/\!32$     & 2.9297e-5   & 3.7901   & 6.7547e-3  & 2.8270  & 7.8296e-4  & 1.9479  & 66.443
    		\\
    		$1\!/\!64$     & 2.2303e-6   & 3.7155   & 8.9557e-4  & 2.9150  & 1.9809e-4  & 1.9827  & 550.33
    		\\
    		$1\!/\!128$    & 2.0105e-7   & 3.4716   & 1.1533e-4  & 2.9570  & 4.9746e-5  & 1.9935  & 4670.8
    		\\
    		$1\!/\!256$    & 2.1835e-8   & 3.2029   & 1.4635e-5  & 2.9783  & 1.2439e-5  & 1.9997  & 36117.4
    		\\
    		\hline
    	\end{tabular}
    	\label{table:Semi-DLN-25-Linf}
    \end{table}

	\begin{table}[ptbh]
    	\centering
    	\caption{$L^{2}$-norm of error and rate for semi-implicit DLN scheme ($\theta=2/\sqrt{5}$)}
    	\begin{tabular}{lccccccc}
    		\hline
    		\hline
    		$k=h$  & $\| |u \!-\! u^{h}| \|_{\!2\!,\!0\!}$ & $\!R\!$ 
    		& $\| |u \!-\! u^{h}| \|_{\!2\!,\!1\!}$ & $\!R\!$  
    		& $\| |p \!-\! p^{h}| \|_{\!2\!,\!0\!}$   & $\!R\!$ & Time(s)
    		\\
    		\hline
    		\hline
    		$1\!/\!16$     & 2.3063e-4   & -      & 2.5960e-2  & -     & 1.4135e-3  & -     & 8.3926
    		\\
    		$1\!/\!32$     & 1.6544e-5   & 3.8012   & 3.3095e-3  & 2.9716  & 3.2890e-4  & 2.1036  & 66.443
    		\\
    		$1\!/\!64$     & 1.4587e-6   & 3.5036   & 4.1515e-4  & 2.9949  & 7.8584e-5  & 2.0653  & 550.33
    		\\
    		$1\!/\!128$    & 1.5816e-7   & 3.2052   & 5.1876e-5  & 3.0005  & 1.9165e-5  & 2.0357  & 4670.8
    		\\
    		$1\!/\!256$    & 1.8980e-8   & 3.0588   & 6.4789e-6  & 3.0012  & 4.7299e-6  & 2.0186  & 36117.4
    		\\
    		\hline
    	\end{tabular}
    	\label{table:Semi-DLN-25-L2}
    \end{table}

	\begin{table}[ptbh]
    	\centering
    	\caption{$L^{\infty}$-norm of error and rate for semi-implicit DLN scheme ($\theta=1$)}
    	\begin{tabular}{lccccccc}
    		\hline
    		\hline
    		$k\!=\!h$  & $\| |u \!-\! u^{h}| \|_{\!\infty\!,\!0\!}$ & $\!R\!$ 
    		& $\| |u \!-\! u^{h}| \|_{\!\infty\!,\!1\!}$   & $\!R\!$  
    		& $\| |p \!-\! p^{h}| \|_{\!\infty\!,\!0\!}$   & $\!R\!$ & Time(s) 
    		\\
    		\hline
    		\hline
    		$1\!/\!16$     & 4.1609e-4   & -      & 4.9275e-2  & -     & 3.2988e-3  & -     & 8.1731
    		\\
    		$1\!/\!32$     & 3.0244e-5   & 3.7821   & 7.0048e-3  & 2.8144  & 8.5923e-4  & 1.9408  & 65.027
    		\\
    		$1\!/\!64$     & 2.2919e-6   & 3.7220   & 9.3311e-4  & 2.9082  & 2.1772e-4  & 1.9806  & 527.63
    		\\
    		$1\!/\!128$    & 2.0408e-7   & 3.4893   & 1.2046e-4  & 2.9535  & 5.4716e-5  & 1.9924  & 4598.7
    		\\
    		$1\!/\!256$    & 2.1951e-8   & 3.2168   & 1.5306e-5  & 2.9764  & 1.3687e-5  & 1.9991  & 36536.1
    		\\
    		\hline
    	\end{tabular}
    	\label{table:Semi-DLN-1-Linf}
    \end{table}

	\begin{table}[ptbh]
    	\centering
    	\caption{$L^{2}$-norm of error and rate for semi-implicit DLN scheme ($\theta=1$)}
    	\begin{tabular}{lccccccc}
    		\hline
    		\hline
    		$k=h$  & $\| |u \!-\! u^{h}| \|_{\!2\!,\!0\!}$ & $\!R\!$ 
    		& $\| |u \!-\! u^{h}| \|_{\!2\!,\!1\!}$  & $\!R\!$  
    		& $\| |p \!-\! p^{h}| \|_{\!2\!,\!0\!}$  & $\!R\!$ & Time(s)
    		\\
    		\hline
    		\hline
    		$1\!/\!16$     & 2.3657e-4   & -      & 2.6577e-2  & -     & 3.0674e-3  & -     & 8.1731
    		\\
    		$1\!/\!32$     & 1.6833e-5   & 3.8129   & 3.3987e-3  & 2.9672  & 8.1328e-4  & 1.9152  & 65.027
    		\\
    		$1\!/\!64$     & 1.4761e-6   & 3.5115   & 4.3047e-4  & 2.9810  & 2.0835e-4  & 1.9648  & 527.63
    		\\
    		$1\!/\!128$    & 1.5887e-7   & 3.2159   & 5.4163e-5  & 2.9905  & 5.2629e-5  & 1.9851  & 4598.7
    		\\
    		$1\!/\!256$    & 1.8995e-8   & 3.0642   & 6.7926e-6  & 2.9953  & 1.3195e-5  & 1.9959  & 36536.1
    		\\
    		\hline
    	\end{tabular}
    	\label{table:Semi-DLN-1-L2}
    \end{table}
	\begin{table}[ptbh]
    	\centering
    	\caption{$L^{\infty}$-norm of error and rate for fully-implicit DLN scheme ($\theta=2/3$)}
    	\begin{tabular}{lccccccc}
    		\hline
    		\hline
    		$k\!=\!h$  & $\| |u\! - \!u^{h}| \|_{\!\infty\!,\!0\!}$ & $\!R\!$ 
    		& $\| |u \!- \!u^{h}| \|_{\!\infty\!,\!1\!}$   & $\!R\!$  
    		& $\| |p \!-\! p^{h}| \|_{\!\infty\!,\!0\!}$   & $\!R\!$ & Time(s) 
    		\\
    		\hline
    		\hline
    		$1\!/\!16$     & 3.9446e-4   & -        & 4.6505e-2  & -       & 2.7694e-3  & -     & 15.605
    		\\
    		$1\!/\!32$     & 2.8227e-5   & 3.8047   & 6.4700e-3  & 2.8455  & 7.1375e-4  & 1.9561  & 125.958
    		\\
    		$1\!/\!64$     & 2.1586e-6   & 3.7089   & 8.5118e-4  & 2.9262  & 1.8025e-4  & 1.9854  & 1032.9
    		\\
    		$1\!/\!128$    & 1.9754e-7   & 3.4499   & 1.0916e-4  & 2.9630  & 4.5231e-5  & 1.9946  & 8796.17
    		\\
    		$1\!/\!256$    & 2.1700e-8   & 3.1864   & 1.3823e-5  & 2.9814  & 1.1309e-5  & 1.9999  & 71983.6
			\\
    		\hline
    	\end{tabular}
    	\label{table:Full-DLN-23-Linf}
    \end{table}

	\begin{table}[ptbh]
    	\centering
    	\caption{$L^{2}$-norm of error and rate for fully-implicit DLN scheme ($\theta=2/3$)}
    	\begin{tabular}{lccccccc}
    		\hline
    		\hline
    		$k\!=\!h$  & $\| |u \!-\! u^{h}| \|_{\!2\!,\!0\!}$ & $\!R\!$ 
    		& $\| |u \!-\! u^{h}| \|_{\!2\!,\!1\!}$   & $\!R\!$  
    		& $\| |p \!-\! p^{h}| \|_{\!2\!,\!0\!}$   & $\!R\!$ & Time(s)
    		\\
    		\hline
    		\hline
    		$1\!/\!16$     & 2.3156e-4   & -      & 2.6098e-2  & -     & 1.3308e-3  & -     & 15.605
    		\\
    		$1\!/\!32$     & 1.6574e-5   & 3.8044   & 3.3170e-3  & 2.9760  & 3.1649e-4  & 2.0721  & 125.958
    		\\
    		$1\!/\!64$     & 1.4593e-6   & 3.5055   & 4.1514e-4  & 2.9982  & 7.6911e-5  & 2.0409  & 1032.9
    		\\
    		$1\!/\!128$    & 1.5827e-7   & 3.2048   & 5.1850e-5  & 3.0011  & 1.8948e-5  & 2.0211  & 8796.17
    		\\
    		$1\!/\!256$    & 1.9009e-8   & 3.0576   & 6.4763e-6  & 3.0011  & 4.7022e-6  & 2.0107  & 71983.6
    		\\
    		\hline
    	\end{tabular}
    	\label{table:Full-DLN-23-L2}
    \end{table}
	
	\begin{table}[ptbh]
    	\centering
    	\caption{$L^{\infty}$-norm of error and rate for fully-implicit DLN scheme ($\theta=2/\sqrt{5}$)}
    	\begin{tabular}{lccccccc}
    		\hline
    		\hline
    		$k\!=\!h$  & $\| |u \!-\! u^{h}| \|_{\!\infty\!,\!0}$ & $\!R\!$ 
    		& $\| |u \!-\! u^{h}| \|_{\!\infty\!,\!1\!}$   & $\!R\!$  
    		& $\| |p \!-\! p^{h}| \|_{\!\infty\!,\!0\!}$   & $\!R\!$ & Time(s) 
    		\\
    		\hline
    		\hline
    		$1\!/\!16$     & 4.0498e-4   & -      & 4.7831e-2  & -     & 3.0555e-3  & -     & 18.7007
    		\\
    		$1\!/\!32$     & 2.9293e-5   & 3.7892   & 6.7535e-3  & 2.8243  & 7.9133e-4  & 1.9490  & 149.992
    		\\
    		$1\!/\!64$     & 2.2302e-6   & 3.7153   & 8.9555e-4  & 2.9148  & 2.0021e-4  & 1.9828  & 1192.33
    		\\
    		$1\!/\!128$    & 2.0105e-7   & 3.4715   & 1.1533e-4  & 2.9570  & 5.0276e-5  & 1.9936  & 9766.42
    		\\
    		$1\!/\!256$    & 2.1835e-8   & 3.2029   & 1.4635e-5  & 2.9783  & 1.2572e-5  & 1.9997  & 71773.7
    		\\
    		\hline
    	\end{tabular}
    	\label{table:Full-DLN-25-Linf}
    \end{table}

	\begin{table}[ptbh]
    	\centering
    	\caption{$L^{2}$-norm of error and rate for fully-implicit DLN scheme ($\theta=2/\sqrt{5}$)}
    	\begin{tabular}{lccccccc}
    		\hline
    		\hline
    		$k\!=\!h$  & $\| |u \!-\! u^{h}| \|_{\!2\!,\!0\!}$ & $\!R\!$ 
    		& $\| |u \!-\! u^{h}| \|_{\!2\!,\!1\!}$   & $\!R\!$  
    		& $\| |p \!-\! p^{h}| \|_{\!2\!,\!0\!}$   & $\!R\!$ & Time(s)
    		\\
    		\hline
    		\hline
    		$1\!/\!16$     & 2.2951e-4   & -      & 2.5844e-2  & -     & 1.4177e-3  & -     & 18.7007
    		\\
    		$1\!/\!32$     & 1.6541e-5   & 3.7944   & 3.3086e-3  & 2.9655  & 3.2947e-4  & 2.1053  & 149.992
    		\\
    		$1\!/\!64$     & 1.4587e-6   & 3.5033   & 4.1514e-4  & 2.9945  & 7.8653e-5  & 2.0666  & 1192.33
    		\\
    		$1\!/\!128$    & 1.5816e-7   & 3.2052   & 5.1876e-5  & 3.0005  & 1.9172e-5  & 2.0365  & 9766.42
    		\\
    		$1\!/\!256$    & 1.8980e-8    & 3.0588   & 6.4789e-6  & 3.0012  & 4.7302e-6  & 2.0190  & 71773.7
    		\\
    		\hline
    	\end{tabular}
    	\label{table:Full-DLN-25-L2}
    \end{table}

	\begin{table}[ptbh]
    	\centering
    	\caption{$L^{\infty}$-norm of error and rate for fully-implicit DLN scheme ($\theta=1$)}
    	\begin{tabular}{lccccccc}
    		\hline
    		\hline
    		$k\!=\!h$  & $\| |u \!-\! u^{h}| \|_{\!\infty\!,\!0\!}$ & $\!R\!$ 
    		& $\| |u \!-\! u^{h}| \|_{\!\infty\!,\!1\!}$   & $\!R\!$  
    		& $\| |p \!-\! p^{h}| \|_{\!\infty\!,\!0\!}$   & $\!R\!$ & Time(s) 
    		\\
    		\hline
    		\hline
    		$1\!/\!16$     & 4.1578e-4   & -        & 4.9177e-2  & -     & 3.3032e-3  & -     & 18.9342
    		\\
    		$1\!/\!32$     & 3.0240e-5   & 3.7813   & 7.0035e-3  & 2.8118  & 8.5935e-4  & 1.9425  & 157.133
    		\\
    		$1\!/\!64$     & 2.2919e-6   & 3.7219   & 9.3310e-4  & 2.9080  & 2.1773e-4  & 1.9807  & 1214.16
    		\\
    		$1\!/\!128$    & 2.0408e-7   & 3.4893   & 1.2046e-4  & 2.9534  & 5.4718e-5  & 1.9925  & 8492.7
    		\\
    		$1\!/\!256$    & 2.1951e-8   & 3.2168   & 1.5306e-5  & 2.9764  & 1.3687e-5  & 1.9992  & 71309.7
    		\\
    		\hline
    	\end{tabular}
    	\label{table:Full-DLN-1-Linf}
    \end{table}

	\begin{table}[ptbh]
    	\centering
    	\caption{$L^{2}$-norm of error and rate for fully-implicit DLN scheme ($\theta=1$)}
    	\begin{tabular}{lccccccc}
    		\hline
    		\hline
    		$k\!=\!h$  & $\| |u \!-\! u^{h}| \|_{\!2\!,\!0\!}$ & $\!R\!$ 
    		& $\| |u \!-\! u^{h}| \|_{\!2\!,\!1\!}$ & $\!R\!$  
    		& $\| |p \!-\! p^{h}| \|_{\!2\!,\!0\!}$   & $\!R\!$ & Time(s)
    		\\
    		\hline
    		\hline
    		$1\!/\!16$     & 2.3141e-4   & -      & 2.6100e-2  & -     & 3.0611e-3  & -     & 18.9342
    		\\
    		$1\!/\!32$     & 1.6803e-5   & 3.7837   & 3.3893e-3  & 2.9450  & 8.1693e-4  & 1.9058  & 157.133
    		\\
    		$1\!/\!64$     & 1.4757e-6   & 3.5092   & 4.3010e-4  & 2.9782  & 2.0943e-4  & 1.9637  & 1214.16
    		\\
    		$1\!/\!128$    & 1.5887e-7   & 3.2155   & 5.4145e-5  & 2.9898  & 5.2906e-5  & 1.9850  & 8492.7
    		\\
    		$1\!/\!256$    & 1.8995e-8   & 3.0641   & 6.7917e-6  & 2.9949  & 1.3264e-5  & 1.9959  & 71309.7
    		\\
    		\hline
    	\end{tabular}
    	\label{table:Full-DLN-1-L2}
    \end{table}

	\subsection{Adaptive DLN Algorithms for Revised Taylor-Green Problem}
	\label{subsec:test2}
	We apply adaptive semi-implicit DLN algorithms in Section \ref{sec:Implement-DLN} to the revised Taylor-Green problem in the domain $\Omega = [0,1] \times [0,1]$. The exact solutions are
	\begin{gather}
	\label{eq:Taylor-Green-Revised}
	u_{1}(x,y,t) = -\cos(\omega \pi x) \sin(\omega \pi y) \exp ( 2 \omega^2 \pi^2 t/\tau ),  \\
	u_{2}(x,y,t) = \sin(\omega \pi x) \cos(\omega \pi y) \exp ( 2 \omega^2 \pi^2 t/\tau ),  \notag \\
	p(x,y,t) = - \frac{1}{4} \big( \cos(2 \omega \pi x) + \cos(2 \omega \pi y) \big) 
	\exp ( 4 \omega^2 \pi^2 t/\tau ).    \notag
	\end{gather}
	We set $\omega = 1$ and $\tau = 1/\nu = 2500$.
	The exact solutions in \eqref{eq:Taylor-Green-Revised} make the problem more difficult since the 
	Reynolds number is much larger and the energy has an increasing pattern.
	We use both Algorithm \ref{alg:Adaptivity-AB2-like} and Algorithm \ref{alg:Adaptivity-ND} to solve the problem over the time interval $[0,60]$. 
	For Algorithm \ref{alg:Adaptivity-AB2-like}, we use the relative estimator in \eqref{eq:ESTrel-AB2DLN} and set tolerance $\rm{Tol} = 1.\rm{e}-7$ and the safety factor $\kappa = 0.95$. 
	For Algorithm \ref{alg:Adaptivity-ND}, we set $\rm{Tol} = 1.\rm{e}-14$ for $\chi$. The value of $\rm{Tol}$ is chosen to balance accuracy and efficiency.
	For both algorithms, we set the minimum time step $k_{\rm{min}} = 0.0005$, the maximum time step 
	$k_{\rm{max}} = 0.05$, the initial time step $k_{0} = 0.0005$ and the mesh diameter $h = 1/180$. 
	The initial value, boundary value and body force are decided by the exact solutions. We measure the performance of two algorithms by evaluating energy, error of energy, numerical 
	dissipation $\mathcal{E}_{n+1}^{\tt ND}$ and viscosity $\mathcal{E}_{n+1}^{\tt VD}$. 
	Since $\mathcal{E}_{n+1}^{\tt ND}$ vanishes for the DLN method with $\theta = 1$, we test the two adaptive algorithms with $\theta = 2/3$ and $\theta = 2/\sqrt{5}$.  
	\cref{fig:AdaptDLN-Test2} shows the performance of two adaptive DLN algorithms 
	and \cref{table:AdaptDLN-Test2} tells us the number of steps. 

	Algorithm \ref{alg:Adaptivity-AB2-like} surpasses Algorithm \ref{alg:Adaptivity-ND} in terms of accuracy and efficiency: \cref{fig:Test2_Error} shows that the error magnitude of energy is much smaller for Algorithm \ref{alg:Adaptivity-AB2-like} while Algorithm \ref{alg:Adaptivity-ND} takes more number of time steps for both $\theta$ values.
	Two algorithms have similar patterns for numerical dissipation and viscosity in \cref{fig:Test2_ND,fig:Test2_VD}. 
	For both algorithms, $\widehat{T}_{n+1}$ and $\chi_{n+1}$ are kept below the required tolerance after the first few steps in \cref{fig:Test2_Tol} and time steps oscillate between $k_{\rm{max}}$ and $k_{\rm{min}}$ in \cref{fig:Test2_Step}.
	\begin{figure}[ptbh]
		\centering
		\subfigure[Energy $\frac{1}{2} \| u_{n+1}^{h} \|^{2}$]{ \label{fig:Test2_Energy}
			\hspace{-0.99cm}
			\begin{minipage}[t]{0.5\linewidth}
				\centering
				\includegraphics[width=2.7in,height=2.0in]{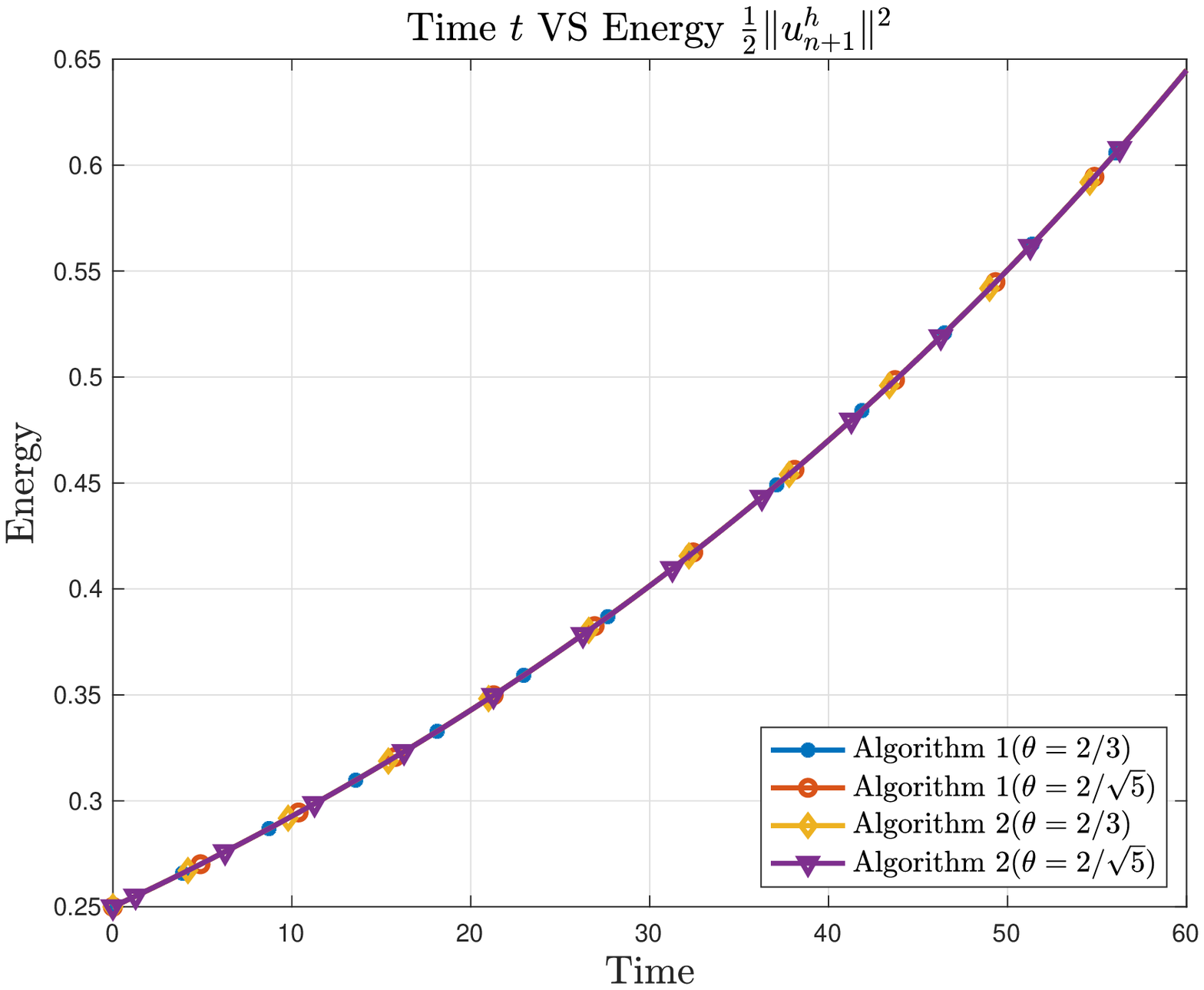}\\
				\vspace{0.02cm}
			\end{minipage}
			\quad}%
		\subfigure[$\log_{10} \big( \frac{1}{2} \big|\| u_{n+1}^{h} \|^{2} - \| u_{n+1} \|^{2} \big| \big) $]{ \label{fig:Test2_Error}
			\hspace{-0.99cm}
			\begin{minipage}[t]{0.5\linewidth}
				\centering
				\includegraphics[width=2.7in,height=2.0in]{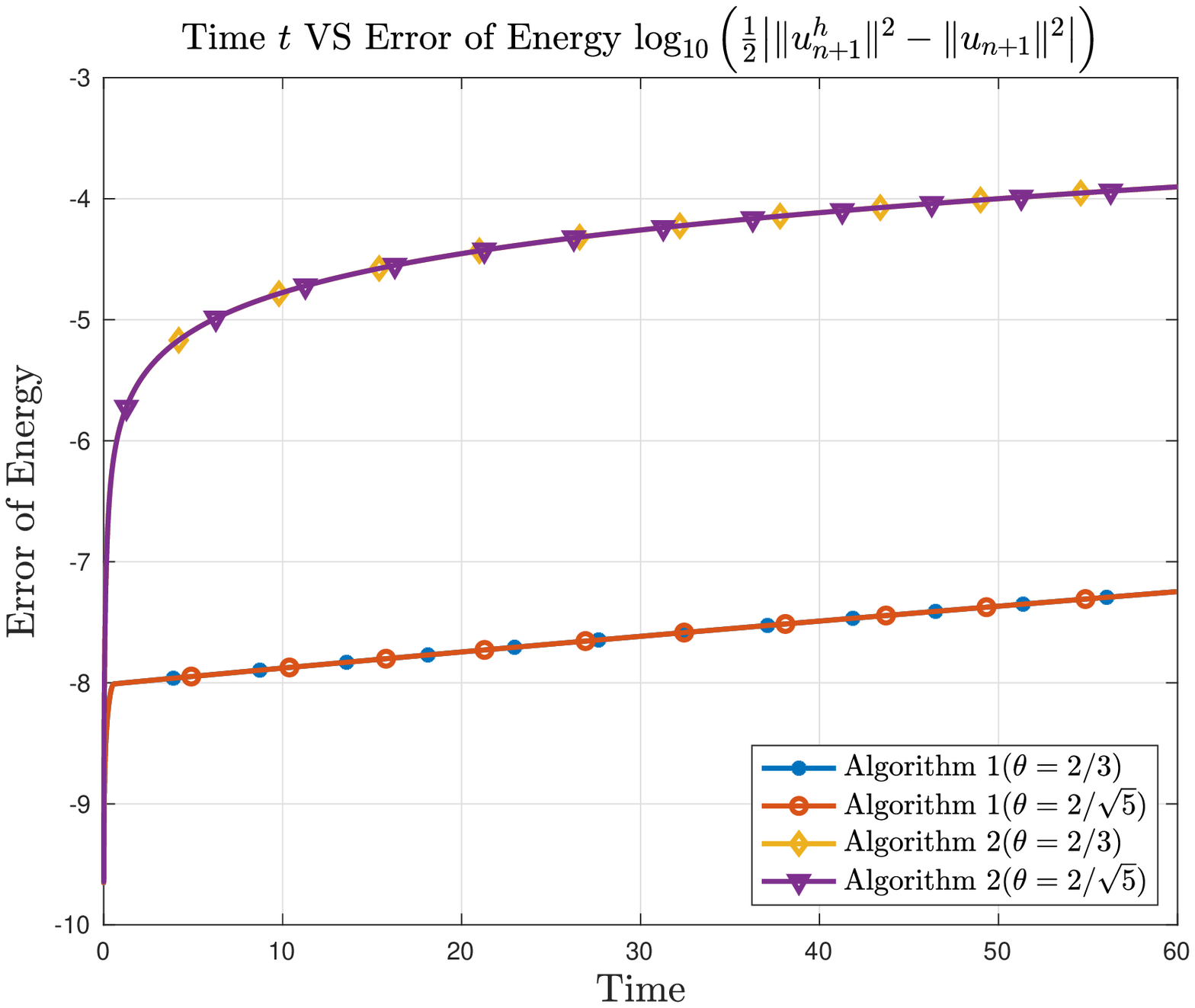}\\
				\vspace{0.02cm}
		\end{minipage}}  	
		\subfigure[Numerical Dissipation: $\log_{10}(\mathcal{E}_{n+1}^{\tt ND})$]{ \label{fig:Test2_ND}
			\hspace{-0.99cm}
			\begin{minipage}[t]{0.5\linewidth}
				\centering
				\includegraphics[width=2.7in,height=2.0in]{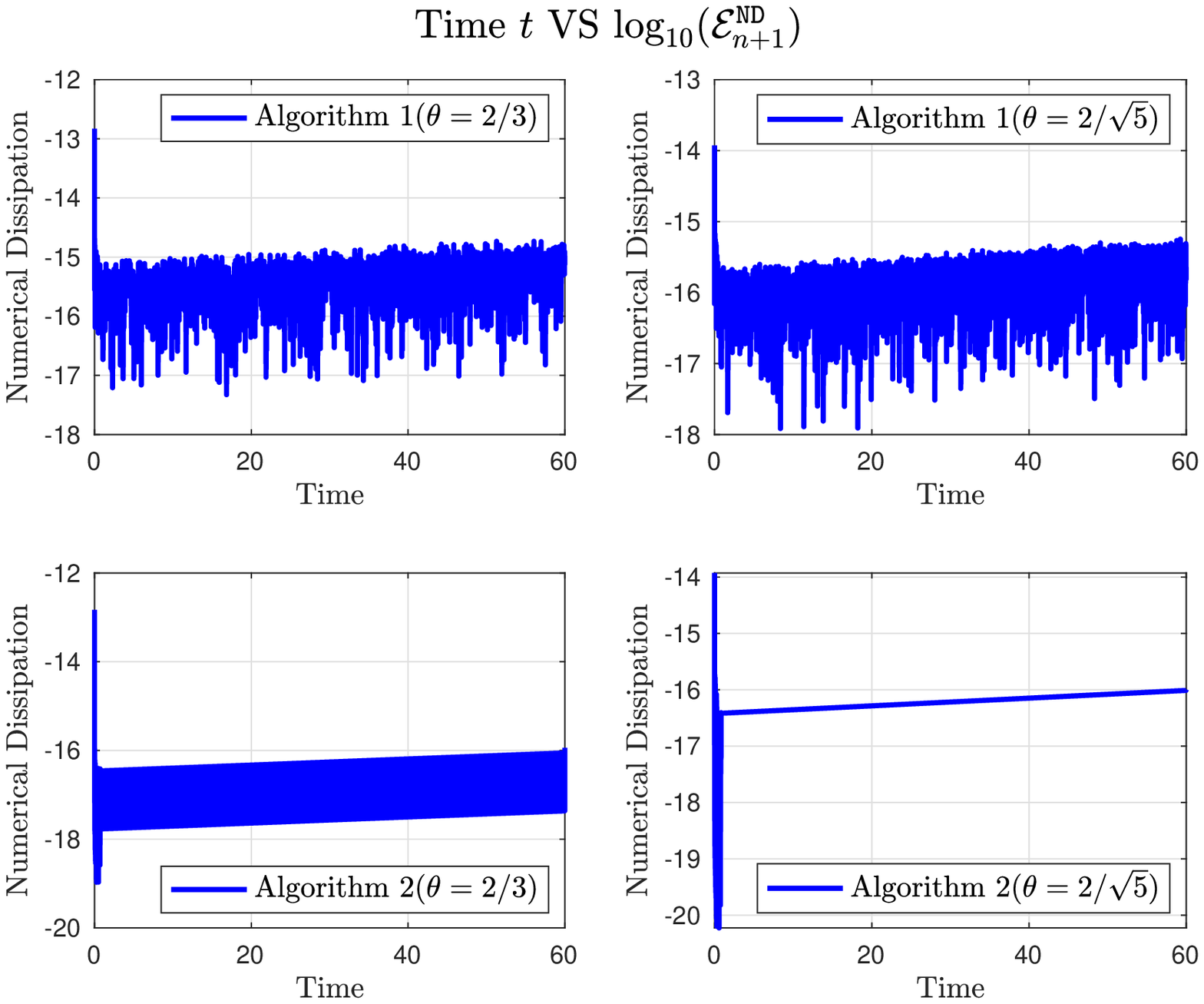}\\
				\vspace{0.02cm}
			\end{minipage}
			\quad}%
		\subfigure[Viscosity: $\mathcal{E}_{n+1}^{\tt VD}$]{ \label{fig:Test2_VD}
			\hspace{-0.99cm}
			\begin{minipage}[t]{0.5\linewidth}
				\centering
				\includegraphics[width=2.7in,height=2.0in]{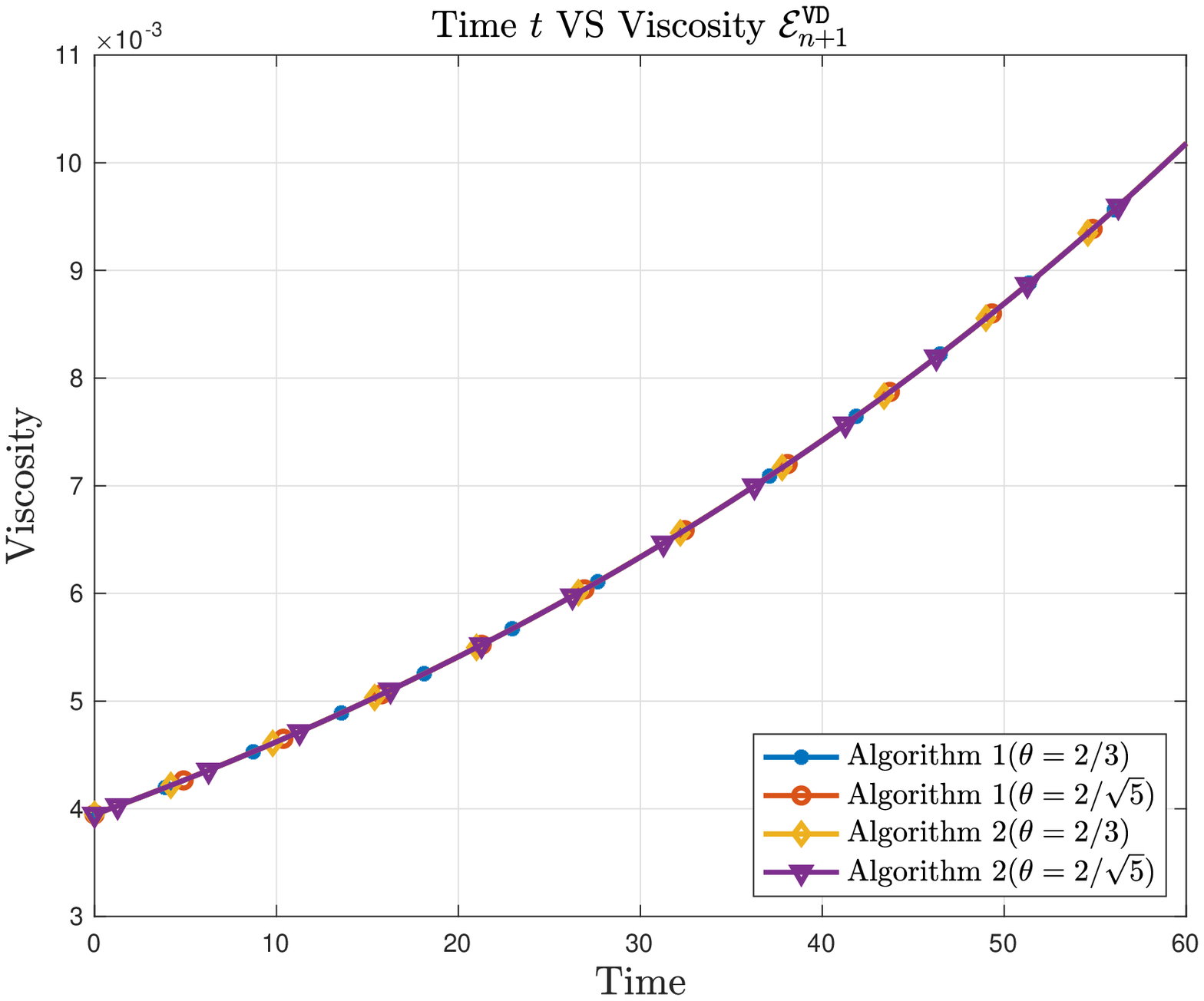}\\
				\vspace{0.02cm}
		\end{minipage}}
		\subfigure[$\log_{10}(\widehat{T}_{n+1})$ and $\log_{10}(\chi_{n+1})$]{ \label{fig:Test2_Tol}
			\hspace{-0.99cm}
			\begin{minipage}[t]{0.5\linewidth}
				\centering
				\includegraphics[width=2.7in,height=2.0in]{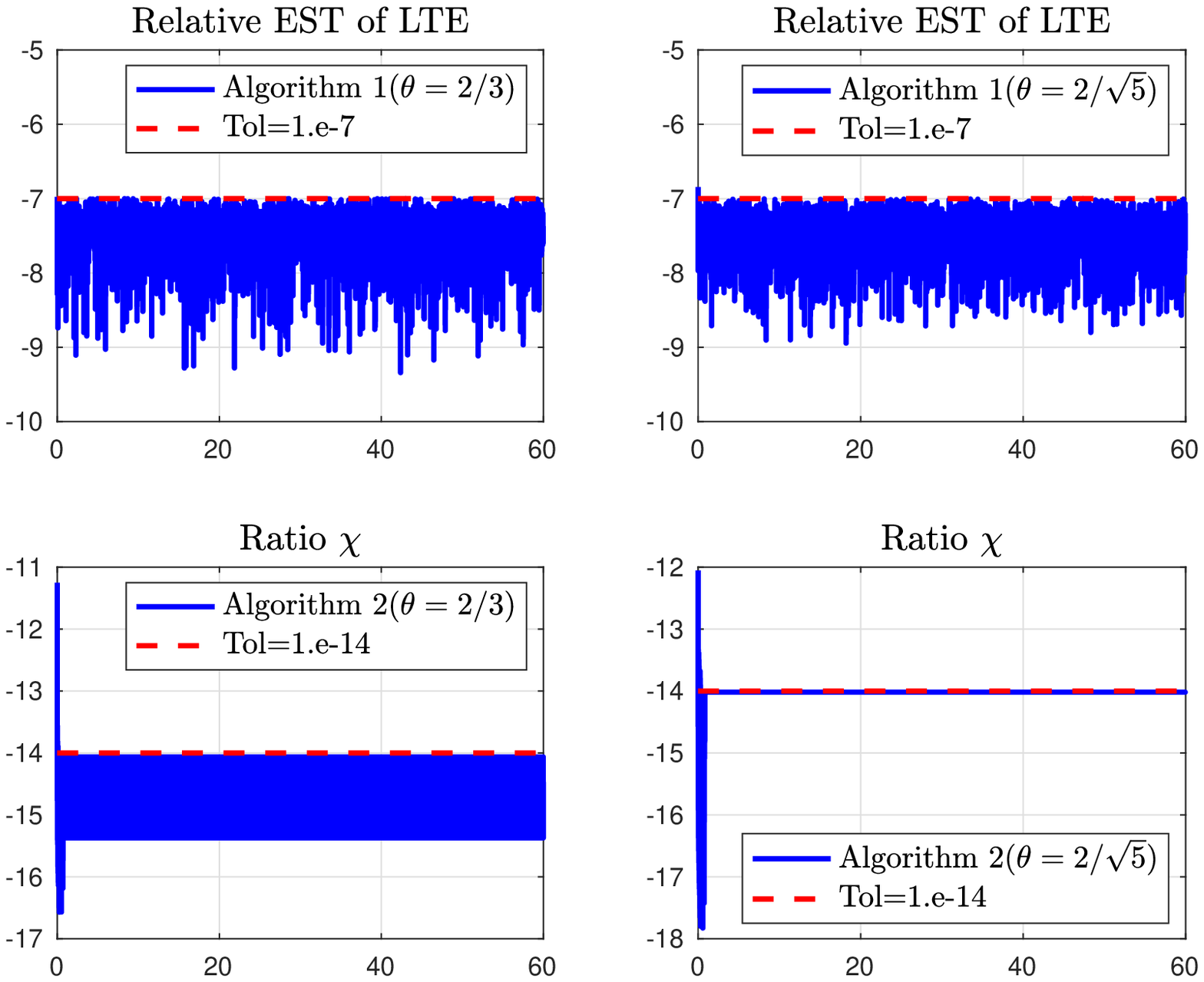}\\
				\vspace{0.02cm}
			\end{minipage}
			\quad}%
		\subfigure[Time step: $\log_{10}(k_{n})$]{ \label{fig:Test2_Step}
			\hspace{-0.99cm}
			\begin{minipage}[t]{0.5\linewidth}
				\centering
				\includegraphics[width=2.7in,height=2.0in]{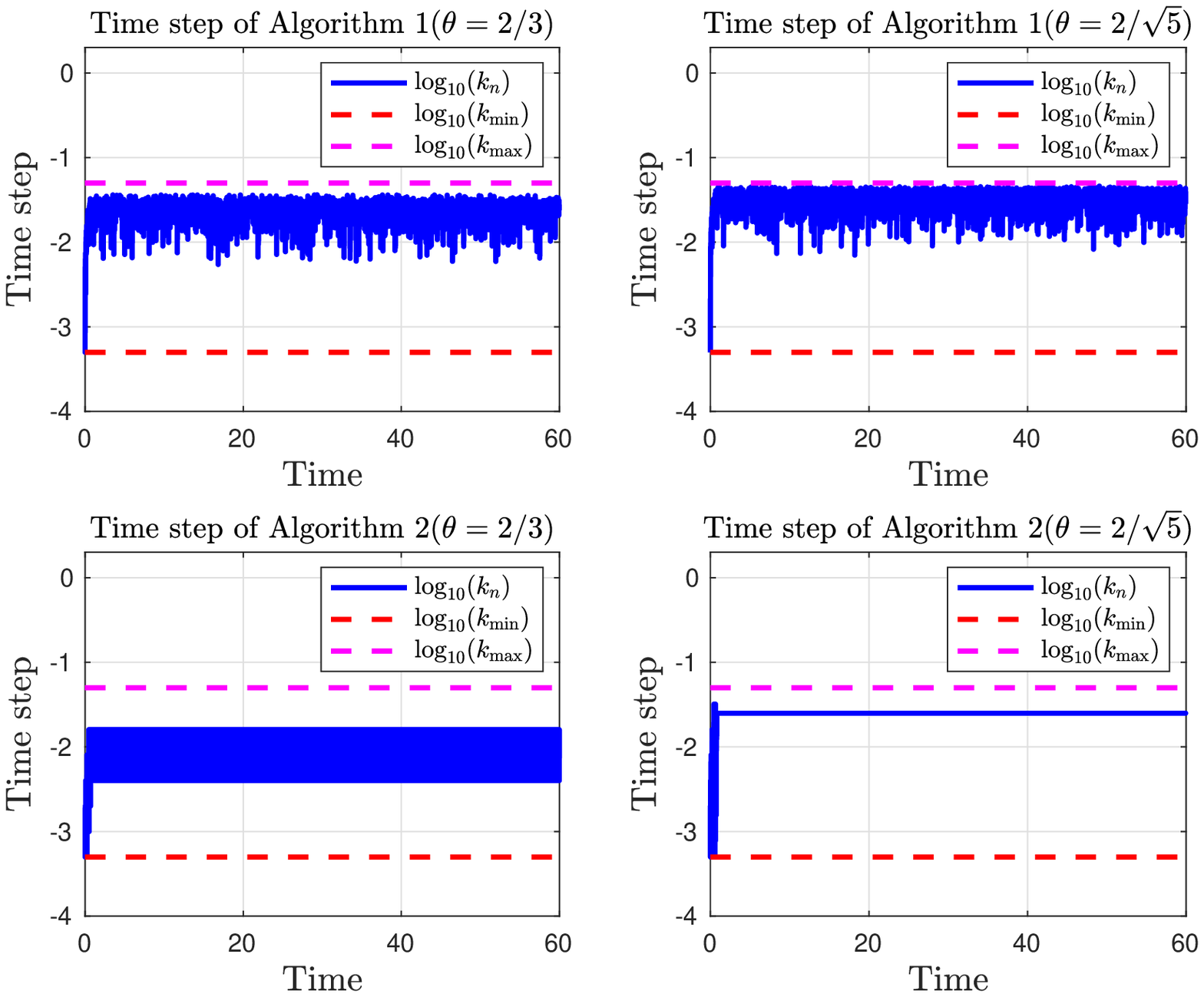}\\
				\vspace{0.02cm}
		\end{minipage}}%
		\par
		\centering
		\vspace{-0.2cm}
		\caption{Algorithm \ref{alg:Adaptivity-AB2-like} has a much smaller error magnitude of energy. 
		The two algorithms have similar patterns for numerical dissipation and viscosity.
		For both algorithms,
		$\widehat{T}_{n+1}$ and $\chi_{n+1}$ are kept below the required tolerance after the first few steps and
		time steps oscillate between $k_{\rm{max}}$ and $k_{\rm{min}}$.}
		\label{fig:AdaptDLN-Test2}
	\end{figure}

	\begin{table}[ptbh]
    	\centering
    	\caption{Algorithm \ref{alg:Adaptivity-AB2-like} takes fewer time steps to simulate the revised Taylor-Green problem than Algorithm \ref{alg:Adaptivity-ND}.}
    	\begin{tabular}{ccc}
    		\hline
    		\hline
    		& Algorithm \ref{alg:Adaptivity-AB2-like} 
    		& Algorithm \ref{alg:Adaptivity-ND}  
    		\\
    		\hline
    		\hline
			$\theta =2/3$         & 2566  & 6580  
			\\
    		\hline
    		$\theta = 2/\sqrt{5}$  & 1967   & 2550  
    		\\
    		\hline
			\hline
    	\end{tabular}
    	\label{table:AdaptDLN-Test2}
    \end{table}

	\subsection{2D Offset Circles Problem}
	We use the 2D offset circles problem proposed by Jiang and Layton \cite{JL14_IJUQ} to verify the stability of the DLN scheme under any arbitrary sequence of time steps and the efficiency of the adaptive algorithms in Section \ref{sec:Implement-DLN}.
	The domain $\Omega \subset \mathbb{R}^{2}$ is  
	\begin{gather*}
		\Omega = \{ (x,y): x^2 + y^2 \leq 1 \text{ and  } (x - 0.5)^{2} + y^2 \geq 0.01 \}.
	\end{gather*}
	The flow in the domain $\Omega$ is driven by the rotational body force 
	\begin{gather*}
	    f(x,y,t) = 
		\begin{bmatrix}
			f_{1}(x,y) \\
			f_{2}(x,y)
		\end{bmatrix}
		=
		\begin{bmatrix}
			-4y(1 - x^2 - y^2) \\
			 4x(1 - x^2 - y^2)
		\end{bmatrix}.
	\end{gather*}
	with the no-slip boundary condition on both circles. 
	We set the Reynolds number $\rm{Re} = 1/\nu = 200$ and simulate the problem over time interval $[0,60]$.
	We use the relative estimator of LTE in \eqref{eq:ESTrel-AB2DLN} for Algorithm \ref{alg:Adaptivity-AB2-like} and set the tolerance $\rm{Tol} = 0.001$ and safety factor $\kappa = 0.95$. 
	For Algorithm \ref{alg:Adaptivity-ND}, we set $\rm{Tol} = 0.01$ for ratio $\chi_{n+1}$. 
	For both adaptive algorithms, the initial time step $k_{0} = 0.005$, the maximum time 
	step $k_{\rm{max}} = 0.05$ and the minimum time step $k_{\rm{min}} = 0.0005$. 
	The domain triangulation is generated by 80 nodes on the boundary of the inner circle and 320 nodes on the boundary of the outer circle.
	Since the exact solutions are unknown, we use the constant step DLN algorithm ($\varepsilon_{n} = 0$) with a small time step ($k=k_{\rm{min}}$) and refined mesh (100 nodes on the boundary of inner circle and 400 nodes on the boundary of outer circle) for reference.
	\cref{fig:Domain_2Doffset_Adaptive,fig:Domain_2Doffset_Constant} show two domain triangulations.
	The number of time steps is presented in \cref{table:AdaptDLN-Test3}.

	\cref{fig:Test3_En} shows that the energy of all algorithms is increasing at the start and then come to the steady level 23 at time $t = 8$. 
	We deduce that Algorithm \ref{alg:Adaptivity-ND}($\theta = 2/\sqrt{5}$) has worse performance for this problem because the energy level of this algorithm is low compared to that of other adaptive algorithms. 
	In addition, the number of time steps is least for Algorithm \ref{alg:Adaptivity-ND} ($\theta = 2/\sqrt{5}$) while the ratio $\chi$ goes above the required tolerance value ($1.\rm{e}-2$) for many times in the simulation. 
	Then we compare \cref{fig:Test3_En_Alg1_DLN23,fig:Test3_En_Alg1_DLN25,fig:Test3_En_Alg2_DLN23} and observe that the energy of Algorithm \ref{alg:Adaptivity-AB2-like}($\theta = 2/\sqrt{5}$) is closer to the energy of reference algorithms with less number of time steps. 
	From \cref{fig:Test3_ND}, we can see that the numerical dissipation of Algorithm \ref{alg:Adaptivity-AB2-like} is at a level as low as that of reference algorithms while that of Algorithm \ref{alg:Adaptivity-ND} is much larger. 
	All the algorithms have similar viscosity patterns in \cref{fig:Test3_VD}. 
	From \cref{fig:Test3_Tol,fig:Test3_Step}, $\widehat{T}$ is always below the required 
	tolerance ($1.\rm{e}-3$) thus the time steps of Algorithm \ref{alg:Adaptivity-AB2-like} never reach $k_{\rm{min}}$. However the ratio $\chi$ goes above the required tolerance $1.\rm{e}-2$ frequently and $k_{n} = k_{\rm{min}}$ occurs very often. The primitive time step controller (doubling and halving time steps) in Algorithm \ref{alg:Adaptivity-AB2-like} reduces the number of time steps and may cause inaccuracy. 

	\begin{figure}[ptbh]
		\subfigure[80 nodes on the inner circle and 320 nodes \\ on the outer circle for adaptive algorithms]{ \label{fig:Domain_2Doffset_Adaptive}
			\hspace{-0.2cm}
			\begin{minipage}[t]{0.45\linewidth}
				\centering
				\includegraphics[width=2.65in,height=1.8in]{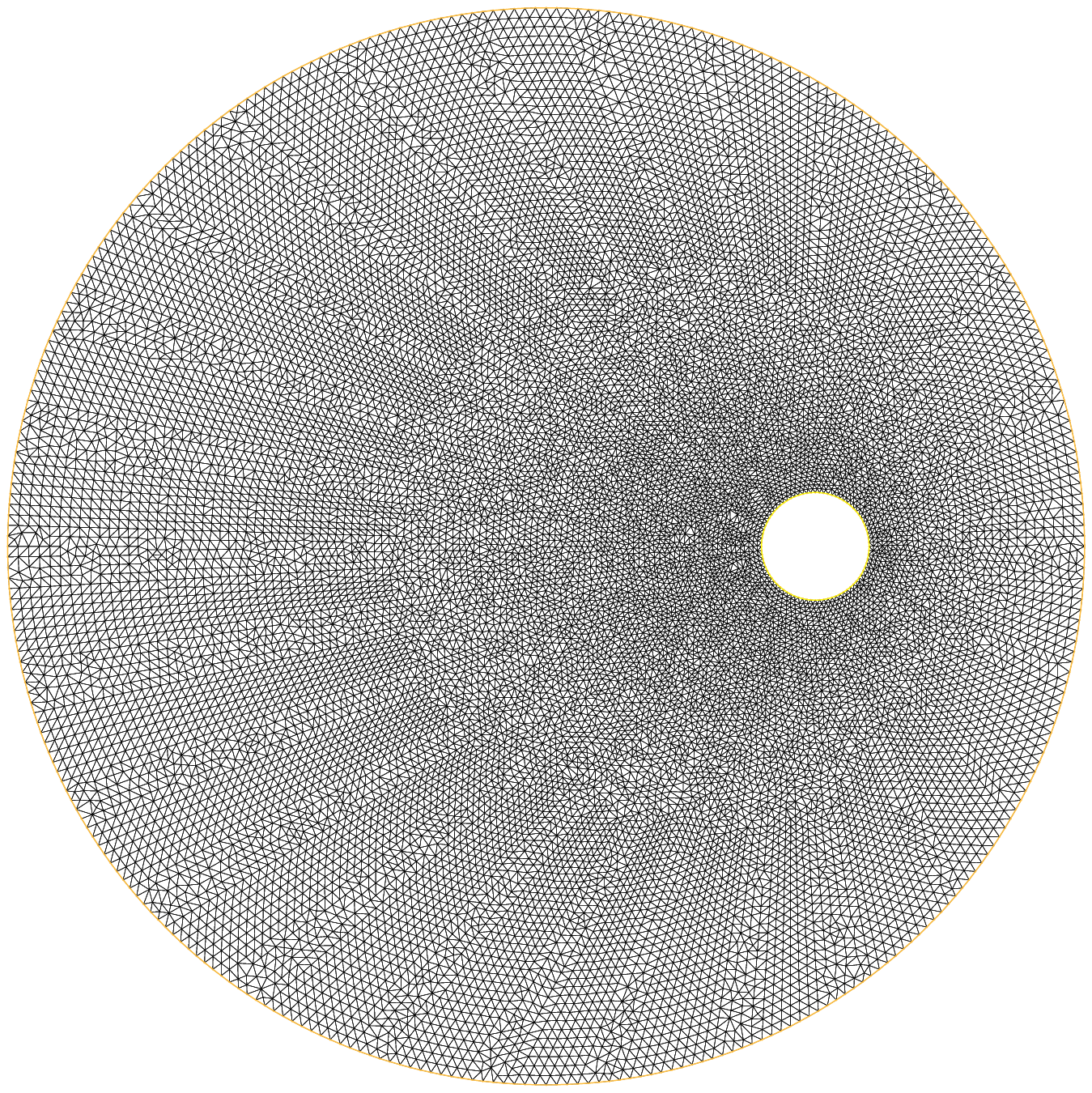}\\
				\vspace{0.02cm}
			\end{minipage}
			\quad}%
		\subfigure[100 nodes on the inner circle and 400 nodes \\ on the outer circle for constant algorithms]{ \label{fig:Domain_2Doffset_Constant}
			\hspace{-0.2cm}
			\begin{minipage}[t]{0.45\linewidth}
				\centering
				\includegraphics[width=2.65in,height=1.8in]{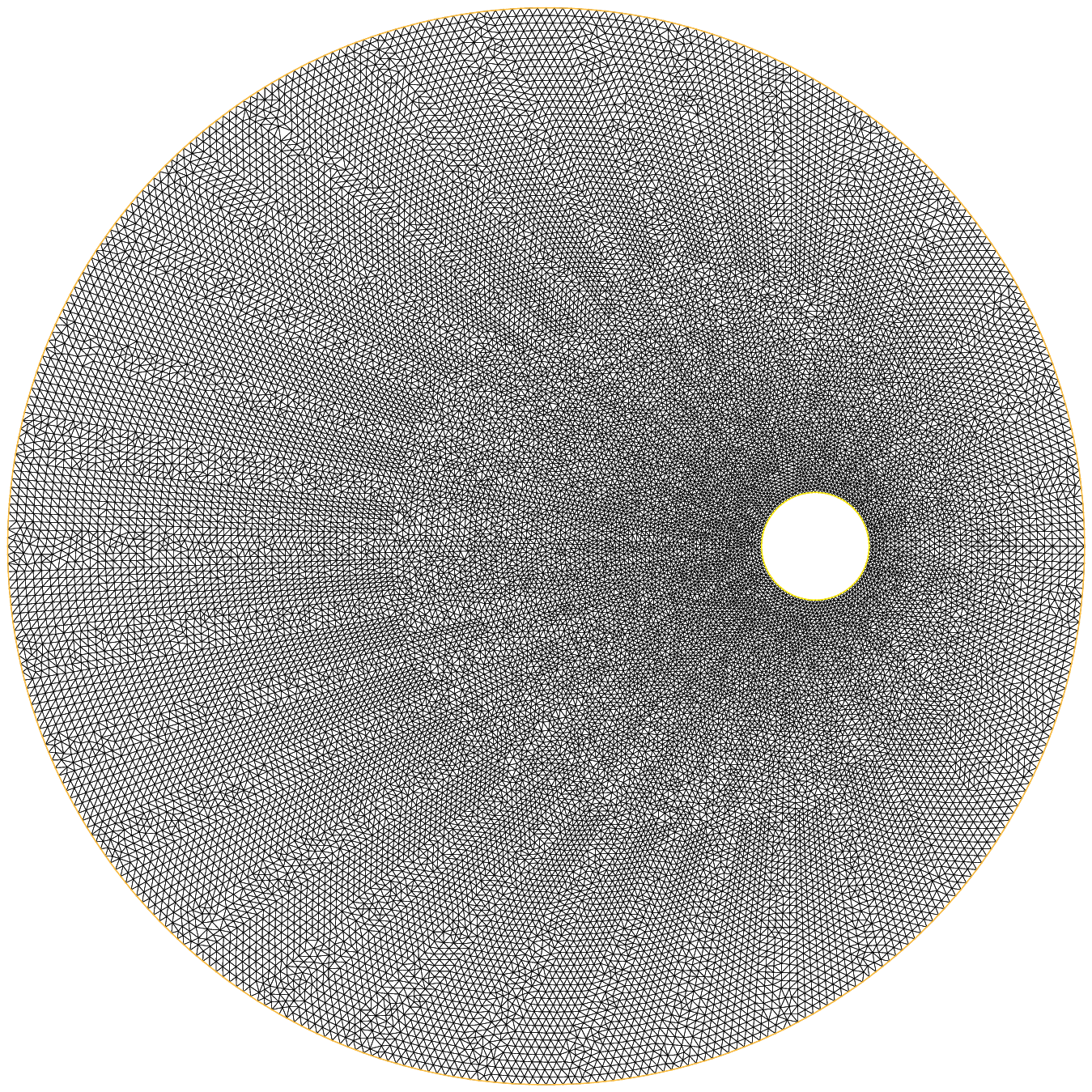}\\
				\vspace{0.02cm}
		\end{minipage}}  	
		\par
		\centering
		\vspace{-0.2cm}
		\caption{Domain triangulations for adaptive DLN algorithms and constant DLN algorithms 
		}
		\label{fig:Domain_2Doffset}
	\end{figure}
	\begin{table}[ptbh]
    	\centering
    	\caption{Algorithm \ref{alg:Adaptivity-AB2-like} takes fewer time steps to simulate the 2D offset Problem than algorithm \ref{alg:Adaptivity-ND}.}
    	\begin{tabular}{cccc}
    		\hline
    		\hline
    		& Algorithm \ref{alg:Adaptivity-AB2-like} 
    		& Algorithm \ref{alg:Adaptivity-ND}  
			& Constant step DLN 
    		\\
    		\hline
    		\hline
			$\theta =2/3$          & 64577   & 5933  & 120000
			\\
    		\hline
    		$\theta = 2/\sqrt{5}$  & 49175   & 3767  & 120000
    		\\
    		\hline
			\hline
    	\end{tabular}
    	\label{table:AdaptDLN-Test3}
    \end{table}

	\begin{figure}[ptbh]
		\centering
		\subfigure[Energy of Algorithm \ref{alg:Adaptivity-AB2-like}, Algorithm \ref{alg:Adaptivity-ND} and reference algorithms]{ \label{fig:Test3_En}
			\begin{minipage}[t]{0.999\linewidth}
				\centering
				\includegraphics[width=5.2in,height=1.5in]{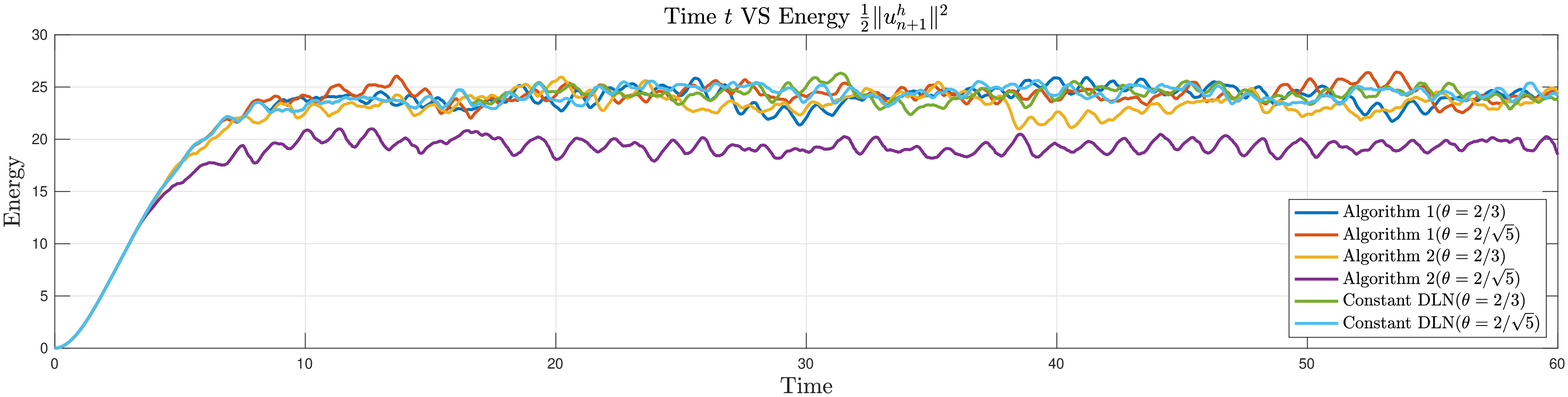}\\
				\vspace{0.02cm}
			\end{minipage}
			}%

		\subfigure[Energy of Algorithm \ref{alg:Adaptivity-AB2-like}($\theta=\frac{2}{3}$) and reference algorithms]{ \label{fig:Test3_En_Alg1_DLN23}
			\begin{minipage}[t]{0.999\linewidth}
				\centering
				\includegraphics[width=5.2in,height=1.5in]{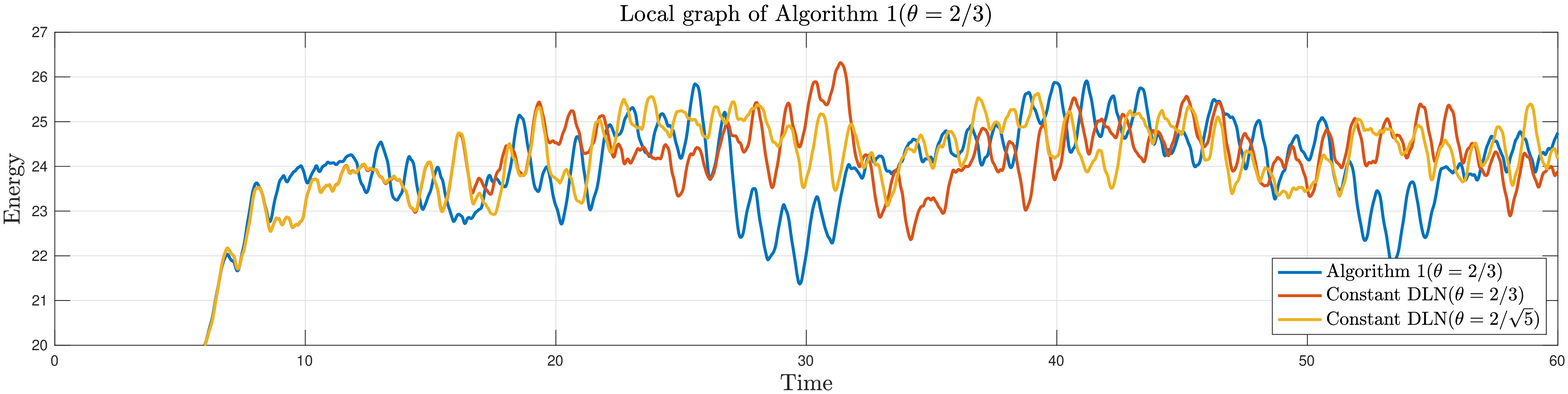}\\
				\vspace{0.02cm}
		\end{minipage}}  	

		\subfigure[Energy of Algorithm \ref{alg:Adaptivity-AB2-like}($\theta=\frac{2}{\sqrt{5}}$) and reference algorithms]{ \label{fig:Test3_En_Alg1_DLN25}
			\begin{minipage}[t]{0.999\linewidth}
				\centering
				\includegraphics[width=5.2in,height=1.5in]{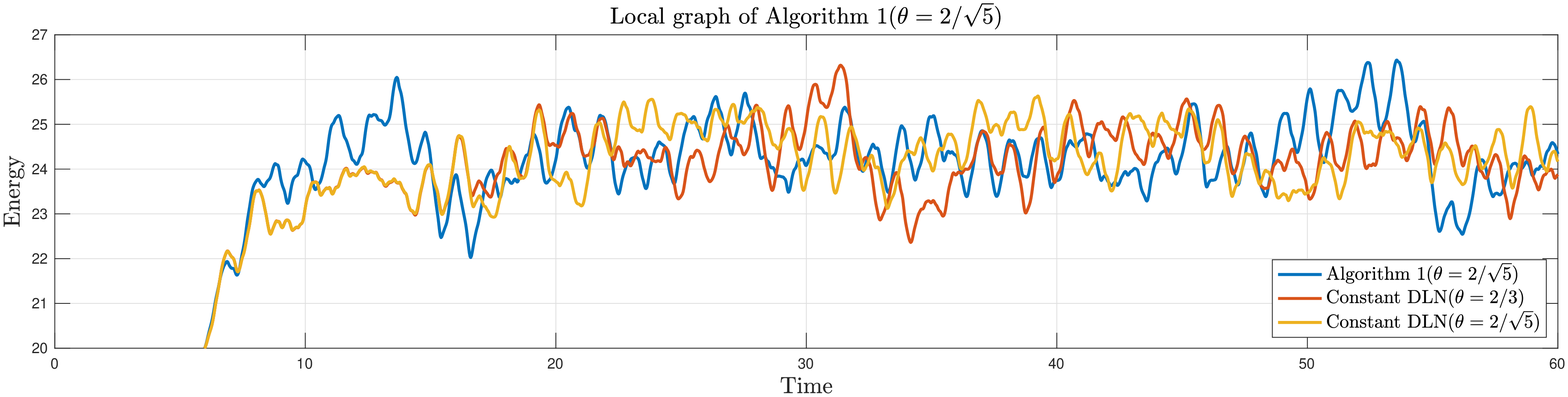}\\
				\vspace{0.02cm}
			\end{minipage}
			}%

		\subfigure[Algorithm \ref{alg:Adaptivity-ND} ($\theta=\frac{2}{3}$) and reference algorithms]{ \label{fig:Test3_En_Alg2_DLN23}
			\begin{minipage}[t]{0.999\linewidth}
				\centering
				\includegraphics[width=5.2in,height=1.5in]{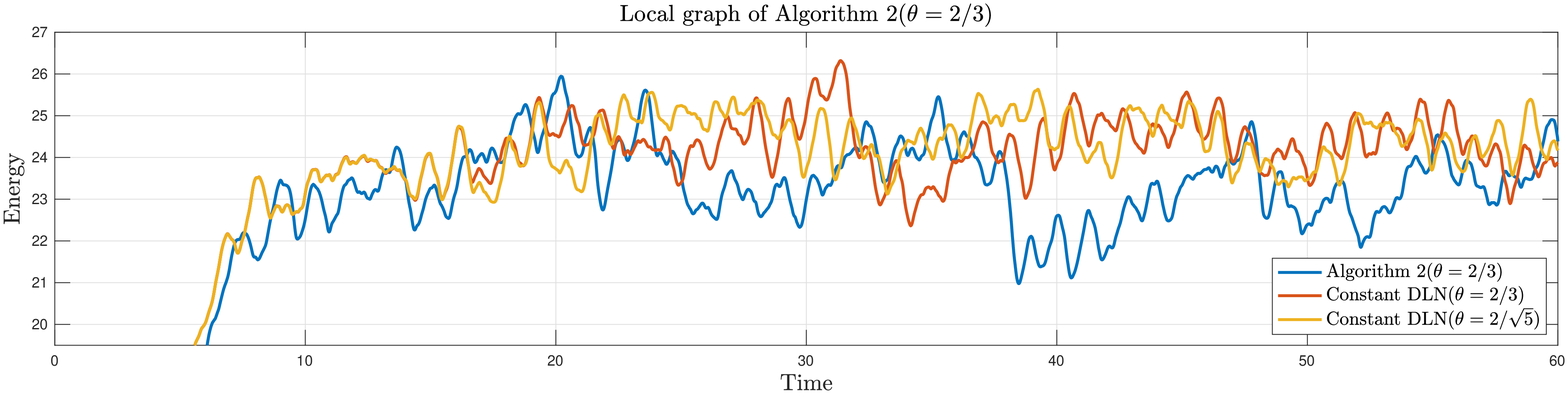}\\
				\vspace{0.02cm}
		\end{minipage} }
		\par
		\centering
		\vspace{-0.2cm}
		\caption{The energy of all algorithms are increasing at the start and then comes to the steady level 23 at time $t = 8$. The energy level of Algorithm \ref{alg:Adaptivity-AB2-like}($\theta = 2/\sqrt{5}$) is low compared to that of other adaptive algorithms and the energy of Algorithm \ref{alg:Adaptivity-AB2-like}($\theta = 2/\sqrt{5}$) is closer to that of reference algorithms.
		}
		\label{fig:En_2Doffset_AdaptDLN}
	\end{figure}

	\begin{figure}[ptbh]
		\centering
		\subfigure[Numerical dissipation: $\log_{10}(\mathcal{E}_{n+1}^{\tt ND})$]{ \label{fig:Test3_ND}
			\hspace{-0.99cm}
			\begin{minipage}[t]{0.999\linewidth}
				\centering
				\includegraphics[width=5.3in,height=2in]{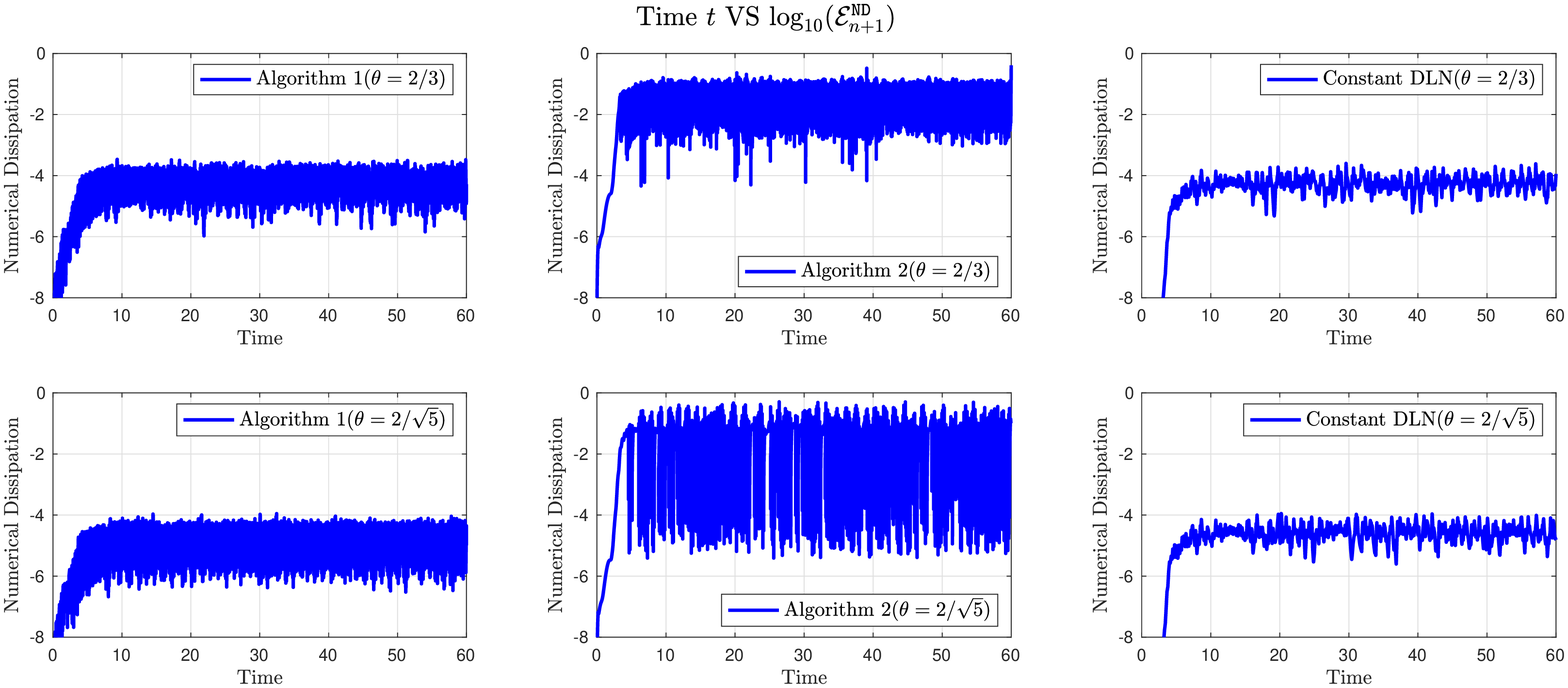}\\
				\vspace{0.02cm}
			\end{minipage}
			\quad}%

		\subfigure[Viscosity: $\log_{10}(\mathcal{E}_{n+1}^{\tt VD})$]{ \label{fig:Test3_VD}
			\hspace{-0.99cm}
			\begin{minipage}[t]{0.999\linewidth}
				\centering
				\includegraphics[width=5.3in,height=2in]{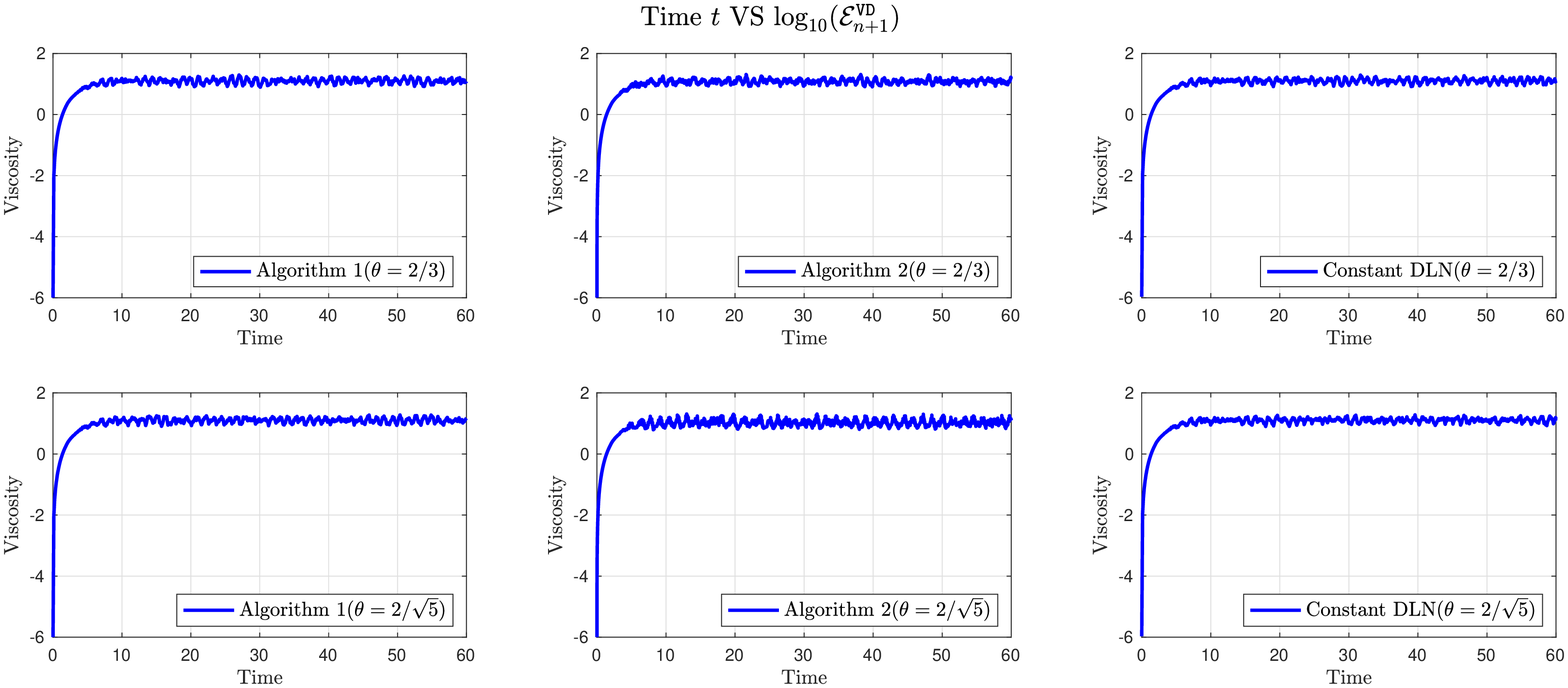}\\
				\vspace{0.02cm}
		\end{minipage}
		\quad}  	

		\subfigure[$\log_{10}(\widehat{T}_{n+1})$ and $\log_{10}(\chi_{n+1})$]{ \label{fig:Test3_Tol}
			\hspace{-0.9cm}
			\begin{minipage}[t]{0.48\linewidth}
				\centering
				\includegraphics[width=2.65in,height=1.8in]{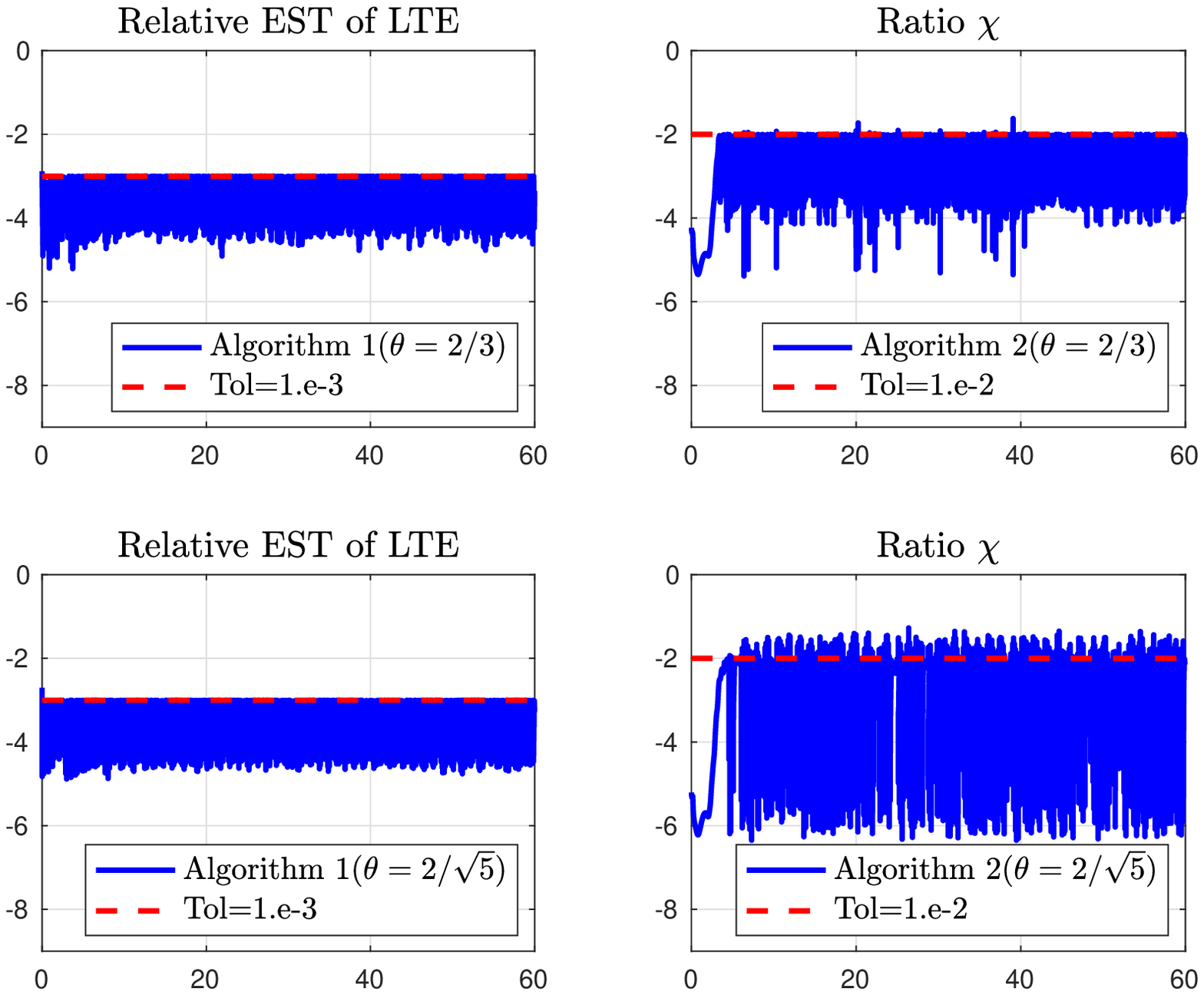}\\
				\vspace{0.02cm}
			\end{minipage}
			\quad}%
		\subfigure[Time step: $\log_{10}(k_{n})$]{ \label{fig:Test3_Step}
			\hspace{-0.9cm}
			\begin{minipage}[t]{0.49\linewidth}
				\centering
				\includegraphics[width=2.65in,height=1.8in]{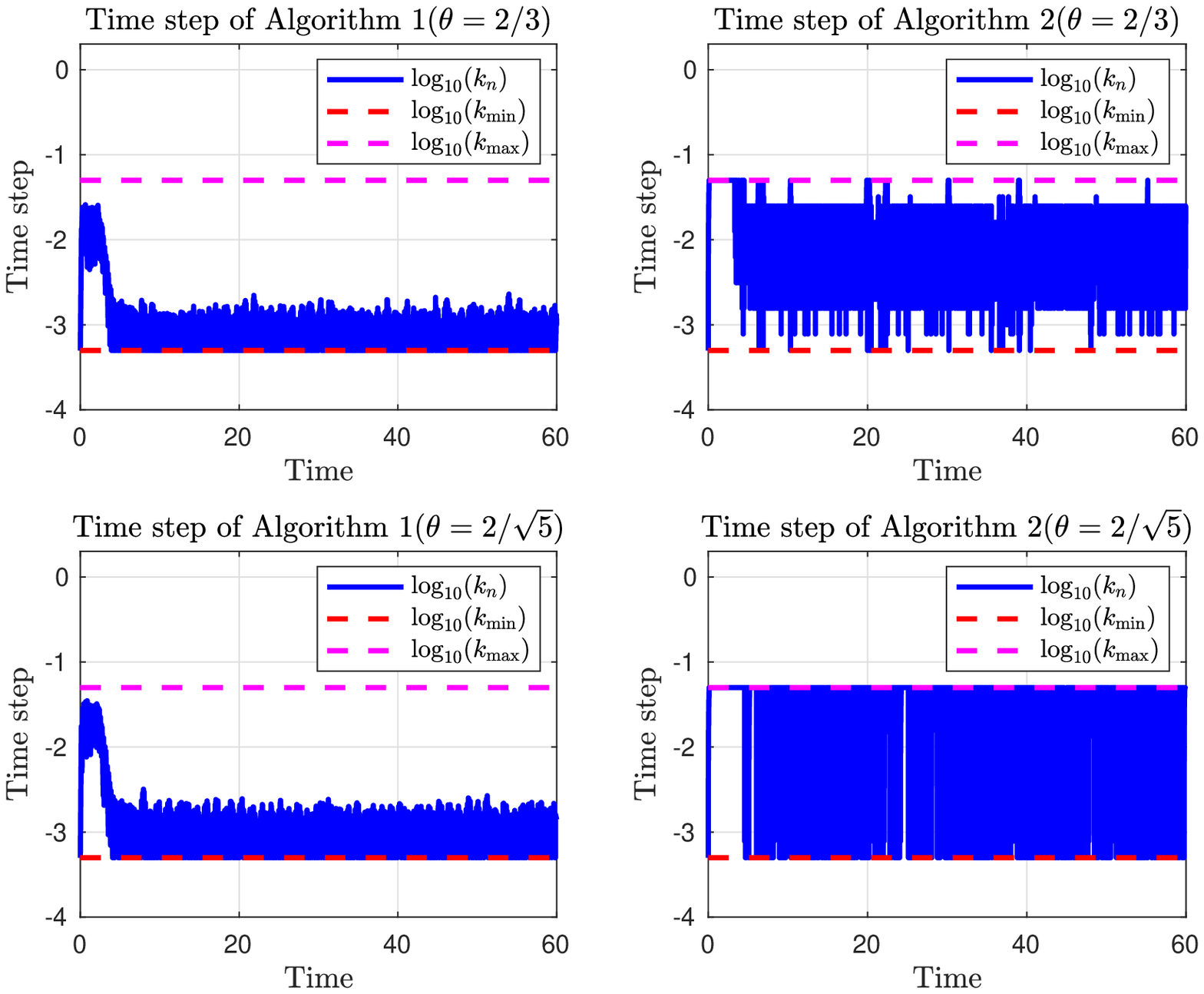}\\
				\vspace{0.02cm}
		\end{minipage}}
		\par
		\centering
		\vspace{-0.2cm}
		\caption{The numerical dissipation of Algorithm \ref{alg:Adaptivity-AB2-like} is at a level as low as that of reference algorithms while that of Algorithm \ref{alg:Adaptivity-ND} is much larger.
		All the algorithms have similar viscosity patterns.
		$\widehat{T}$ is always below the required tolerance ($1.\rm{e}-3$) thus the time steps of 
		Algorithm \ref{alg:Adaptivity-ND} never reach $k_{\rm{min}}$.
		The ratio $\chi$ goes above the required tolerance $1.\rm{e}-2$ frequently 
		and $k_{n} = k_{\rm{min}}$ occurs very often.
		}
		\label{fig:Test3_PlotSummary}
	\end{figure}

	\section{Conclusions}
	\label{sec:Conclusion}
	We propose the semi-implicit DLN scheme for the NSE and avoid non-linear solvers at each time step.
	$G$-stability of the DLN method results in the long-term, unconditional stability of the numerical solutions.
	In the error analysis, we prove that both the velocity and pressure of the variable time-stepping, semi-implicit scheme converge in second order under very moderate time conditions.
	Two adaptive algorithms based on local truncation error and numerical dissipation criteria are presented to improve time efficiency in practice. 
	The advantage of the semi-implicit DLN scheme is observed in numerical tests in 
	Subsection \ref{subsec:test-const}: the semi-implicit scheme obtains the same accuracy as the fully-implicit scheme and reduces the simulation time by half.
	Subsection \ref{subsec:test2} shows that two adaptive DLN algorithms obtain enough accuracy in energy and negligible numerical dissipation even the problem with a large Reynolds number has an increasing energy pattern.
	We verify in the 2D offset problem that the semi-implicit DLN scheme is unconditional, long-time stable in energy with any arbitrary sequence of time steps, and the adaptive DLN algorithm is much more efficient than constant time-stepping DLN scheme (taking less number of time steps and attaining similar magnitude in energy, numerical dissipation and viscosity).

	\section{Acknowledgement} 
	The author thanks Professor Catalin Trenchea (Department of Mathematics, University of Pittsburgh) for very helpful suggestions and discussions.

	\bibliographystyle{abbrv}
	\bibliography{ReferenceNew}

\end{document}